\documentclass[12pt]{amsart}
\usepackage{amssymb,latexsym}
\usepackage{enumerate}
\usepackage[T1]{fontenc}
\usepackage[text={13.5cm,23cm},centering]{geometry}

\newtheorem{thm}{Theorem}[section]
\newtheorem{cor}[thm]{Corollary}
\newtheorem{lem}[thm]{Lemma}

\theoremstyle{definition}
\newtheorem{defin}[thm]{Definition}
\newtheorem{rem}[thm]{Remark}
\newtheorem{exa}[thm]{Example}

\numberwithin{equation}{section}

\usepackage{ifthen}
\usepackage{graphicx}
\usepackage{xcolor}

\newcommand{\abs}[1]{\left| #1 \right|}
\newcommand{\expr}[1]{\left( #1 \right)}
\newcommand{\norm}[1]{\left\| #1 \right\|}
\newcommand{\set}[1]{\left\{ #1 \right\}}
\newcommand{\scalar}[1]{\left< #1 \right>}
\newcommand{\tscalar}[1]{\langle #1 \rangle}
\newcommand{\ind}{\mathbf{1}}
\newcommand{\sub}{\subseteq}
\newcommand{\C}{\mathbf{C}}
\newcommand{\R}{\mathbf{R}}
\newcommand{\Z}{\mathbf{Z}}
\newcommand{\pr}{\mathbf{P}}
\newcommand{\ex}{\mathbf{E}}
\newcommand{\hl}{{(0, \infty)}}
\newcommand{\A}{\mathcal{A}}
\newcommand{\schwartz}{\mathcal{S}}
\newcommand{\fourier}{\mathcal{F}}
\newcommand{\laplace}{\mathcal{L}}
\newcommand{\dynkin}{\mathfrak{A}}
\newcommand{\domain}{\mathcal{D}}
\newcommand{\eps}{\varepsilon}
\newcommand{\ph}{\varphi}
\newcommand{\thet}{\vartheta}
\newcommand{\conv}{*}

\newcommand{\formula}[2][nolabel]
{\ifthenelse{\equal{#1}{nolabel}}
 {\begin{align*} #2 \end{align*}}
 {\ifthenelse{\equal{#1}{}}
  {\begin{align} #2 \end{align}}
  {\begin{align} \label{#1} #2 \end{align}}
 }
}

\DeclareMathOperator{\pv}{pv}
\DeclareMathOperator{\Arg}{Arg}
\DeclareMathOperator{\imag}{Im}
\DeclareMathOperator{\real}{Re}
\DeclareMathOperator{\dist}{dist}
\DeclareMathOperator{\li}{Li}
\DeclareMathOperator{\supp}{supp}
\DeclareMathOperator{\res}{Res}
\DeclareMathOperator{\var}{Var}

\theoremstyle{plain}
\newtheorem{prop}[thm]{Proposition}
\theoremstyle{definition}
\newtheorem{assumption}[thm]{Assumption}
\theoremstyle{remark}

\begin{document}

\title[Subordinate Brownian motions in half-line]{Spectral analysis of subordinate Brownian motions in half-line}

\author[M. Kwa{\'s}nicki]{Mateusz Kwa{\'s}nicki}
\thanks{Work supported by the Polish Ministry of Science and Higher Education grant no. N~N201 373136}
\thanks{The author received financial support of the Foundation for Polish Science}

\address{Institute of Mathematics \\ Polish Academy of Sciences \\ ul. {\'S}niadeckich 8 \\ 00-976 Warszawa, Poland}
\email{m.kwasnicki@impan.pl}
\address{Institute of Mathematics and Computer Science \\ Wroc{\l}aw University of Technology \\ ul. Wybrze{\.z}e Wyspia{\'n}\-skiego 27 \\ 50-370 Wroc{\l}aw, Poland}
\email{mateusz.kwasnicki@pwr.wroc.pl}

\date{}

\begin{abstract}
We study one-dimensional L{\'e}vy processes with L{\'e}vy-Khintchine exponent $\psi(\xi^2)$, where $\psi$ is a complete Bernstein function. These processes are subordinate Brownian motions corresponding to subordinators, whose L{\'e}vy measure has completely monotone density; or, equivalently, symmetric L{\'e}vy processes whose L{\'e}vy measure has completely monotone density on $(0, \infty)$. Examples include symmetric stable processes and relativistic processes. The main result is a formula for the generalized eigenfunctions of transition operators of the process killed after exiting the half-line. A generalized eigenfunction expansion of the transition operators is derived. As an application, a formula for the distribution of the first passage time (or the supremum functional) is obtained.
\end{abstract}



\maketitle

%
%

\section{Introduction and statement of main results}
\label{sec:intro}

%
%

In a recent paper~\cite{bib:kkms10}, spectral problem for the one-dimensional Cauchy process (that is, the symmetric $1$-stable process) killed upon exiting a half-line or an interval was studied. For the half-line, an explicit formula for generalized eigenfunctions of transition operators was obtained using methods developed in the theory of linear water waves. The argument of~\cite{bib:kkms10} relies on some properties specific to the Cauchy process, and it does not easily generalize to other L{\'e}vy processes. The purpose of this article is to derive a similar formula for generalized eigenfunctions in a more general setting, using a modified method.

The class of processes considered here consists of symmetric (one-di\-men\-sion\-al) L{\'e}vy processes $X_t$ with L{\'e}vy measure having completely monotone density function on $(0, \infty)$. Alternatively, this class can be described as (one-dimensional) subordinate Brownian motions, corresponding to subordinators with L{\'e}vy measure having completely monotone density function. Yet another characterization is given by the condition that the L{\'e}vy-Khintchine exponent of $X_t$ has the form $\psi(\xi^2)$ for a \emph{complete Bernstein function} $\psi(\xi)$. The equivalence of the above is given in Proposition~\ref{prop:cbf}. This class of L{\'e}vy processes have attracted much attention in the last decade; see, for example,~\cite{bib:bbkrsv09, bib:cs05, bib:cs06, bib:hk09, bib:ksv09, bib:ksv10a, bib:ksv10, bib:ksv11, bib:sv06, bib:sv08} for some recent developments. There is also an extensive literature focused specifically on symmetric $\alpha$-stable processes (\cite{bib:bk04, bib:bk06, bib:bb00, bib:bgr10, bib:cs97, bib:d90, bib:d04, bib:dm07, bib:i09, bib:k97, bib:k98, bib:mo69}) and relativistic $\alpha$-stable processes (\cite{bib:bmr09, bib:bmr10, bib:cks12, bib:ks06, bib:r02}), which are included here as examples.

We consider the process $X_t$ killed upon leaving the half-line $(0, \infty)$. Let $P^\hl_t$, $\A_\hl$ and $\domain(\A_\hl; L^\infty)$ denote the corresponding transition operators, the $L^\infty(\hl)$ generator, and the domain of $\A_\hl$, respectively (formal definitions are given in the Preliminaries). The following generalization of Theorem~2 from~\cite{bib:kkms10} is the main result of the article. Its main advantage is the explicit description of the eigenfunctions.

\begin{thm}
\label{th:eigenfunctions}
Suppose that the L{\'e}vy-Khintchine exponent of $X_t$ has the form $\psi(\xi^2)$ for a complete Bernstein function $\psi$. For all $\lambda > 0$, there is a bounded function $F_\lambda$ on $\hl$ which is the eigenfunction of $P^\hl_t$ and $\A_\hl$:
\formula{
 P^\hl_t F_\lambda(x) & = e^{-t \psi(\lambda^2)} F_\lambda(x) && \text{and} & \A_\hl F_\lambda(x) & = -\psi(\lambda^2) F_\lambda(x)
}
for all $x > 0$. The function $F_\lambda$ is characterized by its Laplace transform:
\formula[eq:lf]{
 \laplace F_\lambda(\xi) & = \frac{\lambda}{\lambda^2 + \xi^2} \, \exp\expr{\frac{1}{\pi} \int_0^\infty \frac{\xi}{\xi^2 + \zeta^2} \, \log \frac{\psi'(\lambda^2) (\lambda^2 - \zeta^2)}{\psi(\lambda^2) - \psi(\zeta^2)} \, d\zeta}
}
for $\xi \in \C$ such that $\real \xi > 0$. Furthermore, for $x > 0$ we have
\formula[eq:f]{
 F_\lambda(x) & = \sin(\lambda x + \thet_\lambda) - G_\lambda(x) ,
}
where $\thet_\lambda \in [0, \pi/2)$, and $G_\lambda(x)$ is a bounded, completely monotone function on $(0, \infty)$. More precisely, we have
\formula[eq:theta]{
 \thet_\lambda & = -\frac{1}{\pi} \int_0^\infty \frac{\lambda}{\lambda^2 - \zeta^2} \, \log \frac{\psi'(\lambda^2) (\lambda^2 - \zeta^2)}{\psi(\lambda^2) - \psi(\zeta^2)} \, d\zeta ,
}
and $G_\lambda$ is the Laplace transform of a finite measure $\gamma_\lambda$ on $(0, \infty)$. When $\psi(\xi)$ extends to a function $\psi^+(\xi)$ holomorphic in the upper complex half-plane $\{\xi \in \C : \imag \xi > 0\}$ and continuous in $\{\xi \in \C : \imag \xi \ge 0\}$, and furthermore $\psi^+(-\xi) \ne \psi(\lambda)$ for all $\xi > 0$, then the measure $\gamma_\lambda$ is absolutely continuous, and
\formula[eq:gamma0]{
\begin{aligned}
 \gamma_\lambda(d\xi) & = \frac{1}{\pi} \expr{\imag \frac{\lambda \psi'(\lambda^2)}{\psi(\lambda^2) - \psi^+(-\xi^2)}} \\ & \qquad \times \exp\expr{-\frac{1}{\pi} \int_0^\infty \frac{\xi}{\xi^2 + \zeta^2} \, \log \frac{\psi'(\lambda^2) (\lambda^2 - \zeta^2)}{\psi(\lambda^2) - \psi(\zeta^2)} \, d\zeta} d\xi
\end{aligned}
}
for $\xi > 0$.
\end{thm}

In some cases, the above formulae can be substantially simplified. For example, when $X_t$ is the symmetric $\alpha$-stable process, then $\thet_\lambda = (2 - \alpha) \pi/8$, and $F_\lambda(x) = F_1(\lambda x)$; see Section~\ref{sec:examples}.

Any complete Bernstein function $\psi(\xi)$ extends holomorphically to $\C \setminus (-\infty, 0]$. Hence, the assumption in the last part of Theorem~\ref{th:eigenfunctions} concerns the existence of boundary values of $\psi$ (approached from the upper complex half-plane) along the branch cut $(-\infty, 0]$. A formula similar to~\eqref{eq:gamma0} can be given in the general case, as discussed in Remark~\ref{rem:gamma}.

It should be emphasized that although Theorem~\ref{th:eigenfunctions} extends the result of~\cite{bib:kkms10}, its proof is essentially different.

\begin{rem}
\label{rem:anal}
Entirely analytical formulation of Theorem~\ref{th:eigenfunctions} is available; see~\cite{bib:j01} for the analytical definition of $\A_\hl$. Also, our arguments are purely analytical, with the exception of probabilistic proofs of Lemmas~\ref{lem:d2} and~\ref{lem:smooth:domain}.\qed
\end{rem}

The eigenfunctions $F_\lambda$ \emph{never} belong to $L^2(\hl)$: this reflects the fact that the spectrum of transition operators $P^\hl_t$, considered as operators on $L^2(\hl)$, is purely continuous. Nevertheless, the functions $F_\lambda$ yield a \emph{generalized eigenfunction expansion} of $P^\hl_t$. In classical eigenfunction expansion, a function is decomposed with respect to a complete orthonormal set of eigenfunctions, and the operator acts at each component independently. Informally, in \emph{generalized} eigenfunction expansions, the principle is the same, but eigenfunctions are no longer square-integrable and there are uncountably many of them. A formal statement is given in the following result, which generalizes Theorem~3 in~\cite{bib:kkms10}, and explicitly provides functional calculus for $\A_\hl$. Here $\A_\hl$ and $\domain(\A_\hl; L^2)$ denote the $L^2(\hl)$ generator of the semigroup $P^\hl_t$, and the domain of $\A_\hl$, respectively (again, see Preliminaries for formal definitions).

\begin{thm}
\label{th:spectral}
With the notation of Theorem~\ref{th:eigenfunctions}, let
\formula[eq:pistar]{
 \Pi f(\lambda) & = \int_0^\infty f(x) F_\lambda(x) dx && \text{for $\lambda > 0$, $f \in C_c(\hl)$.}
}
Then $\sqrt{2 / \pi} \, \Pi$ extends to a unitary operator on $L^2(\hl)$, and
\formula{
 \Pi P^\hl_t f(\lambda) & = e^{-t \psi(\lambda^2)} \Pi f(\lambda) && \text{for $f \in L^2(\hl)$.}
}
Furthermore, $f \in \domain(\A_\hl; L^2)$ if and only if $\psi(\lambda^2) \Pi f(\lambda)$ is in $L^2(\hl)$, and
\formula{
 \Pi \A_\hl f(\lambda) = -\psi(\lambda^2) \Pi f(\lambda) && \text{for $f \in \domain(\A_\hl; L^2)$.}
}
\end{thm}

\begin{rem}
Below we provide a complete proof of Theorem~\ref{th:spectral} under an additional condition that $\Pi$ is injective, and we verify that $\Pi$ is injective when $X_t$ is the symmetric $\alpha$-stable process, $\alpha \in (0, 2]$ (that is, $\psi(\xi) = \xi^{\alpha/2}$). In the preliminary version of this article, a relatively easy to check, but rather restrictive sufficient condition for injectivity of $\Pi$ was given (Lemma~19 in~\cite{bib:k10:v1}), and it was conjectured that in fact $\Pi$ is always injective. After a few months, this conjecture was solved in affirmative by Jacek Ma{\l}ecki, and the proof of injectivity of $\Pi$ is given in the recent preprint~\cite{bib:kmr11a}. We announce the result of~\cite{bib:kmr11a} and state Theorem~\ref{th:spectral} in full generality.\qed
\end{rem}

\begin{rem}
There is a striking similarity between formulae~\eqref{eq:lf} and~\eqref{eq:gamma0} on one hand, and some formulae in fluctuation theory of L{\'e}vy processes on the other. More precisely, the exponent in~\eqref{eq:lf} and~\eqref{eq:gamma0} resembles known formulae for the Laplace exponent of ladder processes (see Corollary~9.7 in~\cite{bib:f74}) and the Laplace transform of the supremum functional (Theorem~1 in~\cite{bib:bd57}). This phenomenon can be easily explained at the analytical level: in both cases explicit formulae are obtained using the \emph{Wiener-Hopf method}, a technique for solving some integral equations in the half-line via Fourier transform. Fluctuation theory of L{\'e}vy processes relies upon the Wiener-Hopf factorization of the $\lambda$-potential operator (that is, the resolvent) of the transition semigroup of a L{\'e}vy process. The Fourier symbol of this operator is $1 / (\lambda + \psi(\xi^2))$. In the proof of Theorem~\ref{th:eigenfunctions}, the operator with Fourier symbol $(\psi(\lambda^2) - \psi(\xi^2)) / (\lambda^2 - \xi^2)$ is factorized.

A probabilistic formulation of the fluctuation theory is available through local times, excursion theory and ladder processes; see~\cite{bib:b98, bib:d07, bib:k06, bib:s99}. It is an interesting open problem whether there is a similar \emph{probabilistic} derivation of the formula for $F_\lambda$.\qed
\end{rem}

The proof of Theorem~\ref{th:spectral} (under the assumption that $\Pi$ is injective) follows the line of~\cite{bib:kkms10}. Proofs of Theorems~\ref{th:eigenfunctions} and~\ref{th:spectral} are rather technical. Therefore, it may be helpful to keep in mind the following well-recognized example.

\begin{exa}
Suppose that $X_t$ is the Brownian motion, with variance $2 t$ (so that the generator of $X_t$ is the one-dimensional Laplace operator $d^2 / dx^2$, and $\psi(\xi) = \xi$). In this case $\A_\hl$ is the Laplace operator in $(0, \infty)$ with Dirichlet boundary condition at $0$, and $P^\hl_t$ is the classical heat semigroup on $\hl$ with Dirichlet boundary condition at $0$. The eigenfunctions of $\A_\hl$ and $P^\hl_t$ are simply $F_\lambda(x) = \sin(\lambda x)$, with corresponding eigenvalues $-\lambda^2$ and $e^{-t \lambda^2}$, respectively. The integral transform $\Pi$ is the Fourier sine transform on $(0, \infty)$, and Theorem~\ref{th:spectral} states that the Fourier sine transform is the unitary mapping (up to a constant factor $\sqrt{2 / \pi}$), which diagonalizes the action of the heat semigroup.
\end{exa}

To the author's knowledge, no results similar to Theorems~\ref{th:eigenfunctions} and~\ref{th:spectral} were available for L{\'e}vy processes other than the Brownian motion (possibly with drift; this requires a minor modification to the above example) and the Cauchy process (studied in~\cite{bib:kkms10}). Although generalized eigenfunction expansions similar to Theorem~\ref{th:spectral} have been studied for various classes of operators (see~\cite{bib:g59} for the case of general Markov processes, and~\cite{bib:psw89, bib:s82} for diffusion processes with Feynman-Kac potential), the explicit description of eigenfunctions was not available. Consequently, their applications were limited. The novelty of Theorem~\ref{th:spectral} is in that the generalized eigenfunctions $F_\lambda(x)$ are given in a fairly explicit form, allowing various estimates, asymptotic analysis and numerical approximation. For this reason, one can expect that Theorems~\ref{th:eigenfunctions} and~\ref{th:spectral} will find a variety of applications; some are already discussed below.

Theorems~\ref{th:eigenfunctions} and~\ref{th:spectral} yield formulae for the transition density $p^\hl_t(x, y)$ of $X_t$ killed upon leaving the half-line $(0, \infty)$, and for the distribution of the first exit time $\tau_\hl$ from the positive half-line. These results are contained in the following two theorems. By $\pr_x$ we denote the probability for the process $X_t$ starting at a fixed point $x > 0$.

\begin{thm}
\label{th:pdt}
Let $t > 0$. If $e^{-t \psi(\xi^2)}$ is integrable in $\xi > 0$, then
\formula[eq:pdt]{
 p^\hl_t(x, y) & = \frac{2}{\pi} \int_0^\infty e^{-t \psi(\lambda^2)} F_\lambda(x) F_\lambda(y) d\lambda && \text{for $x, y > 0$.}
}
\end{thm}

\begin{thm}
\label{th:fpt}
Let $t > 0$. If
\formula[eq:fpt:a1]{
 \sup_{\xi > 0} \frac{\xi |\psi''(\xi)|}{\psi'(\xi)} < 2 ,
}
and
\formula[eq:fpt:a2]{
 \int_1^\infty \sqrt{\frac{\psi'(\xi^2)}{\psi(\xi^2)}} \, e^{-t \psi(\xi^2)} d\xi < \infty ,
}
then
\formula[eq:fpt]{
 \pr_x(\tau_\hl > t) & = \frac{2}{\pi} \int_0^\infty \sqrt{\frac{\psi'(\lambda^2)}{\psi(\lambda^2)}} \, e^{-t \psi(\lambda^2)} F_\lambda(x) d\lambda && \text{for $x > 0$.}
}
If $\sqrt{\psi(\xi^2) \psi'(\xi^2)} e^{-t \psi(\xi^2)}$ is integrable in $\xi > 1$ for all $t > 0$, then furthermore
\formula[eq:fptd]{
 \pr_x(\tau_\hl \in dt) & = \frac{2}{\pi} \expr{\int_0^\infty \sqrt{\psi'(\lambda^2) \psi(\lambda^2)} \, e^{-t \psi(\lambda^2)} F_\lambda(x) d\lambda} dt
}
for $x, t > 0$.
\end{thm}

\begin{rem}
Theorem~\ref{th:fpt} in full generality is proved in the recent preprint~\cite{bib:kmr11a}. In this article we provide a much simpler proof, under a more restrictive condition
\formula[eq:fpt:a3]{
 \limsup_{\xi \to 0^+} \frac{\xi |\psi''(\xi)|}{\psi'(\xi)} & < 1 , & \limsup_{\xi \to \infty} \frac{\xi |\psi''(\xi)|}{\psi'(\xi)} & < 1 .
}
This condition implies both~\eqref{eq:fpt:a1} and~\eqref{eq:fpt:a2}, as well as integrability of $\sqrt{\psi(\xi^2) \psi'(\xi^2)} e^{-t \psi(\xi^2)}$ (see the proof of the theorem). The proof given below relies on Theorem~\ref{th:spectral}. Noteworthy, many interesting examples, including symmetric $\alpha$-stable processes and relativistic $\alpha$-stable processes, satisfy~\eqref{eq:fpt:a3}. By Corollary~\ref{cor:power2}, condition~\eqref{eq:fpt:a3} is satisfied when $\psi$ is regularly varying of positive order at $0$ and at $\infty$.\qed
\end{rem}

It should be emphasized that the integrand in~\eqref{eq:pdt} has two oscillatory factors, and cancellations in the integral are essential. Hence, formula~\eqref{eq:pdt} is problematic for numerical computations. For the Cauchy process, a significant simplification of~\eqref{eq:pdt} is possible (see Theorem~4 in~\cite{bib:kkms10}). It is an open problem whether formula~\eqref{eq:pdt} can be simplified in the more general case, even for symmetric $\alpha$-stable processes for general $\alpha \in (0, 2]$.

Cancellation is of less importance in Theorem~\ref{th:fpt}; see~\cite{bib:kmr11a}, where~\eqref{eq:fptd} is used to obtain estimates for the density function of the distribution of $\tau_\hl$.

Since $\xi |\psi''(\xi)| / \psi'(\xi) < 2$ for all $\xi > 0$, the supremum in~\eqref{eq:fpt:a1} is always not greater than $2$ (see Preliminaries). The assumptions~\eqref{eq:fpt:a1} and~\eqref{eq:fpt:a2} are rather mild regularity and growth conditions for $\psi(\xi)$; examples are given in Section~\ref{sec:examples}.

\begin{rem}
By symmetry, the first exit time $\tau_\hl$ for $X_t$ starting at $x > 0$ has the same distribution as the first passage time $\tau_x$ through a barrier at $x$ for $X_t$ starting at $0$. On the other hand, first passage times are related to the supremum functional:
\formula{
 \pr_0 \expr{\sup_{s \in [0, t]} X_s < x} & = \pr_0(\tau_x > t) = \pr_x(\tau_\hl > t) .
}
Hence, Theorem~\ref{th:fpt} gives, in a rather general setting, an explicit expression for the distribution of first passage times and the supremum functional, which are fundamental objects in fluctuation theory of L{\'e}vy processes. Noteworthy, the double Laplace transform (in $t$ and $x$) of the distribution of $\tau_x$ is known for general L{\'e}vy processes since 1957 (Theorem~1 in~\cite{bib:bd57}). Nevertheless, explicit formulae for $\pr_x(\tau_\hl > t)$ are known only in some special cases; see~\cite{bib:ds10, bib:gj10, bib:hk09, bib:k10, bib:kkp10} for some recent developments in this area. Also, a formula for the single Laplace transform (in $t$) of $\pr(\tau_x > t)$ for a large class of symmetric L{\'e}vy processes was obtained only very recently in~\cite{bib:kmr11}.\qed
\end{rem}

Theorems~\ref{th:eigenfunctions} and~\ref{th:spectral} can be applied to certain systems of PDEs, related to traces of two-dimensional diffusions, as discussed in Section~\ref{sec:applications}. Furthermore, there are at least four recent preprints exploiting the results of the present article. An application of Theorem~\ref{th:eigenfunctions} to the spectral theory of symmetric $\alpha$-stable process in the interval can be found in~\cite{bib:k10a}. The case of the relativistic $1$-stable process in half-line and interval is studied in detail in~\cite{bib:kkm11}. This application may be of particular interest in quantum physics; see~\cite{bib:els08, bib:gko95, bib:ls09} for related research. In~\cite{bib:fg11}, Theorems~\ref{th:eigenfunctions} and~\ref{th:spectral} are used to obtain refined semi-classical asymptotics for eigenvalues of higher-dimensional isotropic $\alpha$-stable processes in domains. Finally, the proof of Theorems~\ref{th:spectral} and~\ref{th:fpt} in full generality, as well as their application to a detailed analysis of first passage times, can be found in~\cite{bib:kmr11a}.

We conclude the introduction with a brief description of the structure of the article. The Preliminaries section contains definitions and standard properties used in the article, and auxiliary lemmas. It includes some background on distribution theory, L{\'e}vy processes and their generators, subordinate Brownian motions and complete Bernstein functions. In Section~\ref{sec:cbf} a certain transformation related to the Wiener-Hopf factorization is studied. Next, in Section~\ref{sec:eigenfunctions}, we derive the formula for $F_\lambda$ and prove Theorem~\ref{th:eigenfunctions}. Theorems~\ref{th:spectral}, \ref{th:pdt} and~\ref{th:fpt} are proved in Section~\ref{sec:spectral}. Examples, including symmetric $\alpha$-stable processes and relativistic $\alpha$-stable processes, are studied in detail in Section~\ref{sec:examples}. In Section~\ref{sec:applications} we give an application of our results to systems of PDEs (Theorem~\ref{th:pde}).

%
%

\section{Preliminaries}
\label{sec:preliminaries}

%
%

\subsection{Distribution theory}

In the theory of partial differential equations, it is a common case that one first finds a weak solution to a problem, and then, by showing that the solution is sufficiently regular, one argues that it is in fact a strong solution. This approach will be used to prove Theorem~\ref{th:eigenfunctions}: first we find a distributional eigenfunction of $\A_\hl$, and then we show that in fact it belongs to the domain of $\A_\hl$. Our argument relies heavily on Fourier methods, and the operators involved are typically non-local, so the required background on distribution theory is slightly different from the one used in partial differential equations. It should be emphasized that the use of distributions could be avoided, at least partially, at the price of less clear exposition. However, the language of distribution theory seems to be well-suited for problems involving generators of killed L{\'e}vy processes.

Let $\schwartz$ denote the class of Schwartz functions in $\R^d$ ($d = 1, 2, ...$), and let $\schwartz'$ be the space od tempered distributions in $\R^d$. If $\ph \in \schwartz$ and $F \in \schwartz'$, we write $\tscalar{F, \ph}$ for the value of $F$ at $\ph$. Below we recall some well-known properties of tempered distributions; for a detailed exposition of the theory, see e.g.~\cite{bib:v02}.

If $\ph \in \schwartz$, the Fourier transform of $\ph$ is $\fourier \ph(\xi) = \int_{\R^d} e^{i \xi \cdot x} \ph(x) dx$. For $F \in \schwartz'$, $\fourier F$ is the tempered distribution satisfying $\tscalar{\fourier F, \ph} = \tscalar{F, \fourier \ph}$.

The convolution of $\ph_1, \ph_2 \in \schwartz$ is defined in the usual way, $\ph_1 * \ph_2(x) = \int_{\R^d} \ph_1(y) \ph_2(x - y) dy$. When $\ph \in \schwartz$ and $F \in \schwartz'$, then $F * \ph$ is an infinitely smooth function, defined by
\formula{
 F * \ph(x) & = \tscalar{F, \ph_x} , && \text{where} & \ph_x(y) & = \ph(x - y) .
}
The convolution of two distributions is not well-defined in general. Suppose that $F_1, F_2 \in \schwartz'$. We say that $F_1$ and $F_2$ are \emph{$\schwartz'$-convolvable} if for all $\ph_1, \ph_2 \in \schwartz$, the functions $F_1 * \ph_1$ and $F_2 * \ph_2$ are convolvable in the usual sense, i.e. the integral $\int_{\R^d} (F_1 * \ph_1)(y) (F_2 * \ph_2)(x - y) dy$ exists for all $x$. When this is the case, the \emph{$\schwartz'$-convolution} $F_1 \conv F_2$ is the unique distribution $F$ satisfying $F * (\ph_1 * \ph_2) = (F_1 * \ph_1) * (F_2 * \ph_2)$ for $\ph_1, \ph_2 \in \schwartz$. Note that there are other non-equivalent definitions of the convolution of distributions, and $\schwartz'$-convolution is often denoted as $F_1 \circledast F_2$; for the discussion of various notions of convolvability, the reader is referred to~\cite{bib:dv78, bib:v02}.

Recall that the \emph{support} of a distribution $F$ is the smallest closed set $\supp F$ with the property that $\tscalar{F, \ph} = 0$ for all $\ph \in \schwartz$ such that $\ph(x) = 0$ for $x \in \supp F$. If any of the tempered distributions $F_1$, $F_2$ has compact support, or (in the one-dimensional case) if both $F_1$ and $F_2$ are supported in $[0, \infty)$, then $F_1$ and $F_2$ are automatically $\schwartz'$-convolvable. 

It is well known that the distributions $(F_1 \conv F_2) \conv F_3$ and $F_1 \conv (F_2 \conv F_3)$ need not be equal; however, if the pairs $(F_1, F_2)$ and $(F_2, F_3)$ are $\schwartz'$-convolvable, and furthermore the functions $F_1 * \ph_1$, $F_2 * \ph_2$, and $F_3 * \ph_3$ are convolvable (that is, $(F_1 * \ph_1)(y) (F_2 * \ph_2)(z) (F_3 * \ph_3)(x - y - z)$ is integrable in $y$, $z$ for all $x \in \R^d$) for any $\ph_1, \ph_2, \ph_3 \in \schwartz$, then the $\schwartz'$-convolution of $F_1$, $F_2$ and $F_3$ is associative; see~\cite{bib:v02}, Section~4.2.8.

Any $\schwartz'$-convolvable distributions $F_1$, $F_2$ satisfy the \emph{exchange formula} $\fourier (F_1 \conv F_2) = \fourier F_1 \cdot \fourier F_2$, where the multiplication of distributions $\fourier F_1$ and $\fourier F_2$ extends standard multiplication of functions in an appropriate manner, see~\cite{bib:ho58, bib:v02}. Since we only use the exchange formula when $\fourier F_1$ and $\fourier F_2$ are genuine functions, or when $\fourier F_1$ is a measure and $\fourier F_2$ is a function, we do not discuss the notion of multiplication for general distributions and refer the interested reader to~\cite{bib:k82, bib:si64, bib:v02}.

In the one-dimensional case $d = 1$, the Laplace transform can be defined for tempered distributions. If $\ph \in \schwartz$ is supported in $[0, \infty)$, then the Laplace transform of $\ph$ is denoted by $\laplace \ph(\xi) = \int_\R e^{-\xi x} \ph(x) dx$ ($\real \xi \ge 0$). This is a holomorphic function of $\xi$ in the right complex half-plane $\real \xi > 0$, continuous at the boundary. Clearly, $\fourier \ph(\xi) = \laplace \ph(-i \xi)$. If $F \in \schwartz'$ is supported in $[0, \infty)$, the Laplace transform of $F$ is defined for $\xi \in \C$ with $\real \xi > 0$ by
\formula{
 \laplace F(\xi) & = \scalar{F, e_\xi h} ,
}
where $e_\xi(x) = e^{-\xi x}$ and $h$ is any infinitely smooth function such that $h(x) = 1$ for $x \ge 0$ and $h(x) = 0$ for $x \le -1$. This definition does not depend on the choice of $h$ (see formula~(9.1.4) in~\cite{bib:v02}). The above definition of $\laplace F$ extends the usual definition of the Laplace transform of signed measures on $[0, \infty)$.

Note that while the Fourier transform of a distribution is again a distribution, the Laplace transform is always a (holomorphic) function. For a fixed $t > 0$, the function $\laplace F(t - i s)$ is the Fourier transform of the distribution $e_t F$ (where again $e_t(x) = e^{-t x}$), and as $t \to 0^+$, the distributions $\laplace F(t - i s)$ converge to $\fourier F$ in $\schwartz'$. The exchange formula holds also for the Laplace transform: if both $F_1, F_2 \in \schwartz'$ are supported in $[0, \infty)$, then they are $\schwartz'$-convolvable and $\laplace (F_1 \conv F_2)(\xi) = \laplace F_1(\xi) \laplace F_2(\xi)$ for all $\xi$ with $\real \xi > 0$.

We say that two distributions $F_1$, $F_2$ restricted to an open set $D$ are equal if $\tscalar{F_1, \ph} = \tscalar{F_2, \ph}$ for any $\ph \in \schwartz$ vanishing in $\R^d \setminus D$. Equivalently: $F_1 - F_2$ is supported in $\R^d \setminus D$.

%
%

\subsection{Transition semigroups and generators}
\label{subsec:gen}

Below we recall and extend some standard definitions and properties, which can be found, for example, in~\cite{bib:a04, bib:b98, bib:s99}. In this article we are only concerned with one-dimensional subordinate Brownian motions. However, Lemma~\ref{lem:d2} below might be of interest for general L{\'e}vy processes. For this reason, in this and the next subsection, we work with the general case.

Let $X_t$ be a L{\'e}vy process in $\R^d$. We write $\pr_x$ and $\ex_x$ for the probability and expectation for the process $X_t$ which starts at $x \in \R^d$. The process $X_t$ is completely determined by its L{\'e}vy-Khintchine exponent $\Psi$: for $\xi \in \R^d$ we have
\formula[eq:lk:general0]{
 \ex_0 e^{i \xi \cdot X_t} = e^{-t \Psi(\xi)} ,
}
where
\formula[eq:lk:general]{
 \Psi(\xi) & = \beta \xi \cdot \xi - i \gamma \cdot \xi + \int_{\R^d} (1 - e^{i \xi \cdot z} + \ind_{|z| \le 1} i \xi \cdot z) \nu(dz) .
}
Here $\beta$ is a nonnegative definite $d \times d$ matrix (the \emph{diffusion coefficient}), $\gamma$ is a vector in $\R^d$ (the \emph{drift}), and $\nu$ is the \emph{L{\'e}vy measure} of $X_t$: a Radon measure on $\R^d \setminus \{0\}$ such that $\int_{\R^d} \min(1, |z|^2) \nu(dz) < \infty$. If $\nu$ is absolutely continuous, we denote its density function with the same symbol $\nu(z)$.

The \emph{transition operators} of $X_t$ are defined by
\formula{
 P_t f(x) & = \ex_x f(X_t) = \int_{\R^d} f(y) \pr_x(X_t \in dy) , && t > 0 , \, x \in \R ,
}
whenever the integral is absolutely convergent. Each $P_t$ is a convolution operator, with convolution kernel given by the distribution of $(-X_t)$ under $\pr_0$. By the L{\'e}vy-Khintchine formula~\eqref{eq:lk:general0},
\formula[eq:lk:pt]{
 P_t e_\xi & = e^{-t \Psi(\xi)} e_\xi && \text{for} & e_\xi(x) & = e^{i \xi \cdot x} .
}
Furthermore, for $\ph \in \schwartz$,
\formula[eq:ptfourier]{
 \fourier P_t \ph(\xi) & = e^{-t \overline{\Psi(\xi)}} \fourier \ph(\xi) , && \xi \in \R^d ;
}
see Theorem~3.3.3 in~\cite{bib:a04} (note that in~\cite{bib:a04}, Fourier transform is defined with $e^{-i \xi \cdot x}$ instead of $e^{i \xi \cdot x}$). The operators $P_t$ form a contraction semigroup on each of the Lebesgue spaces $L^p(\R^d)$ ($p \in [1, \infty]$), on the space $C_b(\R^d)$ of bounded continuous functions, and on the space $C_0(\R^d)$ of continuous functions vanishing at infinity (all equipped with usual norms). On $L^p(\R^d)$ ($p \in [1, \infty)$) and on $C_0(\R^d)$, this semigroup is strongly continuous (see Section~3.4 in~\cite{bib:a04}). The generator of each of these semigroups is denoted by the same symbol $\A$, and we write $\domain(\A; \mathcal{X})$ for the corresponding domain, where $\mathcal{X}$ is either $L^p(\R^d)$ ($p \in [1, \infty]$), $C_b(\R^d)$ or $C_0(\R^d)$. More precisely, $f \in \domain(\A; \mathcal{X})$ if $f \in \mathcal{X}$ and the limit
\formula{
 \A f & = \lim_{t \to 0^+} \frac{P_t f - f}{t}
}
exists in the topology of $\mathcal{X}$. Note that the limit, if it exists, does not depend on the choice of $\mathcal{X}$ (apart from the fact that in $C_0(\R^d)$ and $C_b(\R^d)$ it is defined pointwise, while in $L^p(\R^d)$ only up to a set of zero measure). Therefore, using a single symbol $\A$ for operators acting on different domains $\domain(\A; \mathcal{X})$ causes no confusion. The Schwartz class $\schwartz$ is a core of $\A$ on each of the spaces $L^p(\R^d)$ ($p \in [1, \infty)$) and $C_0(\R^d)$, and by~\eqref{eq:ptfourier}, for $\ph \in \schwartz$,
\formula[eq:genfourier]{
 \fourier \A \ph(\xi) & = -\overline{\Psi(\xi)} \fourier \ph(\xi) , && \xi \in \R .
}
We abbreviate, for example, $\domain(\A; L^p(\R^d))$ to $\domain(\A; L^p)$.

\begin{rem}
\label{rem:spectral:free}
The spectral theory of $P_t$ and $\A$ is very simple, thanks to the L{\'e}vy Khintchine formula~\eqref{eq:lk:pt}. Indeed, the function $e^{i \xi \cdot x}$ is the eigenfunction of $P_t$, with eigenvalue $e^{-t \Psi(\xi)}$. Hence, $e^{i \xi \cdot x}$ belongs to $\domain(\A; L^\infty)$, and it is an eigenfunction of $\A$ with eigenvalue $-\Psi(\xi)$. The generalized eigenfunction expansion of $P_t$ and $\A$ is given by~\eqref{eq:ptfourier} and~\eqref{eq:genfourier}; the Fourier transform plays the same role as the integral transform~$\Pi$ in Theorem~\ref{th:spectral}.

If $X_t$ is symmetric, then $\Psi(\xi) = \Psi(-\xi)$ is real. Since $\sin(\xi \cdot x + \thet)$ is a linear combination of $e^{i \xi \cdot x}$ and $e^{i (-\xi) \cdot x}$, it is also the eigenfunction of $P_t$ and $\A$. Note that the eigenfunctions $F_\lambda$ in Theorem~\ref{th:eigenfunctions} behave asymptotically as $\sin(\lambda x + \thet_\lambda)$ as $x \to \infty$.\qed
\end{rem}

When $\ph \in \schwartz$, then the generator of $X_t$ can be written in the form
\formula[eq:gen:lk]{
\begin{aligned}
 \A \ph(x) & = \beta \nabla \ph(x) \cdot \nabla \ph(x) + \gamma \cdot \nabla \ph(x) \\ & + \int_{\R^d} (\ph(x + z) - \ph(x) - \ind_{|z| \le 1} z \cdot \nabla \ph(x)) \nu(dz) ,
\end{aligned}
}
where $\beta$, $\gamma$ and $\nu$ are as in~\eqref{eq:lk:general} (Theorem~31.5 in~\cite{bib:s99}). Furthermore,
\formula[eq:gen:bound]{
 \|\A \ph\|_{L^\infty(\R^d)} & \le C \|\ph\|_{C^2_b(\R^d)} , && \ph \in \schwartz,
}
where $C$ depends only on the process $X_t$, and $\|\ph\|_{C^2_b(\R^d)}$ is the maximum of $L^\infty(\R^d)$ norms of $\ph$ and its first and second partial derivatives (see the proof of Theorem~31.5 in~\cite{bib:s99}).

We need the following extension of~\eqref{eq:gen:lk}. For $C_0(\R^d)$ instead of $C_b(\R^d)$, this is well-known, see Theorem~31.5 in~\cite{bib:s99}. The statement given below is, however, difficult to find in the literature, so we provide a proof. Let $C_b^\infty(\R^d)$ be the class of functions $F$ such that $F$ and all partial derivatives of $F$ (of all orders) belong to $C_b(\R^d)$.

\begin{lem}
\label{lem:gen:cb}
If $F \in C_b^\infty(\R^d)$, then $F \in \domain(\A; C_b)$ and
\formula[eq:gen:cb]{
\begin{aligned}
 \A F(x) & = \beta \nabla F(x) \cdot \nabla F(x) + \gamma \cdot \nabla F(x) \\ & + \int_{\R^d} (F(x + z) - F(x) - \ind_{|z| \le 1} z \cdot \nabla F(x)) \nu(dz) .
\end{aligned}
}
\end{lem}

\begin{rem}
\label{rem:domain:counterexample}
Note that in contrast to the $C_0(\R^d)$ case, $F$ may fail to belong to $\domain(\A; C_b)$ when $F$ and its first and second partial derivatives are in $C_b(\R^d)$. A simple counterexample can be easily constructed for one-dimensional Brownian motion: if $F, F', F'' \in C_b(\R)$, but $F''$ is not \emph{uniformly} continuous, then the convergence of $(P_t f - f) / t$ may fail to be uniform. We omit the details.\qed
\end{rem}

\begin{proof}[Proof of Lemma~\ref{lem:gen:cb}]
Let $\|f\|_{C^k_b(\R^d)}$ denote the maximum of $L^\infty(\R^d)$ norms of $f$ and its partial derivatives of order not greater than $k$, and denote $M = \|F\|_{C^4_b(\R^d)}$. There is a sequence $\ph_n \in \schwartz$ such that $\ph_n$ and its first and second partial derivatives converge locally uniformly to $F$ and the corresponding first and second partial derivatives of $F$, and furthermore $\|\ph_n\|_{C^4_b(\R^d)} \le C M$ (with $C$ depending only on the dimension $d$). For $f \in \domain(\A; L^\infty)$ and $x \in \R^d$, we have (see~\cite{bib:d65})
\formula[eq:dpt]{
\begin{aligned}
 P_t f(x) - f(x) & = \int_0^t P_s \A f(x) ds = \int_0^t \A P_s f(x) ds .
\end{aligned}
}
In particular, $\|P_t f - f\|_{L^\infty(\R^d)} \le t \|\A f\|_{L^\infty(\R^d)}$. Since $\ph_n \in \schwartz \sub \domain(\A; C_0)$, by~\eqref{eq:dpt} we have
\formula{
 \abs{\frac{P_t \ph_n(x) - \ph_n(x)}{t} - \A \ph_n(x)} & = \abs{\frac{1}{t} \int_0^t \A P_s \ph_n(x) ds - \A \ph_n(x)} \\
 & \le \frac{1}{t} \int_0^t |\A (P_s \ph_n - \ph_n)(x)| ds .
}
By~\eqref{eq:gen:bound},
\formula{
 \abs{\frac{P_t \ph_n(x) - \ph_n(x)}{t} - \A \ph_n(x)} & \le \frac{1}{t} \int_0^t \|P_s \ph_n - \ph_n\|_{C_b^2(\R^d)} ds .
}
Since $P_s$ and $\A$ commute with partial differential operators on $\schwartz$, by~\eqref{eq:dpt},
\formula{
 \abs{\frac{P_t \ph_n(x) - \ph_n(x)}{t} - \A \ph_n(x)} & \le \frac{1}{t} \int_0^t s \|\A \ph_n\|_{C_b^2(\R^d)} ds = \frac{t}{2} \, \|\A \ph_n\|_{C_b^2(\R^d)} .
}
Finally, by~\eqref{eq:gen:bound},
\formula{
 \abs{\frac{P_t \ph_n(x) - \ph_n(x)}{t} - \A \ph_n(x)} & \le \frac{t}{2} \, \|\ph_n\|_{C_b^4(\R^d)} \le \frac{C t M}{2} \, .
}
As $n \to \infty$, we have $\ph_n(x) \to F(x)$ and $P_t \ph_n(x) \to P_t F(x)$ (by dominated convergence). By~\eqref{eq:gen:lk}, Taylor's theorem and dominated convergence, $\A \ph_n(x)$ converges to the right-hand side of~\eqref{eq:gen:cb}, which we denote by $G(x)$, and the convergence is locally uniform in $x \in \R^d$. It follows that
\formula{
 \abs{\frac{P_t F(x) - F(x)}{t} - G(x)} & \le \frac{C t M}{2} \, .
}
Since $x \in \R^d$ was arbitrary, we have $F \in \domain(\A; C_b)$ and $\A F(x) = G(x)$, as desired.
\end{proof}

Informally, the generator $\A$ is also a convolution operator, but the convolution kernel is a tempered distribution. We give this a precise meaning; see also~\cite{bib:b98, bib:s99}, and~\cite{bib:bb99} for the case of symmetric $\alpha$-stable processes.

\begin{defin}
\label{def:gen:distr}
The \emph{distributional generator} of $X_t$ is the tempered distribution $A \in \schwartz'$ defined by the formula
\formula{
 \scalar{A, \ph} & = \A \ph(0) , && \ph \in \schwartz.
}
\end{defin}

Let $\check{\ph}(x) = \ph(-x)$ for $\ph \in \schwartz$. We define $\check{A}$ (the \emph{reflection} of $A$) by the formula $\tscalar{\check{A}, \ph} = \tscalar{A, \check{\ph}}$. Hence, $\tscalar{A, \ph} = \check{A} * \ph(0)$ and $\A \ph = \check{A} * \ph$ for $\ph \in \schwartz$ (see~\cite{bib:v02}). By~\eqref{eq:lk:pt}, $\fourier \check{A}(\xi) = -\overline{\Psi(\xi)}$ and $\fourier A(\xi) = -\Psi(\xi)$ for $\xi \in \R^d$.

\begin{prop}
\label{prop:gen:distr}
When $F \in \domain(\A; L^\infty)$, then $\A F = \check{A} \conv F$.
\end{prop}

\begin{proof}
We have, by the definition of $\schwartz'$-convolution,
\formula{
 (\check{A} \conv F) * \ph_1 * \ph_2 & = (\check{A} * \ph_1) * (F * \ph_2) = (\A \ph_1) * F * \ph_2 .
}
Recall that $\schwartz \sub \domain(\A; L^1)$. Hence, $\A \ph_1$ is the $L^1(\R^d)$ limit of $(P_t \ph_1 - \ph_1) / t$. Since $\ph_1$, $\ph_2$ and the convolution kernel of $P_t$ are integrable, and $F$ is bounded, by Fubini we have
\formula{
 (P_t \ph_1 - \ph_1) * F * \ph_2 & = (P_t F - F) * \ph_1 * \ph_2 .
}
By dominated convergence,
\formula{
 (\check{A} \conv F) * \ph_1 * \ph_2 & = \lim_{t \to 0^+} \frac{(P_t F - F) * \ph_1 * \ph_2}{t} = (\A F) * \ph_1 * \ph_2 \, ,
}
as desired.
\end{proof}

\begin{cor}
\label{cor:gen:cb}
If $F \in C_b^\infty(\R^d)$ (as in Lemma~\ref{lem:gen:cb}), then $\A F \in C_b^\infty(\R^d)$, and $\A$ commutes with partial differential operators on $C_b^\infty(\R^d)$.
\end{cor}

\begin{proof}
Led $D$ be a partial derivative operator (of arbitrary order). By Proposition~\ref{prop:gen:distr}, we have $D \A F = D (\check{A} \conv F) = \check{A} \conv (D F) = \A D F$.
\end{proof}

When $X_t$ is one-dimensional and symmetric, then~\eqref{eq:lk:general} takes the form
\formula[eq:lk]{
 \Psi(\xi) & = \beta \xi^2 + \int_{-\infty}^\infty (1 - \cos(\xi z)) \nu(dz) , && \xi \in \R ,
}
where $\beta \ge 0$ and $\int_\R \min(1, z^2) \nu(dz) < \infty$. For $f \in C_b^\infty(\R^d)$ (as in Lemma~\ref{lem:gen:cb}), we have
\formula[eq:generator]{
 \A f(x) & = \beta f''(x) + \pv \int_\R (f(x + z) - f(x)) \nu(dz) , && x \in \R .
}
Here $\pv \int$ denotes the Cauchy principal value integral:
\formula{
 \pv \int_\R (f(x + z) - f(x)) \nu(dz) & = \lim_{\eps \to 0^+} \int_{\R \setminus (-\eps, \eps)} (f(x + z) - f(x)) \nu(dz) .
}
Of course, when $X_t$ is symmetric, then $\check{A} = A$. 

%
%

\subsection{Killed process and its generator}
\label{subsec:kill}

The main references for the notion of a killed process (or part of a Markov process) are~\cite{bib:bg68, bib:d65}, where general strong Markov processes are studied. Here we consider a L{\'e}vy process $X_t$ in $\R^d$. Let $\tau_D$ be the time of first exit from $D$, $\tau_D = \inf \set{t \ge 0 : X_t \notin D}$. The \emph{killed process} $X^D_t$ is a strong Markov process in $D$ with lifetime $\tau_D$, such that $X^D_t = X_t$ for all $t < \tau_D$. The transition operators $P^D_t$, the generator $\A_D$, and domains $\domain(\A_D; \mathcal{X})$ corresponding to the killed process $X^D_t$ are defined in the same way as for the free process $X_t$, that is,
\formula{
 P^D_t f(x) & = \ex_x f(X^D_t) = \ex_x(f(X_t) \ind_{t < \tau_D}) ,
}
and
\formula{
 \A_D f & = \lim_{t \to 0^+} \frac{P^D_t f - f}{t} ,
}
whenever the limit exists in the topology of the function space $\mathcal{X}$; in this case, we write $f \in \domain(\A_D; \mathcal{X})$. Again we abbreviate, for example, $\domain(\A_D; L^p(D))$ to $\domain(\A_D; L^p)$.

When $X_t$ is the Brownian motion, then $\A$ is the Laplace operator in $\R^d$, and $\A_D$ is the Laplace operator in $D$ with Dirichlet boundary condition. Even in this case the relation between $\A_D$ and $\A$ is very delicate. On the other hand, this relation is crucial for the proof of Theorem~\ref{th:eigenfunctions}. Hence, we now discuss this topic in more detail.

Suppose first that $X_t$ is a compound Poisson process: $\beta = 0$, $\gamma = 0$ and $\nu$ is a finite measure. In this case, the situation is clear. The following result seems to be well-known to specialists, but it is difficult to find in literature (cf. Example~3.3.7 in~\cite{bib:a04} for the case $D = \R^d$).

\begin{lem}
\label{lem:cp}
Let $X_t$ be a compound Poisson process in $\R^d$, with L{\'e}vy measure $\nu$, and let $p \in [1, \infty]$. Then the semigroup $P^D_t$ is strongly continuous on $L^p(D)$, the generator $\A_D$ is a bounded operator on each $L^p(D)$, and
\formula{
 \A_D f(x) & = \ind_D(x) \A f(x) = \ind_D(x) \int_{\R^d} (f(x + z) - f(x)) \nu(dz)
}
for all $f \in L^p(D)$, with the usual extension $f(x) = 0$ for $x \notin D$.
\end{lem}

\begin{proof}
Denote by $M = \nu(\R^d \setminus \{0\})$ the intensity of jumps of $X_t$, and let $0 < \tau_1 < \tau_2 < ...$ be the sequence of (random) times when $X_t$ jumps. Let $f$ be in some $L^p(D)$ ($p \in [1, \infty]$), and let $x \in D$ be the starting point of $X_t$. Note that
\formula{
 P^D_t f(x) = \ex_x(f(X_t) \ind_{t < \tau_D}) & = \ex_x(f(X_t) \ind_{t < \tau_1}) \\ & + \ex_x(f(X_t) \ind_{\tau_1 \le t < \tau_2}) \\ & + \ex_x(f(X_t) \ind_{\tau_2 \le t < \tau_D}) ;
}
indeed, in the first summand $t < \tau_1$ implies that $t < \tau_D$, and in the second one, when $\tau_1 \le t < \tau_2$ and $t \ge \tau_D$, then $f(X_t) = 0$. We have $f(X_t) = f(x)$ in the first summand, and in the second one, $f(X_t) = f(X_{\tau_1})$ is independent of $(\tau_1, \tau_2)$. Furthermore, $\pr_x(X_{\tau_1} \in x + dz) = M^{-1} \nu(dz)$. Hence, if $\check{\nu}(E) = \nu(-E)$, we have
\formula{
 \ex_x f(X_{\tau_1}) & = \int_{\R^d} f(x + z) \nu(dz) = \int_{\R^d} f(x - z) \check{\nu}(dz) = f * \check{\nu}(x) .
}
It follows that,
\formula{
 P^D_t f(x) & = f(x) \pr_x(t < \tau_1) \\ & + M^{-1} f * \check{\nu}(x) \pr_x(\tau_1 \le t < \tau_2) \\ & + \ex_x(f(X_t) \ind_{\tau_2 \le t < \tau_D}) ,
}
and so, finally,
\formula[eq:cpest]{
\begin{aligned}
 & \abs{\frac{P^D_t f(x) - f(x)}{t} - (f * \check{\nu}(x) - M f(x))}
 \\ & \le \abs{\frac{\pr_x(t < \tau_1) - 1 + M t}{t} \, f(x)} + \abs{\frac{\pr_x(\tau_1 \le t < \tau_2) - M t}{M t} \, f * \check{\nu}(x)} \\ & \hspace*{15em} + \abs{\frac{\ex_x(f(X_t) \ind_{\tau_2 \le t < \tau_D})}{t}} .
\end{aligned}
}
The proof will be complete once we show that each of the summands on the right-hand side converges to $0$ in $L^p(D)$.

Since $\tau_1, \tau_2, ...$ are increase times of a Poisson process, we have
\formula{
 \pr_x(\tau_k \le t < \tau_{k+1}) & = \frac{(M t)^k}{k!} \, e^{-M t} , && t > 0 , \, k = 0, 1, ... ,
}
where $\tau_0 = 0$. In particular, $|\pr_x(t < \tau_1) - 1 + M t| = |e^{-M t} - 1 + M t| \le (M t)^2 / 2$ and $|\pr_x(\tau_1 \le t < \tau_2) - M t| = |M t e^{-M t} - M t| \le (M t)^2$, which shows that the first two summands in~\eqref{eq:cpest} converge to $0$ in $L^p(\R^d)$. Furthermore,
\formula{
 \abs{\ex_x(f(X_t) \ind_{\tau_2 \le t < \tau_D})} & \le \ex_x(|f(X_t)| \ind_{\tau_2 \le t}) \\
 & = \sum_{k = 2}^\infty \ex_x(|f(X_{\tau_k})| \ind_{\tau_k \le t < \tau_{k+1}}) \\
 & = \sum_{k = 2}^\infty \ex_x |f(X_{\tau_k})| \pr_x(\tau_k \le t < \tau_{k+1}) \\
 & = \sum_{k = 2}^\infty \frac{|f| * (\check{\nu}^{*k})(x)}{M^k} \, \frac{(M t)^k}{k!} \, e^{-M t} .
}
Since the total mass of $\check{\nu}^{*k}$ is $M^k$, we conclude that
\formula{
 \norm{\ex_x(|f|(X_t) \ind_{\tau_2 \le t < \tau_D})}_{L^p(D)} & \le \norm{\sum_{k = 2}^\infty \frac{f * (\check{\nu}^{*k})}{M^k} \, \frac{(M t)^k}{k!} \, e^{-M t}}_{L^p(\R^d)} \\
 & \le \sum_{k = 2}^\infty \norm{f}_{L^p(\R^d)} \, \frac{(M t)^k}{k!} \, e^{-M t} \\
 & = \norm{f}_{L^p(\R^d)} (1 - e^{-M t} - M t e^{-M t}) ,
}
which is of order $t^2$ as $t \to 0^+$. Hence, also the third summand in~\eqref{eq:cpest} converges to $0$.
\end{proof}

When $X_t$ is not a compound Poisson process, then the generators of $X_t$ and $X^D_t$ are unbounded operators, and the relation between the domains of $\A$ and $\A_D$ is much less obvious. A point $x \in \partial D$ is said to be a \emph{regular boundary point of $D$} if $\inf \{ t > 0 : X_t \notin D \}  = 0$ a.s. with respect to $\pr_x$ (note that here the inequality $t > 0$ is strict, while $t \ge 0$ is used in the definition of $\tau_D$). Let $C_0(D)$ denote the space of $C_0(\R^d)$ functions vanishing in $\R^d \setminus D$. If every $x \in \partial D$ is a regular boundary point and $X_t$ has the strong Feller property (that is, $P_t$ maps $L^\infty(\R^d)$ to $C_b(\R^d)$), then the operators $P^D_t$ form a strongly continuous contraction semigroup on $C_0(D)$ (that is, $X^D_t$ has the Feller property; see~\cite{bib:c85}). Note that if $X_t$ is a symmetric L{\'e}vy process in $\R$ (and not a compound Poisson process) and $D = \hl$, then $0$, the only boundary point of $D$, is regular (see Theorem~47.5 in~\cite{bib:s99}).

We have the following fundamental result due to Dynkin.

\begin{defin}[Dynkin characteristic operator, \cite{bib:d65}]
Let $x \in \R^d$ and denote by $\tau_{x, \eps}$ the first exit time from the ball centered at $x$ and with radius $\eps$. If the limit
\formula{
 \dynkin_x f & = \lim_{\eps \to 0^+} \frac{\ex_x f(X_{\tau_{x, \eps}}) - f(x)}{\ex_x \tau_{x, \eps}}
}
exists, we write $f \in \domain(\dynkin_x)$. If $f \in \domain(\dynkin_x)$ for all $x \in D$, we write $f \in \domain(\dynkin_D)$, and we define $\dynkin_D f(x) = \dynkin_x f$.
\end{defin}

\begin{lem}[Theorem~5.5 in~\cite{bib:d65} for L{\'e}vy processes]
\label{lem:d}
Suppose that $X_t$ is a L{\'e}vy process in $\R^d$ which is a strong Feller process. Suppose furthermore that $D \sub \R^d$ is open, and every boundary point of $D$ is regular. Let $f \in C_0(D)$. Then $f \in \domain(\A_D; C_0)$ if and only if $f \in \domain(\dynkin_D)$ and $\dynkin_D f \in C_0(D)$. In this case, $\A_D f(x) = \dynkin_D f(x)$ for $x \in D$. \qed
\end{lem}

The above lemma states that, in a sense, the operators $\A_D$ and $\A$ have the same \emph{pointwise} definition, but their domains are different. This can be made even more explicit using the language of distribution theory, as we now describe. Recall that $A$ is the distributional generator, defined in Definition~\ref{def:gen:distr}, and $\check{A}$ is the reflection of $A$.

\begin{lem}
\label{lem:d2}
Suppose that $X_t$ is a L{\'e}vy process in $\R^d$ which is a strong Feller process. Suppose furthermore that $D \sub \R^d$ is open, and every boundary point of $D$ is regular. Let $F \in C_0(D)$. If the distribution $\check{A} \conv F$ restricted to $D$ is equal to a $C_0(D)$ function, then $F \in \domain(\A_D; C_0)$ and $\A_D F(x) = \check{A} \conv F(x)$ for $x \in D$.
\end{lem}

In the proof, we use the \emph{Dynkin's formula} (\cite{bib:d65}, formula~(5.8)): when $\tau$ is a Markov time, $\ex_x \tau < \infty$, and $f \in \domain(\A; C_0)$, then
\formula[eq:dynkin]{
 \ex_x f(X_\tau) - f(x) & = \ex_x \expr{\int_0^\tau \A f(X_s) ds} , && x \in \R^d .
}
We denote by $B(x, \eps)$ and $\overline{B}(x, \eps)$ the open and closed balls centered at $x$ and with radius $\eps$.

\begin{proof}[Proof of Lemma~\ref{lem:d2}]
Fix $x \in D$ and let $r > 0$. Let $g_r \in \schwartz$ be an approximation to identity: $g_r(y) \ge 0$, $\int_{\R^d} g_r(y) dy = 1$ and $g_r(y) = 0$ when $|y| \ge r$. We define $f = F * g_r$. Then $f \in C_b^\infty(\R^d) \sub \domain(\A; C_b)$ (see Lemma~\ref{lem:gen:cb}), and 
\formula[eq:d2:dec]{
 \A f(y) = \check{A} * (F * g_r) = (\check{A} \conv F) * g_r = F \conv (\check{A} \conv g_r) .
}
(The convolution of $\check{A}$, $F$ and $g_r$ is associative because $F$ is a bounded function, $g_r$ has compact support, and $\check{A} * \ph$ is integrable for all $\ph \in \schwartz$.) In particular, $\A f \in C_0(\R^d)$, and therefore $f \in \domain(\A; C_0)$. Let $\tau_{x,\eps}$ be the time of first exit from $B(x, \eps)$. By Dynkin's formula~\eqref{eq:dynkin},
\formula{
 \abs{\frac{\ex_x f(X(\tau_{x,\eps})) - f(x)}{\ex_x \tau_{x,\eps}} - \A f(x)} & = \frac{1}{\ex_x \tau_{x,\eps}} \abs{\ex_x \int_0^{\tau_{x,\eps}} (\A f(X_t) - \A f(x)) dt} \\
 & \le \sup_{y \in B(x, \eps)} |\A f(y) - \A f(x)| .
}
Suppose that $r + \eps < \dist(x, \R^d \setminus D)$. Then using~\eqref{eq:d2:dec} for the right-hand side, we obtain
\formula{
 \abs{\frac{\ex_x f(X(\tau_{x,\eps})) - f(x)}{\ex_x \tau_{x,\eps}} - \A f(x)} & \le \sup_{y \in B(x, r + \eps)} |\check{A} \conv F(y) - \check{A} \conv F(x)| .
}
Consider the limit as $r \to 0^+$. Then $g_r$ approximates the Dirac delta measure $\delta_0$. Hence, $f = F * g_r$ converges uniformly to $F$, and (again by~\eqref{eq:d2:dec}) $\A f(x)$ converges to $\check{A} \conv F(x)$ ($x \in D$ is fixed). It follows that
\formula{
 \abs{\frac{\ex_x F(X(\tau_{x,\eps})) - F(x)}{\ex_x \tau_{x,\eps}} - \check{A} \conv F(x)} & \le \sup_{y \in \overline{B}(x, \eps)} |\check{A} \conv F(y) - \check{A} \conv F(x)| .
}
As $\eps \to 0^+$, the right-hand side tends to zero. Therefore, $F \in \domain(\dynkin_x)$ and $\dynkin_x F = \check{A} \conv F(x)$. Since $x \in D$ was arbitrary, we conclude that $F \in \domain(\dynkin_D)$, and $\dynkin_D F(x) = \check{A} \conv F(x)$ for $x \in D$. The result follows from Lemma~\ref{lem:d}.
\end{proof}

Lemma~\ref{lem:d2} seems to be now. However, a related distributional approach to generators of the free and killed processes was used for example in~\cite{bib:bb99}, in the study of harmonic functions with respect to symmetric $\alpha$-stable processes.

We need the following auxiliary result. Its $L^2$ analogue is relatively easy to prove using Dirichlet forms. The uniform version seems to be new, especially in the case of unbounded $D$ and $F$ not vanishing at infinity.

\begin{lem}
\label{lem:smooth:domain}
Suppose that $r > 0$, $D \sub \R^d$ is an open set, $F \in C_b^\infty(\R^d)$, and $F(x) = 0$ whenever $\dist(x, \R^d \setminus D) \le r$. Suppose furthermore, that $\A F(x) = 0$ for all $x \in \partial D$. Then $F \in \domain(\A_D; C_b)$ and $\A_D F(x) = \A F(x)$ for $x \in D$.
\end{lem}

Recall that $B(x, \eps)$ is the open ball centered at $x$ and with radius $\eps$. In the proof, the following well-known result is used. For completeness, we give a simple proof.

\begin{prop}
\label{prop:ball}
For all $\eps > 0$, $\pr_0(t \ge \tau_{B(0, \eps)}) / t$ is bounded in $t > 0$.
\end{prop}

\begin{proof}
Let $h \in \schwartz$ satisfy $h(0) = 1$, $h(x) = 0$ for $x \notin B(0, \eps)$, and $h(x) \le 1$ for all $x$. Define $\tau = \min(t, \tau_{B(0, \eps)})$. Then, by Dynkin's formula~\eqref{eq:dynkin},
\formula{
 \pr_0(t \ge \tau_{B(0, \eps)}) & = 1 - \pr_0(t < \tau_{B(0, \eps)}) \le h(0) - \ex_0 h(X_\tau) \\
 & = - \ex_x \expr{\int_0^\tau \A h(X_s) ds} \le t \|\A h\|_{L^\infty(\R^d)} \, . \qedhere
}
\end{proof}

\begin{proof}[Proof of Lemma~\ref{lem:smooth:domain}]
Since there are few tools to handle functions non-vanishing at infinity, we use the definition of $\A_D$. For $\eps \in (0, r]$, denote by $D_\eps$ the set of $x$ such that $\dist(x, \R^d \setminus D) \ge \eps$, and $D_\eps' = \R^d \setminus D_\eps$. Our strategy is to show that away from the boundary of $D$ (in $D_\eps$) we have $P^D_t F \approx P_t f$, while near the boundary (in $D \setminus D_\eps$), $P^D_t F \approx 0$.

Let $g_\eps(t) = \pr_x(t \ge \tau_{B(x, \eps)})$. Note that $g_\eps(t)$ does not depend on $x \in \R^d$, and that by Proposition~\ref{prop:ball}, $g_\eps(t) / t$ is a bounded function of $t > 0$.

First, we consider $x \in D_\eps$, where $\eps \in (0, r]$ will be specified later. By Lemma~\ref{lem:gen:cb}, $F \in \domain(\A; C_b)$. Since $F$ vanishes on $D_r'$, for $x \in D_\eps$ we have
\formula{
 \abs{P_t F(x) - P_t^D F(x)} & = \abs{\ex_x(F(X_t) \ind_{t \ge \tau_D})} \\
 & \le \|F\|_{L^\infty(\R)} \ex_x(\ind_{X_t \in D_r} \ind_{t \ge \tau_D}) .
}
Observe that the condition $X_t \in D_r$ and $t \ge \tau_D$ implies that before time $t$, the process \emph{first} exits $D$, and \emph{after that} it exits $D_r'$. Furthermore, $B(X_{\tau_D}, \eps) \sub D_r'$ a.s. Hence, by the strong Markov property,
\formula{
 \ex_x(\ind_{X_t \in D_r} \ind_{t \ge \tau_D}) & \le \pr_x(\text{$X_s$ first exists $D$ and then $B(X_{\tau_D}, \eps)$ before time $t$}) \\
 & = \ex_x(g_\eps(t - \tau_D) \ind_{t \ge \tau_D}) .
}
Since $g_\eps$ is increasing, and $B(x, \eps) \sub D$, we have
\formula{
 \ex_x(\ind_{X_t \in D_r} \ind_{t \ge \tau_D}) & \le g_\eps(t) \pr_x(t \ge \tau_D) \\
 & \le g_\eps(t) \pr_x(t \ge \tau_{B(x, \eps)}) = (g_\eps(t))^2 .
}
We conclude that for $x \in D_\eps$,
\formula[eq:smooth:domain:1]{
\begin{aligned}
 \abs{\frac{P_t^D F(x) - F(x)}{t} - \A F(x)} & \le \norm{\frac{P_t F - F}{t} - \A F}_{L^\infty(\R^d)} \\ & + \frac{\|F\|_{L^\infty(\R)} (g_\eps(t))^2}{t} \, .
\end{aligned}
}

Now we consider $x$ close to the boundary. Fix (small) $\eta > 0$. By Corollary~\ref{cor:gen:cb}, $\A F$ is Lipschitz continuous. Hence, we can choose $\eps \in (0, r/2]$ small enough, so that $|\A F(y)| \le \eta$ when $\dist(y, D \setminus D_\eps) < \eps$. Let $\ph_n \in \schwartz$ be the sequence approximating $F$ as in the proof of Lemma~\ref{lem:gen:cb}. Let $x \in D \setminus D_\eps$ and $\tau = \min(t, \tau_D)$. By Dynkin's formula~\eqref{eq:dynkin},
\formula{
 \abs{\ex_x \ph_n(X_\tau) - \ph_n(x)} & = \abs{\ex_x \expr{\int_0^\tau \A \ph_n(X_s) ds}} \le \ex_x \expr{\int_0^t |\A \ph_n(X_s)| ds} .
}
When $t < \tau_{B(x, \eps)}$, we use the estimate
\formula{
 \int_0^t |\A \ph_n(X_s)| ds & \le t \sup_{y \in B(x, \eps)} |\A \ph_n(y)| .
}
When $t \ge \tau_{B(x, \eps)}$, we simply have
\formula{
 \int_0^t |\A \ph_n(X_s)| ds & \le t \|\A \ph_n\|_{L^\infty(\R^d)} .
}
It follows that
\formula{
 \abs{\ex_x \ph_n(X_\tau) - \ph_n(x)} & \le t \pr_x(t < \tau_{B(x, \eps)}) \sup_{y \in B(x, \eps)} |\A \ph_n(y)| \\
 & \hspace*{6em} + t \pr_x(t \ge \tau_{B(x, \eps)}) \|\A \ph_n\|_{L^\infty(\R^d)} \\
 & \le t \expr{\sup_{y \in B(x, \eps)} |\A \ph_n(y)| + g_\eps(t) \|\A \ph_n\|_{L^\infty(\R^d)}} .
}
Recall that $\ph_n$ and $\A \ph_n$ converge to $F$ and $\A F$, respectively, locally uniformly, and that $\|\A \ph_n\|_{C_b^2(\R^d)} \le C \|\ph_n\|_{C_b^4(\R^d)} \le C^2 \|F\|_{C_b^4(\R^d)}$ for some $C > 0$. By taking the limit as $n \to \infty$ and applying dominated convergence, we obtain that for $x \in D \setminus D_\eps$,
\formula{
 \abs{\ex_x F(X_\tau) - F(x)} & \le t \expr{\sup_{y \in B(x, \eps)} |\A F(y)| + C^2 g_\eps(t) \|F\|_{C_b^4(\R^d)}} .
}
We have $|\A F(y)| \le \eta$ for $y \in B(x, \eps)$ (because $\dist(y, D \setminus D_\eps) < \eps$). Hence,
\formula{
 \abs{\frac{\ex_x F(X_\tau) - F(x)}{t} - \A F(x)} & \le \frac{|\ex_x F(X_\tau) - F(x)|}{t} + |\A F(x)| \\
 & \le 2 \eta + C^2 g_\eps(t) \|F\|_{C_b^4(\R^d)} .
}
Finally, $F(X_{\tau_D}) = 0$ a.s., and therefore
\formula{
 \ex_x F(X_\tau) & = \ex_x(F(X_t) \ind_{t < \tau_D}) = P^D_t F(x) .
}
We conclude that
\formula[eq:smooth:domain:2]{
 \abs{\frac{P^D_t F(x) - F(x)}{t} - \A F(x)} & \le 2 \eta + C g_\eps(t) \|F\|_{C_b^4(\R^d)} .
}
Recall that as $t \to 0^+$, $g_\eps(t) / t$ is bounded. Furthermore, $(P_t F - F) / t$ converges uniformly to $\A F$. Hence, \eqref{eq:smooth:domain:1} and~\eqref{eq:smooth:domain:2} imply that for $t$ small enough,
\formula{
 \norm{\frac{P^D_t F - F}{t} - \ind_D \A F}_{L^\infty(D)} & \le 3 \eta .
}
Since $\eta > 0$ was arbitrary, the proof is complete.
\end{proof}

%
%

\subsection{Subordinate Brownian motions}

From now on, we only consider one-dimensional L{\'e}vy processes. The process $X_t$ corresponding to $\beta = 1$, $\gamma = 0$ and $\nu$ vanishing, is the Brownian motion, running at twice the usual speed (that is, $\var X_t = 2 t$). We denote the transition density of the Brownian motion by $k_s(z)$,
\formula[eq:ks]{
 k_s(z) & = \frac{1}{\sqrt{4 \pi s}} \, \exp \expr{-\frac{z^2}{4 s}} \, , && s > 0 , \, z \in \R .
}
A L{\'e}vy process $X_t$ is said to be a \emph{subordinate Brownian motion} if it can be written in the form $X_t = B_{Z_t}$, where $B_s$ is the Brownian motion, $Z_t$ is a \emph{subordinator} (a non-decreasing L{\'e}vy process), and $B_s$ and $Z_t$ are independent processes. We assume that under $\pr_x$, $B_s$ starts at $x$ and $Z_t$ starts at $0$ a.s. To avoid trivialities, we assume that $X_t$ is non-constant, that is, $\beta > 0$ or $\nu$ is non-zero.

The subordinator $Z_t$ is completely described by its \emph{Laplace exponent} $\psi(\xi)$: we have $\ex_0 e^{-\xi Z_t} = e^{-t \psi(\xi)}$ ($\xi \in \C$, $\real \xi > 0$), where
\formula[eq:psi]{
 \psi(\xi) & = \beta \xi + \int_0^\infty (1 - e^{-\xi s}) \mu(ds) , && \real \xi > 0 .
}
Here $\beta \ge 0$ is the same as in~\eqref{eq:lk}, and $\mu$ is the L{\'e}vy measure of $Z_t$: a Radon measure on $(0, \infty)$ satisfying $\int_0^\infty \min(1, s) \mu(ds) < \infty$. The relation between the L{\'e}vy measures $\mu$ and $\nu$ of $Z_t$ and $X_t$ is
\formula[eq:nu]{
 \nu(z) = \int_0^\infty k_s(z) \mu(ds) , && z \in \R \setminus \set{0} ,
}
with $k_s(z)$ defined in~\eqref{eq:ks}. The L{\'e}vy-Khintchine exponent of the subordinate Brownian motion $X_t$ satisfies $\Psi(\xi) = \psi(\xi^2)$.

The class of subordinate Brownian motions is naturally divided into compound Poisson processes (with bounded Laplace exponents $\psi$) and processes corresponding to unbounded $\psi$. The latter class consists of strong Feller processes: if $Z_t > 0$ a.s., then $X_t = B_{S_t}$ has absolutely continuous distribution. This property is important for applications of Lemma~\ref{lem:d2}. See~\cite{bib:bbkrsv09, bib:ksv11, bib:ssv10} for more information on subordinate Brownian motions.

\begin{rem}
One needs to distinguish the \emph{killed subordinate} and the \emph{subordinate killed} processes. The first one is $X_t^D$ introduced above, while the other one is obtained by subordination of the killed Brownian motion. We emphasize that these two processes are essentially different. In particular, the spectral theory for the subordinate killed process is trivial: its eigenfunctions are the same as the eigenfunctions of the (insubordinate) killed Brownian motion, and only the corresponding eigenvalues are different. For a discussion of the relation between killed subordinate and subordinate killed processes, see e.g.~\cite{bib:sv08}.\qed
\end{rem}

We recall some standard definitions; see~\cite{bib:ssv10} for further information. A function $f : (0, \infty) \to \R$ is said to be \emph{completely monotone} on $(0, \infty)$ if $f^{(n)}(z) \ge 0$ for all $z > 0$ and $n = 0, 1, 2, ...$. A function $f : (0, \infty) \to \R$ is a \emph{Bernstein function} if $f$ is nonnegative and $f'$ is completely monotone. By Bernstein's theorem, every completely monotone function is the Laplace transform of some Radon measure on $[0, \infty)$, and every Bernstein function has the form
\formula[eq:bf]{
 f(z) & = c_1 + c_2 z + \int_0^\infty (1 - e^{-z s}) m(ds) , && z > 0 ,
}
for some $c_1, c_2 \ge 0$ and a Radon measure $m$ on $(0, \infty)$ satisfying $\int_0^\infty \min(1, s) m(ds) < \infty$. Whenever we write $f(0)$ when $f$ is a Bernstein function, we mean $\lim_{z \to 0^+} f(z)$. Note that if $f(0) = 0$, then $f$ has the form~\eqref{eq:psi}; hence, a Bernstein function $f$ is the Laplace exponent of a subordinator if and only if $f(0) = 0$.

%
%

\subsection{Complete Bernstein functions and Stieltjes functions}

A Bernstein function $f$ for which the measure $m$ in the representation~\eqref{eq:bf} has a completely monotone density function is said to be a \emph{complete Bernstein function} (CBF in short; also known as \emph{operator monotone function}). Recall that $X_t$ is a subordinate Brownian motion, $\mu$ is the L{\'e}vy measure of the underlying subordinator, and $\psi$ is its Laplace exponent (so that the L{\'e}vy-Khintchine exponent of $X_t$ is $\psi(\xi^2)$).

\begin{prop}
\label{prop:cbf}
The following conditions are equivalent:
\begin{enumerate}
\item[(a)] $\psi$ is a complete Bernstein function;
\item[(b)] $\mu(ds)$ has a completely monotone density function $\mu(s)$;
\item[(c)] for a Radon measure $\mu_0$ on $(0, \infty)$ such that the integral $\int_0^\infty \min(\zeta^{-1}, \zeta^{-2}) \mu_0(d\zeta)$ is finite,
\formula[eq:cbf]{
 \psi(\xi) & = \beta \xi + \frac{1}{\pi} \int_0^\infty \frac{\xi}{\xi + \zeta} \, \frac{\mu_0(d\zeta)}{\zeta} \, ;
}
\item[(d)] $\psi$ extends to a holomorphic function in $\C \setminus (-\infty, 0]$, which leaves the upper and the lower complex half-planes invariant;
\item[(e)] $\nu(z)$ is a completely monotone function on $(0, \infty)$.
\end{enumerate}
A symmetric L{\'e}vy process satisfying~(e) is automatically a subordinate Brownian motion.
\end{prop}

\begin{proof}
Equivalence of~(a) and~(b) is just the definition of complete Bernstein functions. Equivalence of~(b)--(d) is standard, see~\cite{bib:ssv10}, Chapters~6 and~11. Equivalence of~(c) and~(e) seems to be well-known to specialists (cf.~\cite{bib:r83}), but a direct reference to the literature is difficult to find. The proof is a rather straightforward calculation involving Fourier transform and an application of Bernstein's representation theorem; for completeness, we provide the details.

Suppose that~(c) holds. Then, for $\xi \in \R$,
\formula{
 \Psi(\xi) = \psi(\xi^2) & = \beta \xi^2 + \frac{1}{\pi} \int_0^\infty \frac{\xi^2}{\xi^2 + \zeta} \, \frac{\mu_0(d\zeta)}{\zeta} .
}
By an explicit calculation, for $\zeta > 0$,
\formula{
 \int_{-\infty}^\infty (1 - \cos(\xi z)) e^{-\sqrt{\zeta} |z|} dz & = \frac{2}{\sqrt{\zeta}} - \frac{2 \sqrt{\zeta}}{\xi^2 + \zeta} = \frac{2}{\sqrt{\zeta}} \, \frac{\xi^2}{\xi^2 + \zeta} \, .
}
Hence, by Fubini,
\formula{
 \Psi(\xi) & = \beta \xi^2 + \frac{1}{2 \pi} \int_0^\infty \expr{\int_{-\infty}^\infty (1 - \cos(\xi z)) e^{-\sqrt{\zeta} |z|} dz} \frac{\mu_0(d\zeta)}{\sqrt{\zeta}} \\
 & = \beta \xi^2 + \int_{-\infty}^\infty (1 - \cos(\xi z)) \expr{\frac{1}{2 \pi} \int_0^\infty e^{-\sqrt{\zeta} |z|} \, \frac{\mu_0(d\zeta)}{\sqrt{\zeta}}} dz .
}
Comparing this with~\eqref{eq:lk}, we obtain
\formula[eq:nu:cbf]{
 \nu(z) & = \frac{1}{2 \pi} \int_0^\infty e^{-\sqrt{\zeta} |z|} \, \frac{\mu_0(d\zeta)}{\sqrt{\zeta}} ,
}
which proves that $\nu(z)$ is completely monotone on $(0, \infty)$. Conversely, if~(e) holds, that is, if $\nu(z)$ is completely monotone on $(0, \infty)$, then, by Bernstein's theorem, $\nu(z)$ has the representation~\eqref{eq:nu:cbf}, and the above reasoning can be reversed to prove that $X_t$ is a subordinate Brownian motion, and that~(c) holds.
\end{proof}

The assumptions of the main theorems can be put in the following form.

\begin{assumption}
\label{asmp}
The process $X_t$ is a subordinate Brownian motion satisfying any of the equivalent conditions of Proposition~\ref{prop:cbf}.
\end{assumption}

We note the following immediate consequence of Proposition~\ref{prop:cbf}.

\begin{cor}
\label{corr:cbf}
The following conditions are equivalent:
\begin{enumerate}
\item[(a)] $f$ is a complete Bernstein function; that is, \eqref{eq:bf} holds for some $c_1, c_2 \ge 0$ and a measure $m$ with completely monotone density;
\item[(b)] for some $c_1, c_2 \ge 0$ and a Radon measure $m_0$ on $(0, \infty)$ such that $\int_0^\infty \min(s^{-1}, s^{-2}) m_0(ds) < \infty$,
\formula[eq:cbf0]{
 f(z) & = c_1 + c_2 z + \frac{1}{\pi} \int_0^\infty \frac{z}{z + s} \, \frac{m_0(ds)}{s} \, , && z > 0;
}
\item[(c)] $f(z) \ge 0$ for $z > 0$ and either $f$ is constant, or $f$ extends to a holomorphic function in $\C \setminus (-\infty, 0]$, which leaves the upper and the lower complex half-planes invariant;
\item[(d)] $f(\xi^2) - f(0)$ is the L{\'e}vy-Khintchine exponent of a symmetric L{\'e}vy process, whose L{\'e}vy measure has a completely monotone density function on $(0, \infty)$.\qed
\end{enumerate}
\end{cor}

The right-hand side of formula~\eqref{eq:cbf0} defines a holomorphic function in $\C \setminus (-\infty, 0]$. We often denote this extension by the same symbol $f$. We will also use the related notion of a Stieltjes function: $f$ is said to be a \emph{Stieltjes function} if there are $\tilde{c}_1, \tilde{c}_2 \ge 0$ and a Radon measure $\tilde{m}_0$ on $(0, \infty)$ satisfying $\int_0^\infty \min(1, s^{-1}) \tilde{m}_0(ds) < \infty$ such that
\formula[eq:stieltjes]{
 f(z) & = \tilde{c}_1 + \frac{\tilde{c}_2}{z} + \frac{1}{\pi} \int_0^\infty \frac{1}{z + s} \, \tilde{m}_0(ds) , && z > 0 .
}
Since $1 / (z + s) = \int_0^\infty e^{-\xi s} e^{-\xi z} d\xi$, we see that $f$ has the representation~\eqref{eq:stieltjes} if and only if it is the Laplace transform of the measure
\formula{
 \tilde{c}_1 \delta_0(d\xi) + (\tilde{c}_2 + \pi^{-1} \laplace \tilde{m}_0(\xi)) d\xi
}
on $[0, \infty)$, where $\delta_0$ is the Dirac delta measure. In other words, $f$ is the Laplace transform of a completely monotone function, plus a nonnegative constant.

We list some standard properties of complete Bernstein and Stieltjes functions.

\begin{prop}[Chapter~3 in~\cite{bib:j01} and Chapter~7 in~\cite{bib:ssv10}]
\label{prop:cbf:prop}
Suppose that $f$ and $g$ are not constantly equal to $0$.
\begin{enumerate}
\item[(a)] The following conditions are equivalent: $f(z)$ is a CBF; $z / f(z)$ is a CBF; $1 / f(z)$ is a Stieltjes function; $f(z) / z$ is a Stieltjes function.
\item[(b)] If $f$, $g$ are CBF, $a, C > 0$ and $\alpha \in (0, 1)$, then $f(z) + g(z)$, $C f(z)$, $f(C z)$, $(z - a) / (f(z) - f(a))$ (extended continuously at $z = a$), $f(g(z))$ and $(f(z^\alpha))^{1/\alpha}$ are CBF.
\qed
\end{enumerate}
\end{prop}

One of the fundamental properties of complete Bernstein and Stieltjes functions is that the characteristics in representations~\eqref{eq:cbf0} and~\eqref{eq:stieltjes} can be easily recovered from (the holomorphic extension of) the function. We illustrate this for Stieltjes functions. By $\delta_0$ we denote the Dirac delta measure at $0$.

\begin{prop}
\label{prop:cbf:repr}
Let $f$ be (the holomorphic extension of) a Stieltjes function with representation~\eqref{eq:stieltjes}. Then
\formula[eq:stieltjes:constants]{
 \tilde{c}_1 & = \lim_{z \to \infty} f(z) , & \tilde{c}_2 & = \lim_{z \to 0^+} (z f(z)) ,
}
and
\formula[eq:jump]{
 \tilde{m}_0(ds) + \pi \tilde{c}_2 \delta_0(ds) & = \lim_{\eps \to 0^+} (-\imag f(-s + i \eps) ds) ,
}
with the limit understood in the sense of weak convergence of measures.
\end{prop}

Formula~\eqref{eq:jump} is well-known in complex analysis, it is a variant of the Sokhotskyi-Plemelj formula for the Cauchy-Stieltjes transform in the upper complex half-plane. The proof of a version for complete Bernstein functions is contained in Theorem~6.2 in~\cite{bib:ssv10} (see Corollary~6.3 and Remark~6.4 therein). We provide a simple argument based on the representation theorem for harmonic functions in half-plane.

\begin{proof}[Proof of Proposition~\ref{prop:cbf:repr}]
Since $-\imag f$ is a nonnegative harmonic function in the upper complex half-plane, it can be represented using the Poisson kernel,
\formula[eq:poisson:harm]{
 -\imag f(x + i y) & = C y + \frac{1}{\pi} \int_{-\infty}^\infty \frac{y}{y^2 + (x - s)^2} \, \tilde{m}(ds)
}
for unique costant $C = \lim_{y \to \infty} (-\imag f(i y) / y)$ and unique measure $\tilde{m}(ds) = \lim_{\eps \to 0^+} (-\imag f(s + i \eps) ds)$. On the other hand, by~\eqref{eq:stieltjes},
\formula{
 -\imag f(x + i y) & = \frac{\tilde{c}_2 y}{y^2 + x^2} + \frac{1}{\pi} \int_0^\infty \frac{y}{y^2 + (x + s)^2} \, \tilde{m}_0(ds) .
}
Formula~\eqref{eq:jump} follows simply by comparing the above two representations. The expressions~\eqref{eq:stieltjes:constants} are direct consequences of~\eqref{eq:stieltjes} and dominated convergence.
\end{proof}

We need three more properties of complete Bernstein functions.

\begin{prop}
\label{prop:cbf:ratio}
If $f$ is a complete Bernstein function and $a > 0$, then $g(z) = (1 - z/a) / (1 - f(z) / f(a))$ and $h(z) = \log g(z)$ are complete Bernstein functions.
\end{prop}

Here we extend $g$ and $h$ continuously at $z = a$, $g(a) = f(a) / (a f'(a))$ and $h(a) = \log (f(a) / (a f'(a)))$.

\begin{proof}
By Proposition~\ref{prop:cbf:prop}(b), $g(z) = (f(a) / a) (x - a) / (f(z) - f(a))$ is a CBF. Note that $g(0) = 1$, and so also $g(z) - 1$ is a CBF. Furthermore, $\log(1 + z)$ is a CBF. Again by Proposition~\ref{prop:cbf:prop}(b), $h(z) = \log(1 + (g(z) - 1))$ is a CBF.
\end{proof}

\begin{lem}
\label{lem:g}
Let $f$ be (the holomorphic extension of) a complete Bernstein function, and let $a > 0$. Define
\formula{
 f^*(z) & = \frac{1}{2 i a} \expr{\frac{f(i a)}{z - i a} - \frac{f(-i a)}{z + i a}} , && z \in \C \setminus \{-i a, i a\} .
}
Then $g(z) = f^*(z) - f(z) / (z^2 + a^2)$ is a Stieltjes function, and it is the Laplace transform of a completely monotone function $G$ on $(0, \infty)$. Furthermore, $G$ is the Laplace transform of a finite measure on $(0, \infty)$.
\end{lem}

\begin{proof}
Let $f$ have the representation~\eqref{eq:cbf0}, and let $h(z) = f(z) - c_1 - c_2 z$. Fix $z \in \C \setminus (-\infty, 0]$. Since
\formula{
 \frac{1}{z + s} \, \frac{1}{a^2 + s^2} & = \biggl(\frac{1}{2 i a (z - i a)} \, \frac{i a}{s + i a} \\ & \qquad + \frac{1}{2 (-i a) (z + i a)} \, \frac{-i a}{s - i a} - \frac{1}{a^2 + z^2} \, \frac{z}{z + s}\biggr) \frac{1}{s} \, ,
}
we have
\formula{
 \frac{1}{\pi} \int_0^\infty \frac{1}{z + s} \, \frac{m_0(ds)}{a^2 + s^2} & = \frac{1}{2 i a} \expr{\frac{h(i a)}{z - i a} - \frac{h(-i a)}{z + i a}} - \frac{h(z)}{a^2 + z^2} \, .
}
Furthermore, 
\formula{
 \frac{1}{2 i a} \expr{\frac{c_1 + c_2 i a}{z - i a} - \frac{c_1 - c_2 i a}{z + i a}} - \frac{c_1 + c_2 z}{a^2 + z^2} & = 0 .
}
We conclude that
\formula{
 \frac{1}{\pi} \int_0^\infty \frac{1}{z + s} \, \frac{m_0(ds)}{a^2 + s^2} & = \frac{1}{2 i a} \expr{\frac{f(i a)}{z - i a} - \frac{f(-i a)}{z + i a}} - \frac{f(z)}{a^2 + z^2} = g(z) .
}
By the definition, $g$ is the Stieltjes function, and by the remark following~\eqref{eq:stieltjes}, $g$ is the Laplace transform of a completely monotone function $G(x) = \laplace \gamma(x)$, where $\gamma(ds) = (1 / \pi) (a^2 + s^2)^{-1} m_0(ds)$ is a finite measure on $(0, \infty)$.
\end{proof}

\begin{prop}
\label{prop:cbf:ests}
If $f$ is a complete Bernstein function, then
\begin{enumerate}
\item[(a)] $0 \le z f'(z) \le f(z)$ for $z > 0$;
\item[(b)] $0 \le -z f''(z) \le 2 f'(z)$ for $z > 0$;
\item[(c)] $|f(z)| \le (\sin(\eps/2))^{-1} \, f(|z|)$ for $z \in \C$, $|\Arg z| \le \pi - \eps$, $\eps \in (0, \pi)$;
\item[(d)] $|z f'(z)| \le (\sin(\eps/2))^{-1} \, f(|z|)$ for $z$ as in~(c);
\item[(e)] $|f(z)| \le C(f, \eps) (1 + |z|)$ for $z$ as in~(c).
\end{enumerate}
In~(c)--(e), the holomorphic extension of $f$ is denoted by the same symbol.
\end{prop}

\begin{proof}
All statements follow easily from the representation~\eqref{eq:cbf0} and differentiation under the integral sign. Indeed, (a) reduces to
\formula{
 0 & \le c_2 z + \frac{1}{\pi} \int_0^\infty \frac{s z}{(s + z)^2} \, \frac{m_0(ds)}{s} \le c_1 + c_2 z + \frac{1}{\pi} \int_0^\infty \frac{z}{s + z} \, \frac{m_0(ds)}{s} \, ,
}
which follows from the inequality $s z / (s + z)^2 \le z / (s + z)$ for $s, z > 0$. In a similar manner, (b) follows from $2 s z / (s + z)^3 \le 2 s / (s + z)^2$. To prove (c)--(e), observe that for $z \in \C$ with $\Arg z \in [-\pi + \eps, \pi - \eps]$,
\formula{
 |s + z|^2 & = s^2 + 2 s \real z + |z|^2 \ge s^2 - 2 s |z| \cos(\eps) + |z|^2 .
}
Hence, by a simple calculation,
\formula{
 |s + z|^2 & \ge \frac{1 - \cos(\eps)}{2} \, (s + |z|)^2 = (\sin(\eps/2))^2 (s + |z|)^2 .
}
Therefore, (c) is a consequence of $|z / (s + z)| \le (\sin(\eps/2))^{-1} \, |z| / (s + |z|)$, and~(d) is proved using $|s z / (s + z)^2| \le |z| / |s + z| \le (\sin(\eps/2))^{-1} |z| / (s + |z|)$. Finally, to prove~(e), use~(c) and the inequality $|z| / (s + |z|) \le (1 + |z|) / (s + 1)$, verified easily by a direct calculation.
\end{proof}

For further properties of complete Bernstein functions and related notions, see Chapters~6 and~7 in~\cite{bib:ssv10}, or Chapter~3 in~\cite{bib:j01}.

%
%

\section{Transformation of the Laplace exponent}
\label{sec:cbf}

%
%

Recall that $X_t = B_{Z_t}$ is a subordinate Brownian motion, $\psi(\xi)$ is the Laplace exponent of the subordinator $Z_t$, and $\Psi(\xi) = \psi(\xi^2)$ is the L{\'e}vy-Khintchine exponent of $X_t$. Assumption~\ref{asmp} is in force; that is, $\psi$ is assumed to be a \emph{complete} Bernstein function. In this section we introduce the operation $\psi \mapsto \psi^\dagger$. Although it will be used for complete Bernstein functions only, the definition is quite general.

\begin{defin}
\label{def:dagger}
When $\psi$ is a positive function on $(0, \infty)$, and the integral $\int_0^\infty \min(1, \xi^{-2}) \log \psi(\xi) d\xi$ is finite, then we define
\formula[eq:dagger]{
 \psi^\dagger(\xi) & = \exp\expr{\frac{1}{\pi} \, \int_0^\infty \frac{\log \psi(\xi^2 \zeta^2)}{1 + \zeta^2} \, d\zeta} , && \xi > 0 .
}
\end{defin}

\begin{prop}
For $\psi$ as Definition~\ref{def:dagger},
\formula[eq:dagger1]{
 \psi^\dagger(\xi) & = \exp\expr{\frac{1}{\pi} \, \int_0^\infty \frac{\xi \log \psi(\zeta^2)}{\xi^2 + \zeta^2} \, d\zeta} , && \xi > 0 .
}
\end{prop}

\begin{proof}
Substitute $\xi \zeta = s$ in~\eqref{eq:dagger}.
\end{proof}

Note that~\eqref{eq:dagger1} defines a holomorphic extension of $\psi^\dagger$ to the half-plane $\real \xi > 0$. We denote this extension by the same symbol $\psi^\dagger(\xi)$.

\begin{rem}
The definition of $\psi^\dagger$ lies at the very heart of the Wiener-Hopf method. Roughly speaking, if $\psi^\dagger(\xi)$ can be extended continuously to the boundary of the region $\real \xi > 0$, that is, to the imaginary axis, then we have $\psi(\xi^2) = \psi^\dagger(i \xi) \psi^\dagger(-i \xi)$, a Wiener-Hopf factorization of $\psi(\xi^2)$. For complete Bernstein functions, this is formally proved in Lemma~\ref{lem:cbfextension} below. The Wiener-Hopf method is described in Section~\ref{sec:eigenfunctions}.

The function $\psi^\dagger$ plays an important role in fluctuation theory of L{\'e}vy processes. When $X_t$ is a symmetric L{\'e}vy process, then its L{\'e}vy-Khintchine exponent $\Psi$ is symmetric, $\Psi(\xi) = \psi(\xi^2)$ for some nonnegative $\psi$ on $[0, \infty)$. Suppose that $X_t$ is not a compound Poisson process. By Corollary~9.7 in~\cite{bib:f74}, the Laplace exponent $\kappa(z, \xi)$ of the bivariate ascending ladder process satisfies
\formula{
 \kappa(z, \xi) & = (z + \psi)^\dagger(\xi) , && z \ge 0, \, \xi > 0 .
}
This relation was used extensively in~\cite{bib:ksv10, bib:ksv11, bib:kmr11}.\qed
\end{rem}

The fundamental Lemma~\ref{lem:cbfextension} states that if $\psi$ is a CBF, then $\psi^\dagger$ is a CBF. Before we prove it, we establish some simple properties of $\psi^\dagger$.

\begin{prop}
\label{prop:dagger:prop}
Whenever both sides of the following identities make sense, we have:
\begin{enumerate}
\item[(a)] $(1 / \psi)^\dagger = 1 / \psi^\dagger$, $(\psi^\alpha)^\dagger = (\psi^\dagger)^\alpha$ ($\alpha \in \R$) and $(\psi_1 \psi_2)^\dagger = \psi_1^\dagger \psi_2^\dagger$;
\item[(b)] $(C^2 \psi)^\dagger = C \psi^\dagger$ for $C > 0$;
\item[(c)] when $\psi(\xi) = \xi$, then $\psi^\dagger(\xi) = \xi$;
\item[(d)] if appropriate limits of $\psi$ exist, then the corresponding limits of $\psi^\dagger$ exist, and $\lim_{\xi \to 0^+} \psi^\dagger(\xi) = (\lim_{\xi \to 0^+} \psi(\xi))^{1/2}$, $\lim_{\xi \to \infty} \psi^\dagger(\xi) = (\lim_{\xi \to \infty} \psi(\xi))^{1/2}$.
\end{enumerate}
\end{prop}

\begin{proof}
Properties~(a) follow immediately from formula~\eqref{eq:dagger1}. In~(b), one uses also $\int_0^\infty \xi / (\xi^2 + \zeta^2) d\zeta = \pi/2$. For $\psi(\xi) = \xi$, by~\eqref{eq:dagger} we have
\formula{
 \psi^\dagger(\xi) & = \exp\expr{\frac{1}{\pi} \int_0^\infty \frac{2 \log \xi + 2 \log \zeta}{1 + \zeta^2} \, d\zeta} \\
 & = \exp\expr{\log \xi + \frac{2}{\pi} \int_0^\infty \frac{\log \zeta}{1 + \zeta^2} \, d\zeta} = \xi ;
}
here we used the identity $\int_0^\infty (1 + s^2)^{-1} \log s \, ds = 0$, proved easily by a substitution $r = 1/s$. Finally, to show~(d), assume that a finite, positive limit $\psi(0) = \lim_{\xi \to 0^+} \psi(\xi)$ exists. It can be proved that the functions $\log \psi(\xi^2 \zeta^2) / (1 + \zeta^2)$ of $\zeta$ are uniformly integrable as $\xi \to 0^+$ (this is clear when $\psi$ is a CBF; we omit the details in the general case). It follows that
\formula{
 \lim_{\xi \to 0^+} \int_0^\infty \frac{\log \psi(\xi^2 \zeta^2)}{1 + \zeta^2} \, d\zeta = \int_0^\infty \frac{\log \psi(0)}{1 + \zeta^2} \, d\zeta = \frac{\pi \log \psi(0)}{2} \, ,
}
and hence $\lim_{\xi \to 0^+} \psi^\dagger(\xi) = \sqrt{\psi(0)}$. When the limit of $\psi$ is $0$ or $\infty$, one uses in a similar way Fatou's lemma; we omit the details. Limits at $\infty$ are dealt with in a similar way.
\end{proof}

\begin{prop}
\label{prop:explogbound}
Let $\psi$ be as in Definition~\ref{def:dagger}. If $C_1, C_2 > 0$, $\alpha \in \R$ and $\psi(\xi) \le  (C_1^2 + C_2^2 \xi)^\alpha$ for all $\xi > 0$, then, with the same constants, $|\psi^\dagger(\xi)| \le |C_1 + C_2 \xi|^\alpha$, $\real \xi > 0$. In a similar manner, if $\psi(\xi) \ge (C_1^2 + C_2^2 \xi)^\alpha$ for $\xi > 0$ for some $C_1, C_2 > 0$, $\alpha \in \R$, then $|\psi^\dagger(\xi)| \ge |C_1 + C_2 \xi|^\alpha$, $\real \xi > 0$.
\end{prop}

\begin{proof}
For $\xi = t + i s$ ($t > 0$, $s \in \R$), we have
\formula{
 |\psi^\dagger(\xi)|^2 & = \exp \expr{\frac{1}{\pi} \int_0^\infty \real \frac{2 \xi}{\xi^2 + \zeta^2} \, \log \psi(\zeta^2) d\zeta} \\
 & = \exp \expr{\frac{1}{\pi} \int_0^\infty \real \expr{\frac{1}{\xi - i \zeta} + \frac{1}{\xi + i \zeta}} \log \psi(\zeta^2) d\zeta} \\
 & = \exp \expr{\frac{1}{\pi} \int_0^\infty \expr{\frac{t}{t^2 + (\zeta - s)^2} + \frac{t}{t^2 + (\zeta + s)^2}} \log \psi(\zeta^2) d\zeta} \\
 & = \exp \expr{\frac{1}{\pi} \int_{-\infty}^\infty \frac{t}{t^2 + (\zeta - s)^2} \, \log \psi(\zeta^2) d\zeta} .
}
Suppose that $\psi(\zeta) \le  (C_1^2 + C_2^2 \zeta)^\alpha$ for $\zeta > 0$. Then
\formula{
 \log \psi(\zeta^2) & \le \alpha \log (C_1^2 + C_2^2 \zeta^2) = 2 \alpha \real (\log (C_1 - C_2 i \zeta)) ,
}
and the right-hand side extends to a bounded below harmonic function in the upper half-plane $\imag \zeta > 0$ (with vanishing linear term in the representation similar to~\eqref{eq:poisson:harm}). By the Poisson formula,
\formula{
 \frac{1}{\pi} \int_{-\infty}^\infty \frac{t}{t^2 + (\zeta - s)^2} \, \real (\log (C_1 - C_2 i \zeta)) d\zeta & = \real (\log (C_1 - C_2 i (s + i t))) .
}
Hence,
\formula{
 |\psi^\dagger(\xi)|^2 & \le \exp \expr{2 \alpha \real (\log (C_1 - C_2 i (s + i t)))} \\
 & = ((C_1 + C_2 t)^2 + C_2^2 s^2)^\alpha = |C_1 + C_2 \xi|^{2 \alpha}
}
This proves the first statement. The other one follows by reversing the inequalities in the above argument.
\end{proof}

\begin{rem}
In Proposition~2.1 in~\cite{bib:kmr11} it is proved that if $\psi(\xi)$ and $\xi / \psi(\xi)$ are increasing, then $\psi(\xi) / 2 \le \psi^\dagger(\xi) \le 2 \psi(\xi)$. For complete Bernstein functions, this was independently proved in Proposition~3.7 in~\cite{bib:ksv11}.\qed
\end{rem}

Recall that under Assumption~\ref{asmp}, $\psi$ is a complete Bernstein function. Hence, $\psi$ extends to a holomorphic function in $\C \setminus (-\infty, 0]$; we use the same symbol $\psi$ for this extension. First, we define an auxiliary function
\formula[eq:psistar]{
 f(\xi) & = \exp\expr{\frac{1}{\pi} \, \int_0^\infty \frac{\log \psi(\xi \zeta^2)}{1 + \zeta^2} \, d\zeta} , && \xi \in \C \setminus (-\infty, 0],
}
so that $\psi^\dagger(\xi) = f(\xi^2)$. Note that $(f(\xi))^2$ is an (integral-type) weighted geometric mean of the family of complete Bernstein functions $\psi(\xi \zeta^2)$ ($\zeta > 0$), so that $(f(\xi))^2$ is a complete Bernstein function. This is formally proved in the following simple result.

\begin{lem}
\label{lem:cbfmean}
If $\psi$ is a complete Bernstein function, then $f(\xi)$ and $(f(\xi))^2$ are complete Bernstein functions.
\end{lem}

\begin{proof}
Clearly, $f(\xi) \ge 0$ for $\xi > 0$. Furthermore,
\formula{
 \Arg f(\xi) & = \frac{1}{\pi} \int_0^\infty \frac{\imag \log \psi(\xi \zeta^2)}{1 + \zeta^2} \, d\zeta = \frac{1}{\pi} \int_0^\infty \frac{\Arg \psi(\xi \zeta^2)}{1 + \zeta^2} \, d\zeta .
}
When $\imag \xi > 0$, we have $\Arg \psi(\xi \zeta^2) \in (0, \pi)$. It follows that $\Arg f(\xi) \in (0, \pi/2)$, that is, $\imag f(\xi) > 0$ and $\imag (f(\xi))^2 > 0$. Finally, $f(\bar{\xi}) = \overline{f(\xi)}$, so $\imag f(\xi) < 0$ and $\imag (f(\xi))^2 < 0$ when $\imag \xi < 0$. The result follows by Corollary~\ref{corr:cbf}(c).
\end{proof}

Below we prove that $\psi^\dagger(\xi) = f(\xi^2)$ is a complete Bernstein function; this is stronger than the assertion of Lemma~\ref{lem:cbfmean}, see Proposition~\ref{prop:cbf:prop}(b).

\begin{lem}
\label{lem:cbfextension}
If $\psi$ is a complete Bernstein function, then $\psi^\dagger$ is a complete Bernstein function. The holomorphic extension of $\psi^\dagger$ (denoted by the same symbol) satisfies
\formula[eq:duality]{
 \psi^\dagger(\xi) \psi^\dagger(-\xi) & = \psi(-\xi^2) , && \xi \in \C \setminus \R .
}
\end{lem}

\begin{rem}
\label{rem:wh:fact}
The first statement of the above lemma was proved independently in Proposition~2.4 in~\cite{bib:ksv10} using Theorem~6.10 in~\cite{bib:ssv10}. After the preliminary version of the article was made available, Jacek Ma{\l}ecki pointed out to the author that the other statement of the lemma, formula~\eqref{eq:duality}, at least for purely imaginary $\xi$, can be deduced from the Wiener-Hopf factorization of the L{\'e}vy-Khintchine exponent in fluctuation theory, see formula~(VI.4) in~\cite{bib:b98}. The novelty of Lemma~\ref{lem:cbfextension} lies in the combination of the two parts, which is one of the key steps in the derivation of the explicit formula for the eigenfunctions $F_\lambda$ in Theorem~\ref{th:eigenfunctions}.\qed
\end{rem}

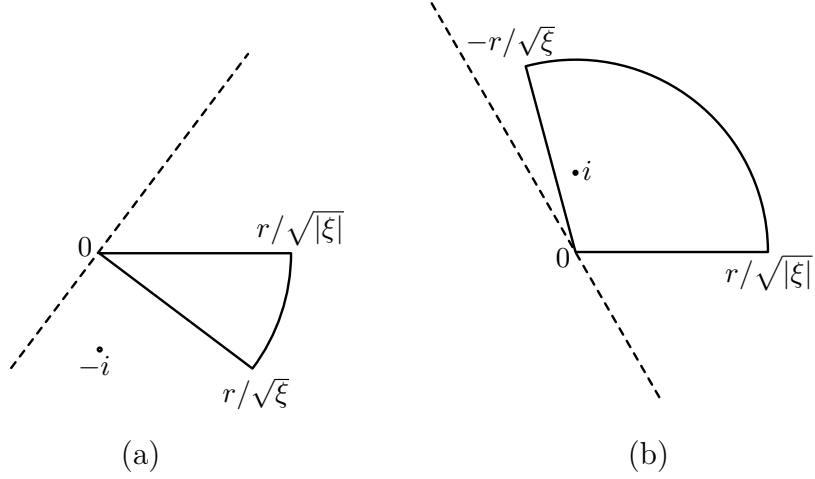
\begin{figure}
\centering
\begin{tabular}{ccc}
\def\svgwidth{5.40cm}\small

\begingroup
  \makeatletter
  \providecommand\color[2][]{%
    \errmessage{(Inkscape) Color is used for the text in Inkscape, but the package 'color.sty' is not loaded}
    \renewcommand\color[2][]{}%
  }
  \providecommand\transparent[1]{%
    \errmessage{(Inkscape) Transparency is used (non-zero) for the text in Inkscape, but the package 'transparent.sty' is not loaded}
    \renewcommand\transparent[1]{}%
  }
  \providecommand\rotatebox[2]{#2}
  \ifx\svgwidth\undefined
    \setlength{\unitlength}{270.85635986pt}
  \else
    \setlength{\unitlength}{\svgwidth}
  \fi
  \global\let\svgwidth\undefined
  \makeatother
  \begin{picture}(1,1.0675136)%
    \put(0,0){\includegraphics[width=\unitlength]{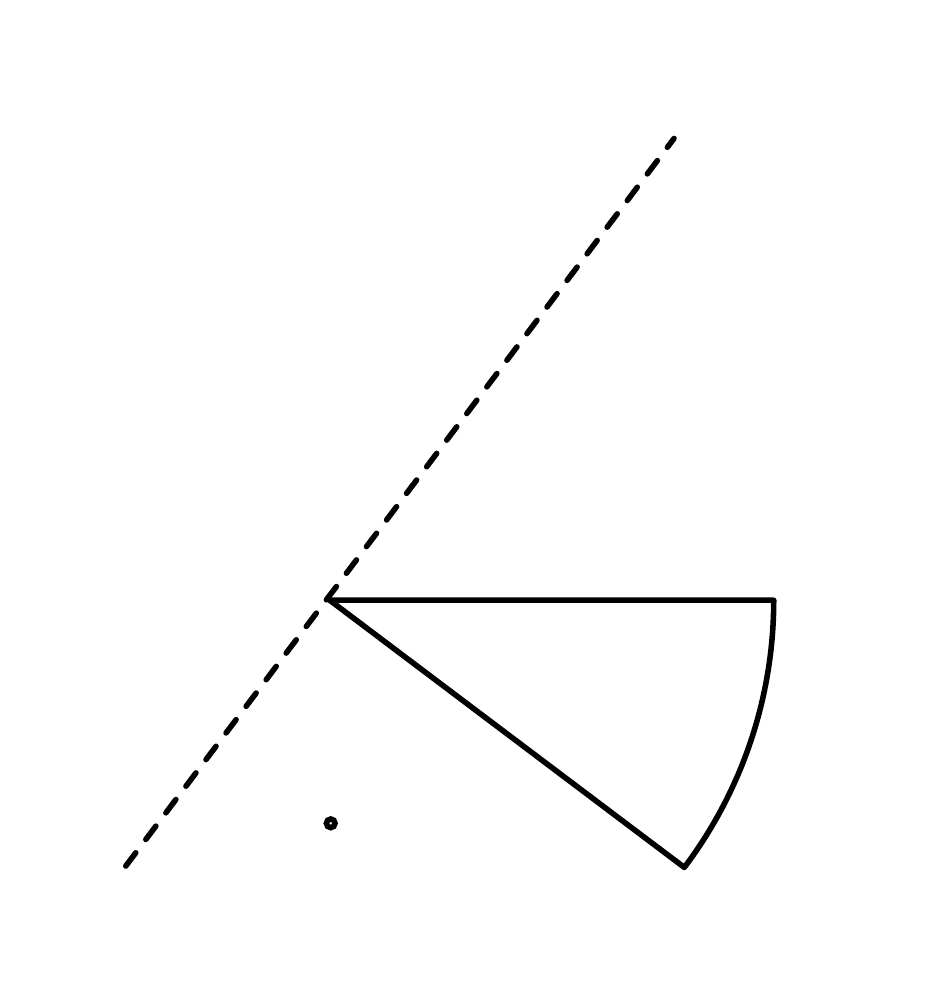}}%
    \put(0.29742268,0.42240919){\color[rgb]{0,0,0}\makebox(0,0)[lb]{\smash{$0$}}}%
    \put(0.6542081,0.06303393){\color[rgb]{0,0,0}\makebox(0,0)[lb]{\smash{$r/\sqrt{\xi}$}}}%
    \put(0.73789728,0.46891429){\color[rgb]{0,0,0}\makebox(0,0)[lb]{\smash{$r/\sqrt{|\xi|}$}}}%
    \put(0.29836725,0.1298281){\color[rgb]{0,0,0}\makebox(0,0)[lb]{\smash{$-i$}}}%
  \end{picture}%
\endgroup
&
\def\svgwidth{6.38cm}\small

\begingroup
  \makeatletter
  \providecommand\color[2][]{%
    \errmessage{(Inkscape) Color is used for the text in Inkscape, but the package 'color.sty' is not loaded}
    \renewcommand\color[2][]{}%
  }
  \providecommand\transparent[1]{%
    \errmessage{(Inkscape) Transparency is used (non-zero) for the text in Inkscape, but the package 'transparent.sty' is not loaded}
    \renewcommand\transparent[1]{}%
  }
  \providecommand\rotatebox[2]{#2}
  \ifx\svgwidth\undefined
    \setlength{\unitlength}{318.85715332pt}
  \else
    \setlength{\unitlength}{\svgwidth}
  \fi
  \global\let\svgwidth\undefined
  \makeatother
  \begin{picture}(1,0.90680997)%
    \put(0,0){\includegraphics[width=\unitlength]{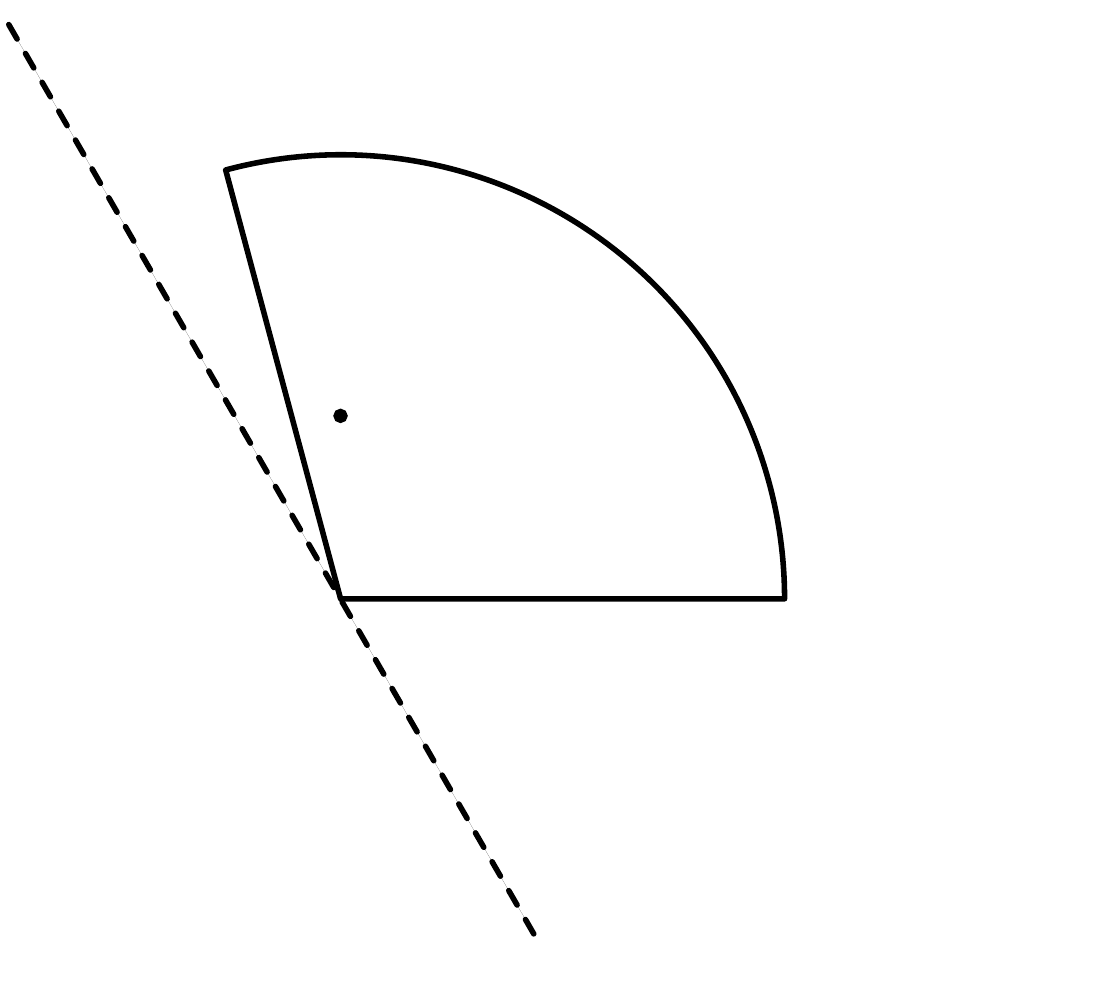}}%
    \put(0.26634191,0.33373005){\color[rgb]{0,0,0}\makebox(0,0)[lb]{\smash{$0$}}}%
    \put(0.61959603,0.30444088){\color[rgb]{0,0,0}\makebox(0,0)[lb]{\smash{$r/\sqrt{|\xi|}$}}}%
    \put(0.07921712,0.78757119){\color[rgb]{0,0,0}\makebox(0,0)[lb]{\smash{$-r/\sqrt{\xi}$}}}%
    \put(0.32160686,0.51651332){\color[rgb]{0,0,0}\makebox(0,0)[lb]{\smash{$i$}}}%
  \end{picture}%
\endgroup
\\
(a)&(b)
\end{tabular}
\caption{Two contours used in the proof of Lemma~\ref{lem:cbfextension}. Dashed line depicts the boundary of the region in which the integrand is meromorphic. In (a), $\real \xi > 0$ and $\imag \xi > 0$, and the integrand $\log \psi(\xi^2 \zeta^2) / (1 + \zeta^2)$ is meromorphic in the region $\real(\xi \zeta) > 0$ with pole at $-i$, outside the contour of integration. In (b), $\real \xi < 0$ and $\imag \xi > 0$, and the integrand $\log \psi(\xi^2 \zeta^2) / (1 + \zeta^2)$ is meromorphic in $\real(-\xi \zeta) > 0$ with a simple pole at $i$.
}
\label{fig:contour}
\end{figure}

\begin{proof}
Let $f$ be (the holomorphic extension of) the complete Bernstein function defined in~\eqref{eq:psistar} (see Lemma~\ref{lem:cbfmean}); hence, $\psi^\dagger(\xi) = f(\xi^2)$ for $\xi > 0$, and this identity defines a holomorphic extension of $\psi^\dagger$ to the half-plane $\real \xi > 0$. When $\real \xi > 0$, by~\eqref{eq:psistar} we have
\formula{
 \psi^\dagger(\xi) = f(\xi^2) & = \exp\expr{\frac{1}{\pi} \, \int_0^\infty \frac{\log \psi(\xi^2 \zeta^2)}{1 + \zeta^2} \, d\zeta} .
}
For a fixed $\xi$, the integrand on the right-hand side is a meromorphic function of $\zeta$ in the region $\real(\xi \zeta) > 0$ with a simple pole at $-i$. Furthermore, by Proposition~\ref{prop:cbf:ests}(e), for any $\eps > 0$ the integrand decays at least as fast as $|\zeta|^{-2} \log |\zeta|^2$ when $|\zeta| \to \infty$ in the region $-\pi/2 + \eps < \Arg(\xi \zeta) < \pi/2 - \eps$. Hence, by a standard contour integration (using contours shown in Figure~\ref{fig:contour}~(a) with $r \to \infty$; the pole at $-i$ is outside the contour) and then parametrization $\zeta = s / \sqrt{\xi}$ of $[0, \xi^{-1/2} \infty)$, we obtain that
\formula{
 \psi^\dagger(\xi) & = \exp\expr{\frac{1}{\pi} \, \int_0^{\xi^{-1/2} \infty} \frac{\log \psi(\xi^2 \zeta^2)}{1 + \zeta^2} \, d\zeta} \\ & = \exp\expr{\frac{1}{\pi} \, \int_0^\infty \frac{\sqrt{\xi} \log \psi(\xi s^2)}{\xi + s^2} \, ds} .
}
The formula on the right-hand side clearly defines a holomorphic function in $\C \setminus (-\infty, 0]$ (integrability follows by Proposition~\ref{prop:cbf:ests}(e)), which is the holomorphic extension of $\psi^\dagger(\xi)$. Suppose that $\real \xi < 0$ and $\imag \xi > 0$. Using the parametrization $\zeta = -s / \sqrt{\xi}$ of $[0, -\xi^{-1/2} \infty)$, we obtain that
\formula{
 \psi^\dagger(\xi) & = \exp\expr{-\frac{1}{\pi} \, \int_0^{-\xi^{-1/2} \infty} \frac{\log \psi(\xi^2 \zeta^2)}{1 + \zeta^2} \, d\zeta} .
}
For a fixed $\xi$, the integrand on the right-hand side is a meromorphic function of $\zeta$ in the region $\real(-\xi \zeta) > 0$ with a simple pole at $i$. Furthermore, again by Proposition~\ref{prop:cbf:ests}(e), for any $\eps > 0$, the integrand decays at least as fast as $\zeta^{-2} \log |\zeta|^2$ when $|\zeta| \to \infty$ in the region $-\pi/2 + \eps < \Arg(-\xi \zeta) < \pi/2 - \eps$. Hence, again using standard contour integration (along contours shown in Figure~\ref{fig:contour}~(b) with $r \to \infty$; this time the pole at $i$ is inside the contour) and the residue theorem, we obtain that
\formula{
 \psi^\dagger(\xi) & = \exp\expr{-\frac{1}{\pi} \, \int_0^\infty \frac{\log \psi(\xi^2 \zeta^2)}{1 + \zeta^2} \, d\zeta + 2 i \res_i\expr{\frac{\log \psi(\xi^2 \zeta^2)}{1 + \zeta^2}}} \\
 & = \exp\expr{-\frac{1}{\pi} \, \int_0^\infty \frac{\log \psi(\xi^2 \zeta^2)}{1 + \zeta^2} \, d\zeta + \log \psi(-\xi^2)} = \frac{\psi(-\xi^2)}{f(\xi^2)} \, .
}
Recall that $\real \xi < 0$ and $\imag \xi > 0$. Hence $\real (-\xi) > 0$, so that $f(\xi^2) = f((-\xi)^2) = \psi^\dagger(-\xi)$. Formula~\eqref{eq:duality} follows when $\real \xi < 0$ and $\imag \xi > 0$.

Since $\psi$ and $f$ are complete Bernstein functions, by Corollary~\ref{corr:cbf}(c), $h_1(\xi) = \Arg \psi(-\xi^2)$ and $h_2(\xi) = \Arg f(\xi^2)$ are bounded harmonic functions in the region $\real \xi < 0$, $\imag \xi > 0$, taking values in $[0, \pi]$ and $[-\pi, 0]$, respectively. Hence, $h_1$ and $h_2$ are Poisson integrals of their (bounded, measurable) boundary values $g_1$, $g_2$ (defined on $(-\infty, 0] \cup [0, i \infty)$). Furthermore, we have $g_1(i s) = \Arg \psi(s^2) = 0$ and $g_2(-s) = \Arg f(s^2) = 0$ for $s > 0$, and $g_1(-s) \in [0, \pi]$, $g_2(i s) \in [-\pi, 0]$ for almost all $s > 0$.

Observe that when $\real \xi < 0$, $\imag \xi > 0$, we have
\formula{
 \Arg \psi^\dagger(\xi) & = \Arg \psi(-\xi^2) - \Arg f(\xi^2) = h_1(\xi) - h_2(\xi) .
}
Hence, $h(\xi) = \Arg \psi^\dagger(\xi)$ is a bounded harmonic function in the region $\real \xi < 0$, $\imag \xi > 0$, with boundary value $g$ given by $g = g_1$ on $(-\infty, 0]$, and $g = -g_2$ on $[0, i \infty)$. It follows that $g(\xi) \in [0, \pi]$ for almost every boundary point $\xi$, and so $h(\xi) \in [0, \pi]$ for all $\xi$ such that $\real \xi < 0$, $\imag \xi > 0$. It follows that $\imag \psi^\dagger(\xi) \ge 0$ in this region. We have already seen that $\imag \psi^\dagger(\xi) = \imag f(\xi^2) \ge 0$ when $\real \xi > 0$, $\imag \xi > 0$. Finally, $\psi^\dagger(\bar{\xi}) = \overline{\psi^\dagger(\xi)}$, and so $\imag \psi^\dagger(\xi) \le 0$ when $\imag \xi < 0$. By Corollary~\ref{corr:cbf}(c), $\psi^\dagger$ is a CBF, and~\eqref{eq:duality} extends to all $\xi \in \C \setminus \R$.
\end{proof}

We note a third formula for (the holomorphic extension of) $\psi^\dagger(\xi)$, other than~\eqref{eq:dagger} and~\eqref{eq:dagger1}: when $\psi$ is a complete Bernstein function, then
\formula[eq:dagger2]{
 \psi^\dagger(\xi) & = \exp\expr{\frac{1}{\pi} \, \int_0^\infty \frac{\sqrt{\xi} \log \psi(\xi s^2)}{\xi + s^2} \, ds} , && \xi \in \C \setminus (-\infty, 0] ;
}
see the above proof of Lemma~\ref{lem:cbfextension}.

In Section~\ref{sec:eigenfunctions}, the transformation $\psi \mapsto \psi^\dagger$ is applied to the function
\formula[eq:psilambda]{
 \psi_\lambda(\xi) & = \frac{1 - \xi / \lambda^2}{1 - \psi(\xi) / \psi(\lambda^2)} \, , && \lambda > 0, \, \xi \in \C \setminus (-\infty, 0) .
}
This definition is continuously extended at $\xi = \lambda^2$ by letting $\psi_\lambda(\lambda^2) = \psi(\lambda^2) / (\lambda^2 \psi'(\lambda^2))$. We denote $\psi_\lambda^\dagger(\xi) = (\psi_\lambda)^\dagger(\xi)$.

By Proposition~\ref{prop:cbf:ratio}, $\psi_\lambda$ and $\log \psi_\lambda$ are complete Bernstein functions, and $\psi_\lambda(0) = 1$. It is easy to see that $\psi_\lambda'(\lambda) = (\psi(\lambda^2) |\psi''(\lambda^2)|) / (2 \lambda^2 (\psi'(\lambda^2))^2)$ (note that $\psi''(\lambda^2) < 0$). Since $\psi$ is concave, its graph lies below the tangent line at $\xi = \lambda^2$, that is,
\formula[eq:psilambdaest]{
 \psi_\lambda(\xi) & \le \frac{\psi(\lambda^2)}{\lambda^2 \psi'(\lambda^2)} + \frac{\psi(\lambda^2) |\psi''(\lambda^2)|}{2 \lambda^2 (\psi'(\lambda^2))^2} \, (\xi - \lambda^2) , && \lambda, \xi > 0 .
}
Formula~\eqref{eq:psilambdaest} (plus $\xi - \lambda^2 \le \xi$) combined with Proposition~\ref{prop:explogbound} yields that
\formula[eq:psilambdadaggerest]{
 \psi_\lambda^\dagger(\xi) & \le \abs{\sqrt{\frac{\psi(\lambda^2)}{\lambda^2 \psi'(\lambda^2)}} + \sqrt{\frac{\psi(\lambda^2) |\psi''(\lambda^2)|}{2 \lambda^2 (\psi'(\lambda^2))^2}} \, \xi} , && \lambda > 0, \, \real \xi > 0 .
}
We remark that more detailed estimates of $\psi_\lambda$ and $\psi_\lambda^\dagger$ can be found in~\cite{bib:kmr11a}.

%
%

\section{Derivation of the formula for eigenfunctions}
\label{sec:eigenfunctions}

%
%

Below we prove Theorem~\ref{th:eigenfunctions}. Our strategy is as follows. First we state a distributional version of the spectral problem for $\A_\hl$, and in Lemma~\ref{lem:eigenfunctions} we prove that a distributional solution is automatically a strong solution. The (simpler) case of compound Poisson processes requires a different approach and is studied separately in Proposition~\ref{prop:poisson}. Next, we rephrase the distributional problem as a Wiener-Hopf equation, and we solve it using the Wiener-Hopf method. Finally, we combine the two parts and prove Theorem~\ref{th:eigenfunctions}. In addition, we establish some basic properties of the eigenfunctions.

%
%

\subsection{Distributional and strong eigenfunctions}

In this part, $X_t$ can be an arbitrary one-dimensional \emph{symmetric} L{\'e}vy process with L{\'e}vy-Khintchine exponent $\Psi$. We define $\psi$ by the formula $\Psi(\xi) = \psi(\xi^2)$. Note that $\psi(\xi)$ need not be a Bernstein function.

Recall that $\A$ is the generator of $X_t$, and $\A_\hl$ is the generator of the process $X_t$ killed upon leaving $\hl$, defined on an appropriate domain. Here we only consider the $L^\infty$, $C_b$ and $C_0$ generators, formal definitions are given in Preliminaries. Recall also that $A \in \schwartz'$ is the distributional generator of $X_t$, see Definition~\ref{def:gen:distr}. The Fourier transform of the distribution $A$ is $-\psi(\xi^2)$, and since $A$ is symmetric, it is equal to its reflection $\check{A}$.

Let $F \in L^\infty(\hl)$. As usual, $F$ is extended to entire real line by setting $F(x) = 0$ for $x \le 0$. Given some regularity of $X_t$, if $F \in C_0(\hl)$ then, by Lemma~\ref{lem:d2}, the condition $\A_\hl F = -\psi(\lambda^2) F$ is equivalent to the apparently weaker condition $A \conv F(x) = -\psi(\lambda^2) F(x)$ for $x \in \hl$. This motivates the following definition.

\begin{defin}
\label{def:weak}
Let $\lambda > 0$. A tempered distribution $F$ is said to be a \emph{distributional eigenfunction} of $\A_\hl$, corresponding to the eigenvalue $-\psi(\lambda^2)$, if $F$ is supported in $[0, \infty)$, $\schwartz'$-convolvable with $A$, and $A \conv F + \psi(\lambda^2) F$ is supported in $(-\infty, 0]$. Informally,
\formula[eq:problem]{
\begin{aligned}
 F & = 0 &\qquad& \text{on $(-\infty, 0)$} , \\
 A \conv F & = \psi(\lambda^2) F && \text{on $(0, \infty)$} .
\end{aligned}
}
\end{defin}

\begin{lem}
\label{lem:eigenfunctions}
Suppose that $X_t$ is a symmetric L{\'e}vy process, which has the strong Feller property, and that $0$ is a regular boundary point of $\hl$. Let $\lambda > 0$, and suppose that $F$ is a distributional eigenfunction of $\A_\hl$, corresponding to the eigenvalue $-\psi(\lambda^2)$. If $F$ is bounded and continuous on $\R$, and if
\formula{
 \lim_{x \to \infty} (F(x) - C \sin(\lambda x + \thet)) = 0
}
for some $C, \thet \in \R$, then $F \in \domain(\A_\hl; L^\infty)$ and $\A_\hl F = -\psi(\lambda^2) F$.
\end{lem}

\begin{proof}
Since $F$ does not vanish at infinity, we cannot apply Lemma~\ref{lem:d2} directly to $F$ to prove that $F \in \domain(\A_\hl; L^\infty)$. Thus, we decompose $F$ into the sum of a $C_0(\hl)$ function and a smooth bounded function.

Denote $F^*(x) = C \sin(\lambda x + \thet)$. By Remark~\ref{rem:spectral:free}, $F^* \in \domain(\A; L^\infty)$ and $\A F^* = -\psi(\lambda^2) F^*$. By Proposition~\ref{prop:gen:distr}, $A \conv F^* = -\psi(\lambda^2) F^*$. We choose an infinitely smooth function $h$ such that $h(x) = 0$ for $x \le 1$ and $h(x) = 1$ for $x$ large enough. Furthermore, we require that
\formula[eq:cutoff]{
 \int_1^\infty h(z) F^*(z) \nu(z) dz & = 0 .
}
(Note that such a choice is always possible, as we do not assume that $0 \le h(x) \le 1$.) We write $F = f_1 + f_2$, where $f_1 = h F^*$ and $f_2 = F - h F^*$. The function $f_1$ is infinitely smooth, with all derivatives bounded, so by Lemma~\ref{lem:gen:cb}, $f_1 \in \domain(\A; C_b)$. We have $f_1(x) = 0$ for $x \le 1$, and by~\eqref{eq:gen:cb} and~\eqref{eq:cutoff}, $\A f_1(0) = 0$. Hence, by Lemma~\ref{lem:smooth:domain} and Proposition~\ref{prop:gen:distr}, we have $f_1 \in \domain(\A_\hl; L^\infty)$ and $\A_\hl f_1(x) = \A f_1(x) = A \conv f_1(x)$ for $x \in \hl$. We claim that also $f_2 \in \domain(\A_\hl; L^\infty)$ and $\A_\hl f_2(x) = A \conv f_2(x)$ for $x \in \hl$. Once this is proved, we have $F \in \domain(\A_\hl; L^\infty)$ and $\A_\hl F(x) = A \conv F(x) = -\psi(\lambda^2) F(x)$ for $x > 0$, as desired. Hence, it remains to prove the claim.

We will apply Lemma~\ref{lem:d2} to $f_2 = F - h F^*$. Since $F$ and $h$ vanish on $(-\infty, 0]$, we have $f_2(x) = 0$ for $x \le 0$. Also, $f_2 = (F - F^*) + (1 - h) F^*$, and both $F - F^*$, $(1 - h)$ vanish at infinity, so that $\lim_{x \to \infty} f_2(x) = 0$. In other words, $f_2 \in C_0(\hl)$. Furthermore,
\formula{
 A \conv f_2 & = A \conv F - A \conv (h F^*) = A \conv F - A \conv f_1 .
}
By assumption, $A \conv F$ restricted to $\hl$ is equal to $-\psi(\lambda^2) F(x)$. We have already seen that $A \conv f_1 = \A f_1$ is a continuous function on $\R$ which vanishes at $0$. Hence, $A \conv f_2$ restricted to $\hl$ is equal to a continuous function, vanishing continuously at $0$. Furthermore, for $x > 0$,
\formula{
 A \conv f_2(x) & = A \conv F(x) - A \conv F^*(x) + A \conv ((1 - h) F^*)(x) \\
 & = -\psi(\lambda^2) (F(x) - F^*(x)) + A \conv ((1 - h) F^*)(x) .
}
By assumption, the first term on the right-hand side converges to $0$ as $x \to \infty$. Since $1 - h(x) = 0$ for $x$ large enough, also $A \conv ((1 - h) F^*)(x)$ vanishes as $x \to \infty$ (by Lemma~\ref{lem:gen:cb} and Proposition~\ref{prop:gen:distr}). We conclude that $A \conv f_2$ restricted to $\hl$ is equal to a $C_0(\hl)$ function. By Lemma~\ref{lem:d2}, $f_2 \in \domain(\A_\hl; C_0)$ (and therefore also $f_2 \in \domain(\A_\hl; L^\infty)$), and $\A_\hl f_2(x) = A \conv f_2(x)$ for $x > 0$.
\end{proof}

\begin{rem}
Continuity of $F$ at $0$ is essential. Indeed, when $X_t$ is the Brownian motion (i.e. $\psi(\xi) = \xi$), then $F(x) = \cos(\lambda x) \ind_\hl(x)$ is the distributional eigenfunction of $\A = d^2 / dx^2$ in $(0, \infty)$, but $F$ is not in the $L^\infty(\R^d)$ domain of $\A_\hl$. In particular $P^\hl_t F$ is not equal to $e^{-\lambda^2 t} F$, as can be verified by a direct calculation.\qed
\end{rem}

When $X_t$ is a compound Poisson process, the proof is much easier.

\begin{prop}
\label{prop:poisson}
Suppose that $X_t$ is a symmetric compound Poisson process. Let $\lambda > 0$, and suppose that $F \in L^\infty(\hl)$ is a distributional eigenfunction of $\A_\hl$, corresponding to the eigenvalue $-\psi(\lambda^2)$. Then $\A_\hl F = -\psi(\lambda^2) F$.
\end{prop}

\begin{proof}
The result is a straightforward application of Lemma~\ref{lem:cp} and the fact that $A$ is the signed measure $\nu(dz) - \nu(\R) \delta_0(dz)$.
\end{proof}

%
%

\subsection{Wiener-Hopf factorization}

A method of solving problems of the form~\eqref{eq:problem} was found in~\cite{bib:wh31} in the case when $A$ is an integrable function satisfying some growth conditions. For this reason, problems of this kind are called \emph{Wiener-Hopf equations}, and the algorithm for finding their solutions is the \emph{Wiener-Hopf method}; see~\cite{bib:ggk90, bib:k62} for a detailed exposition. We begin with a brief description of the method. A version of Paley-Wiener theorem states that a function is supported in $[0, \infty)$ or $(-\infty, 0]$ if and only if its Fourier transform extends holomorphically to the upper or lower complex half-plane, respectively. Therefore, \eqref{eq:problem} can be re-written as
\formula{
 & \text{$\fourier F(\xi)$ extends to the upper complex half-plane,} \\
 & \text{$(C - \psi(\xi^2)) \fourier F(\xi)$ extends to the lower complex half-plane,}
}
where $C = \psi(\lambda^2)$. The main ingredient of the \emph{Wiener-Hopf method} is factorization of the symbol $C - \psi(\xi^2)$ into the product of two functions, one extending holomorphically to the upper half-plane, the other one to the lower half-plane (the \emph{Wiener-Hopf factors}). Then $\fourier F$ is taken to be the multiplicative inverse of the first factor.

The fundamental condition for the Wiener-Hopf method is that the symbol does not vanish. In our case, we have $C - \psi(\xi^2) = 0$ for $\xi = \pm \lambda$. For this reason, we apply the Wiener-Hopf method to the \emph{regularized} symbol
\formula{
 \frac{1 - \psi(\xi^2) / \psi(\lambda^2)}{1 - \xi^2 / \lambda^2} = \frac{1}{\psi_\lambda(\xi^2)} \, ;
}
here we use the notation of~\eqref{eq:psilambda}.

From now on, Assumption~\ref{asmp} is in force. It is required in order to establish some properties of $1 / \psi_\lambda(\xi^2)$ and its Wiener-Hopf factors. In particular, since $\psi_\lambda$ is a complete Bernstein function, the factorization is given by Lemma~\ref{lem:cbfextension}:
\formula{
 \frac{1}{\psi_\lambda(\xi^2)} & = \frac{1}{\psi_\lambda^\dagger(-i \xi)} \, \frac{1}{\psi_\lambda^\dagger(i \xi)} \, , && \xi \in \R .
}
The Wiener-Hopf factors are $\psi_\lambda^\dagger(-i \xi)$ and $\psi_\lambda^\dagger(i \xi)$, and the solution is expected to satisfy $\fourier F(\xi) = \psi_\lambda^\dagger(-i \xi) / (1 - \xi^2 / \lambda^2)$, up to a multiplicative constant. Due to the singularity in the denominator, we prefer to work with the Laplace transform of $F$. Below we provide details for the above idea.

Fix $\lambda > 0$, and define a tempered distribution
\formula{
 A_\lambda & = \psi(\lambda^2) \delta_0 + A ,
}
where $A$ is the distributional generator of $X_t$ and $\delta_0$ is the Dirac delta measure. Hence, $\fourier A(\xi) = \psi(\lambda^2) - \psi(\xi^2)$. We decompose $A_\lambda$ into the convolution of the distribution 
\formula[eq:s]{
 S & = \frac{\psi(\lambda^2)}{\lambda^2} \, (\lambda^2 \delta_0 + \delta_0'')
}
(the \emph{singular part}) and two infinitely divisible measures $R_+$ and $R_-$ (\emph{regular parts}) supported in $[0, \infty)$ and $(-\infty, 0]$ respectively.

Clearly,
\formula[eq:fouriers]{
 \fourier S(\xi) & = \frac{\psi(\lambda^2)}{\lambda^2} \, (\lambda^2 - \xi^2) = \psi(\lambda^2) \expr{1 - \frac{\xi^2}{\lambda^2}} .
}
We define the tempered distribution $R$ by means of Fourier transform,
\formula[eq:fourierr]{
 \fourier R(\xi) & = \frac{1 - \psi(\xi^2) / \psi(\lambda^2)}{1 - \xi^2 / \lambda^2} = \frac{1}{\psi_\lambda(\xi^2)} , && \xi \in \R ,
}
with $\psi_\lambda$ defined in~\eqref{eq:psilambda}. Since $S$ is compactly supported, $S$ and $R$ are $\schwartz'$-convolvable, and since $\fourier S(\xi)\fourier R(\xi) = \psi(\lambda^2) - \psi(\xi^2) = \fourier A_\lambda$, by the exchange formula we have
\formula{
 A_\lambda & = S \conv R .
}

\begin{lem}
\label{lem:laplacerp}
There is a probability measure $R_+$ supported in $[0, \infty)$, such that (see~\eqref{eq:dagger1})
\formula[eq:laplacerp]{
 \laplace R_+(\xi) & = \exp\expr{-\frac{1}{\pi} \int_0^\infty \frac{\xi \log \psi_\lambda(\zeta^2)}{\xi^2 + \zeta^2} \, d\zeta} = \frac{1}{\psi_\lambda^\dagger(\xi)} \, , && \real \xi \ge 0 .
}
Furthermore, $R_+$ has a completely monotone density on $(0, \infty)$ (but it may have an atom at $0$).

If $R_-(E) = R_+(-E)$ is the reflection of $R_+$, then $R = R_+ * R_-$ and $\fourier R(\xi) = \fourier R_+(\xi) \fourier R_-(\xi)$ ($\xi \in \R$).
\end{lem}

\begin{proof}
By Lemma~\ref{lem:cbfextension}, $\psi_\lambda^\dagger$ is a complete Bernstein function. By Proposition~\ref{prop:cbf:prop}(a), $1 / \psi_\lambda^\dagger$ is a Stieltjes function, and therefore it is the Laplace transform of a measure $R_+$, which has completely monotone density function on $(0, \infty)$. By Proposition~\ref{prop:dagger:prop}(d), $1 / \psi_\lambda^\dagger(0) = 1 / \psi_\lambda(0) = 1$. Hence, $R_+$ is a probability measure. Finally, $\fourier R_+(\xi) = 1 / \psi_\lambda^\dagger(-i \xi)$ (where the holomorphic extension of $\psi_\lambda^\dagger$ is denoted by the same symbol).

Let $R_-$ be the reflection of $R_+$. By~\eqref{lem:cbfextension},
\formula{
 \fourier R(\xi) & = \frac{1}{\psi_\lambda(\xi^2)} = \frac{1}{\psi_\lambda^\dagger(-i \xi)} \, \frac{1}{\psi_\lambda^\dagger(i \xi)} = \fourier R_+(\xi) \fourier R_-(\xi)
}
for $\xi \in \R$. The lemma follows by the exchange formula.
\end{proof}

\begin{lem}
\label{lem:reduction}
Suppose that $F \in \schwartz'$ is a distribution supported in $[0, \infty)$, for which $F * \ph$ is bounded for any $\ph \in \schwartz$. If $F \conv (S \conv R_+)$ is supported in $\set{0}$, then $F$ is a distributional eigenfunction of $\A_\hl$.
\end{lem}

\begin{proof}
Note that both $F$ and $S \conv R_+$ are supported in $[0, \infty)$, and so they are $\schwartz'$-convolvable. Furthermore, for any $\ph_1, \ph_2, \ph_3 \in \schwartz$, $F * \ph_1$ is bounded, and $(S \conv R_+) * \ph_2$, $R_- * \ph_3$ are integrable. Therefore, the $\schwartz'$-convolution of $F$, $S \conv R_+$ and $R_-$ is associative. In other words, if we let $\tilde{F} = F \conv (S \conv R_+)$, then $\tilde{F} \conv R_- = A_\lambda \conv F$. Clearly, if the support of $\tilde{F}$ is $\set{0}$, then $A_\lambda \conv F = \tilde{F} \conv R_-$ is supported in $(-\infty, 0]$.
\end{proof}

%
%

\subsection{Formula for eigenfunctions}

Lemma~\ref{lem:reduction} enables us to describe a family of solutions to~\eqref{eq:problem}. If $F \conv (S \conv R_+)$ is supported in $\set{0}$, then $\fourier F \cdot \fourier (S \conv R_+) = P$ for some polynomial $P$. Therefore, the Fourier transform of a distributional eigenfunction $F$ is expected to have the form $P / \fourier (S \conv R_+)$; however, extra care is needed because of the zeros of the denominator $\fourier (S \conv R_+)$ at $\pm \lambda$.

Let (see~\eqref{eq:laplacerp})
\formula[eq:lf0]{
\begin{aligned}
 f(\xi) & = \frac{c_\lambda}{(\lambda^2 + \xi^2) \laplace R_+(\xi)} = \frac{c_\lambda \psi_\lambda^\dagger(\xi)}{\lambda^2 + \xi^2} && \qquad \xi > 0 ,
\end{aligned}
}
where
\formula[eq:clambda]{
 c_\lambda & = \sqrt{\frac{\lambda^2}{\psi_\lambda(\lambda^2)}} = \sqrt{\frac{\lambda^4 \psi'(\lambda^2)}{\psi(\lambda^2)}}
}
is a normalization constant. Since $\psi_\lambda^\dagger$ is a complete Bernstein function, $f$ extends to a meromorphic function in $\C \setminus (-\infty, 0]$, with poles at $\pm i \lambda$. We denote this extension again by $f$. By~\eqref{eq:dagger1} and Proposition~\ref{prop:dagger:prop}(b),
\formula[eq:switch]{
\begin{aligned}
 \psi_\lambda^\dagger(\xi) & = \sqrt{\psi_\lambda(\lambda^2)} \, \exp\expr{\frac{1}{\pi} \int_0^\infty \frac{\xi}{\xi^2 + \zeta^2} \, \log \frac{\psi_\lambda(\zeta^2)}{\psi_\lambda(\lambda^2)} \, d\zeta} \\
 & = \sqrt{\frac{\psi(\lambda^2)}{\lambda^2 \psi'(\lambda^2)}} \, \exp\expr{\frac{1}{\pi} \int_0^\infty \frac{\xi}{\xi^2 + \zeta^2} \, \log \frac{\psi'(\lambda^2) (\lambda^2 - \zeta^2)}{\psi(\lambda^2) - \psi(\zeta^2)} \, d\zeta} ;
\end{aligned}
}
this holds true for complex $\xi$ with $\real \xi > 0$. It follows that $f(\xi)$ is equal to the right-hand side of~\eqref{eq:lf}. In order to prove Theorem~\ref{th:eigenfunctions}, we need to show that $f$ is the Laplace transform of a function $F$ of the form~\eqref{eq:f}, and that $F$ is a (distributional, and hence $L^\infty$) eigenfunction of $\A_\hl$.

We define
\formula[eq:fstar]{
 f^*(\xi) & = \frac{c_\lambda}{2 i \lambda} \expr{\frac{\psi_\lambda^\dagger(i \lambda)}{\xi - i \lambda} - \frac{\psi_\lambda^\dagger(-i \lambda)}{\xi + i \lambda}} , && \xi \in \C \setminus \{-i \lambda, i \lambda\} ,
}
and
\formula[eq:gstar]{
 g(\xi) & = f^*(\xi) - f(\xi) = \frac{1}{2 i \lambda} \expr{\frac{c_\lambda \psi_\lambda^\dagger(i \lambda)}{\xi - i \lambda} - \frac{c_\lambda \psi_\lambda^\dagger(-i \lambda)}{\xi + i \lambda}} - \frac{c_\lambda \psi_\lambda^\dagger(\xi)}{\lambda^2 + \xi^2} \, .
}
Note that $C_1 (\xi - i C_2)^{-1}$ is the Laplace transform of $C_1 e^{-i C_2 x} \ind_{(0, \infty)}(x)$. Hence, $f^*$ is the Laplace transform of
\formula[eq:fstar2]{
 F^*(x) & = \frac{c_\lambda \psi_\lambda^\dagger(i \lambda) e^{i \lambda x} - c_\lambda \psi_\lambda^\dagger(-i \lambda) e^{-i \lambda x}}{2 i \lambda} \, \ind_{(0, \infty)}(x) , && x \in \R .
}
By Lemma~\ref{lem:g} (applied to a CBF $c_\lambda \psi_\lambda^\dagger$), $g$ is a Stieltjes function, it is the Laplace transform of a completely monotone function $G(x)$, and $G$ is the Laplace transform of a finite measure. Hence, $f$ is the Laplace transform of $F(x) = F^*(x) - G(x)$.

\begin{prop}
\label{prop:residue}
We have $\psi_\lambda^\dagger(\pm i \lambda) = \sqrt{\psi_\lambda(\lambda^2)} \, e^{\pm i \thet_\lambda}$, where $\thet_\lambda$ is given by~\eqref{eq:theta}.
\end{prop}

\begin{proof}
By~\eqref{eq:switch}, for $\eps \in (0, \lambda)$ we have
\formula{
 \psi_\lambda^\dagger(i \lambda + \eps) & = \sqrt{\psi_\lambda(\lambda^2)} \, \exp\expr{\frac{1}{\pi} \, \int_0^\infty \frac{i \lambda + \eps}{(i \lambda + \eps)^2 + \zeta^2} \, \log \frac{\psi_\lambda(\zeta^2)}{\psi_\lambda(\lambda^2)} \, d\zeta} .
}
We claim that the integrand on the right-hand side is dominated by an integrable function. Then, by dominated convergence,
\formula{
 \psi_\lambda^\dagger(i \lambda) & = \lim_{\eps \to 0^+} \psi_\lambda^\dagger(i \lambda + \eps) \\
 & = \sqrt{\psi_\lambda(\lambda^2)} \, \exp\expr{\frac{1}{\pi} \, \int_0^\infty \frac{i \lambda}{-\lambda^2 + \zeta^2} \, \log \frac{\psi_\lambda(\zeta^2)}{\psi_\lambda(\lambda^2)} \, d\zeta} ,
}
as desired (see~\eqref{eq:theta} and~\eqref{eq:switch}). Hence, it remains to prove the claim.

Observe that
\formula{
 \abs{\frac{i \lambda + \eps}{(i \lambda + \eps)^2 + \zeta^2}} & = \frac{|i \lambda + \eps|}{|i (\lambda - \zeta) + \eps| \, |i (\lambda + \zeta) + \eps|} \\
 & \le \frac{\lambda + \eps}{|\lambda - \zeta| (\lambda + \zeta)} \le \frac{2 \lambda}{|\lambda - \zeta| (\lambda + \zeta)} \, .
}
Since $\psi_\lambda$ is nonnegative and concave, we have $\psi_\lambda(\zeta^2) / \psi_\lambda(\lambda^2) \ge \zeta^2 / \lambda^2$ for $\zeta \le \lambda$, and $\psi_\lambda(\zeta^2) / \psi_\lambda(\lambda^2) \le \zeta^2 / \lambda^2$ for $\zeta \ge \lambda$. This and monotonicity of $\psi_\lambda$ gives that
\formula{
 \abs{\log \frac{\psi_\lambda(\zeta^2)}{\psi_\lambda(\lambda^2)}} & \le \abs{\log \frac{\zeta^2}{\lambda^2}} = 2 |\log \zeta - \log \lambda| .
}
It follows that
\formula{
 \abs{\frac{i \lambda + \eps}{(i \lambda + \eps)^2 + \zeta^2} \, \log \frac{\psi_\lambda(\zeta^2)}{\psi_\lambda(\lambda^2)}} & \le \frac{4 \lambda |\log \zeta - \log \lambda|}{|\zeta - \lambda| (\zeta + \lambda)} \, ,
}
which is integrable in $\zeta > 0$.
\end{proof}

The above proposition and formulae~\eqref{eq:clambda} and~\eqref{eq:fstar2} give that $F^*(x) = \sin(\lambda x + \thet_\lambda) \ind_\hl(x)$. We summarize our development in the following result.

\begin{cor}
\label{cor:f}
The function $f$ is the Laplace transform of a function $F(x) = F^*(x) - G(x)$, where
\formula{
 F^*(x) & = \sin(\lambda x + \thet_\lambda) \ind_{\hl}(x) ,
}
$G$ is completely monotone on $(0, \infty)$ and it is the Laplace transform of a finite measure on $(0, \infty)$.\qed
\end{cor}

\begin{lem}
\label{lem:f}
The function $F$ constructed above is a distributional eigenfunction of $\A_\hl$.
\end{lem}

\begin{proof}
By~\eqref{eq:s}, we have $\laplace S = (\psi(\lambda^2) / \lambda^2) (\lambda^2 + \xi^2)$. Since $F = F^* - G$ and both $F^*$ and $G$ are bounded, also $F$ is a bounded function. Hence, $F$ is $\schwartz'$-convolvable with $S \conv R_+$, and $\laplace F(\xi) \cdot (\laplace S(\xi) \laplace R_+(\xi)) = f(\xi) \cdot ((\psi(\lambda^2) / \lambda^2) (\lambda^2 + \xi^2) / \psi_\lambda^\dagger(\xi)) = c_\lambda \psi(\lambda^2) / \lambda^2$ when $\real \xi > 0$. By the exchange formula, $F \conv (S \conv R_+) = (c_\lambda \psi(\lambda^2) / \lambda^2) \delta_0$, which is a distribution supported in $\set{0}$. An application of Lemma~\ref{lem:reduction} completes the proof.
\end{proof}

\begin{rem}
Clearly, also the distributional derivatives (of all orders) of $F$ are distributional eigenfunctions of $\A_\hl$. However, they are not in the $L^\infty(\hl)$ domain of $\A_\hl$.\qed
\end{rem}

In order to apply Lemma~\ref{lem:eigenfunctions}, we need to show that $F$ vanishes continuously at~$0$.

\begin{lem}
\label{lem:f0}
The function $F$ has a nonnegative right limit at $0$. If $\psi$ is unbounded, then $F$ vanishes continuously at $0$.
\end{lem}

\begin{proof}
Since $G$ is non-increasing (even completely monotone), the limit of $F(x)$ as $x \to 0^+$ exists. Since $\laplace F = f$, where $f$ is given by~\eqref{eq:lf0}, we have
\formula{
 \lim_{x \to 0^+} F(x) & = \lim_{\xi \to \infty} \xi f(\xi) = \lim_{\xi \to \infty} \frac{c_\lambda \xi \psi_\lambda^\dagger(\xi)}{\lambda^2 + \xi^2} = \lim_{\xi \to \infty} \expr{\frac{c_\lambda \xi^2}{\lambda^2 + \xi^2} \, \frac{\psi_\lambda^\dagger(\xi)}{\xi}} .
}
The above expression is nonnegative, proving the first part. Furthermore, by Proposition~\ref{prop:dagger:prop}(a,c), $\psi_\lambda^\dagger(\xi) / \xi = (\psi_\lambda(\xi) / \xi)^\dagger$. If $\psi$ is unbounded on $[0, \infty)$, then $\psi_\lambda(\xi) / \xi$ tends to $0$ as $\xi \to \infty$ (see~\eqref{eq:psilambda}). Hence, by Proposition~\ref{prop:dagger:prop}(d), $\psi_\lambda^\dagger(\xi) / \xi$ converges to $0$ as $\xi \to \infty$, and therefore $\lim_{x \to 0^+} F(x) = 0$.
\end{proof}

\begin{proof}[Proof of Theorem~\ref{th:eigenfunctions}]
In the remark following~\eqref{eq:switch} it was already observed that $\laplace F(\xi) = f(\xi)$ coincides with the right-hand side of the formula~\eqref{eq:lf} for $\laplace F_\lambda$, so that $F = F_\lambda$. We will show that $F_\lambda$ satisfies also the remaining conditions of Theorem~\ref{th:eigenfunctions}.

By Corollary~\ref{cor:f}, $F_\lambda = F^* - G$ has the form~\eqref{eq:f} with completely monotone $G_\lambda(x) = G(x)$, and $G_\lambda$ is the Laplace transform of a finite measure $\gamma_\lambda$ on $(0, \infty)$. By Lemma~\ref{lem:f}, $F_\lambda$ is the distributional eigenfunction of $\A_\hl$.

If $\psi$ is unbounded, then $F_\lambda = F$ vanishes continuously at $0$ (Lemma~\ref{lem:f0}), and $\lim_{x \to \infty} (F_\lambda(x) - \sin(\lambda x + \thet_\lambda)) = 0$ (because $G_\lambda = \laplace \gamma_\lambda$ tends to $0$ at infinity). Furthermore, the subordinate Brownian motion $X_t$ has the strong Feller property, and $0$ is a regular boundary point of $\hl$ (see Preliminaries). Hence, Lemma~\ref{lem:eigenfunctions} applies: $F_\lambda \in \domain(\A_\hl; L^\infty)$ and $\A_\hl F_\lambda = -\psi(\lambda^2) F_\lambda$. When $\psi$ is bounded, then $\A_\hl$ is a bounded operator on $L^\infty(\hl)$, and $\A_\hl F_\lambda = -\psi(\lambda^2) F_\lambda$ by Proposition~\ref{prop:poisson}. In both cases $F_\lambda$ is an $L^\infty$ eigenfunction of $\A_\hl$.

Note that (see~\cite{bib:d65})
\formula{
 P^\hl_t F(x) - F(x) & = \int_0^t P^\hl_s \A_\hl F(x) ds = -\psi(\lambda^2) \int_0^t P^\hl_s f(x) ds
}
for $x \in \hl$. For a fixed $x$, this integral equation is easily solved, and we obtain $P^\hl_t F(x) = e^{-t \psi(\lambda^2)} F(x)$, that is, $F_\lambda$ is the eigenfunction of $P^\hl_t$.

By Proposition~\ref{prop:residue}, $\thet_\lambda = \Arg \psi_\lambda^\dagger(i \lambda)$. Recall that $\psi_\lambda^\dagger(0) = 1$ (by Proposition~\ref{prop:dagger:prop}(d)), and hence (by representation~\eqref{eq:cbf0} for $\psi_\lambda^\dagger$) $\real \psi_\lambda^\dagger(i \lambda) \ge 1$. It follows that $0 \le \thet_\lambda < \pi/2$, as desired. It remains to prove formula~\eqref{eq:gamma0}.

Suppose that (the holomorphic extension of) $\psi(\xi)$ extends to a continuous function $\psi^+(\xi)$ (not to be confused with $\psi^\dagger(\xi)$) in the region $\imag \xi \ge 0$, and $\psi^+(-\xi) \ne \psi(\lambda^2)$ for all $\xi > 0$. Recall that $\laplace G_\lambda = g$ (with $g$ given by~\eqref{eq:gstar}), and $G_\lambda = \laplace \gamma_\lambda$. Hence, by the remark following~\eqref{eq:stieltjes}, $\gamma_\lambda = \pi^{-1} \tilde{m}_0$, where $\tilde{m}_0$ is the representation measure in~\eqref{eq:stieltjes} for the Stieltjes function $g$. To prove formula~\eqref{eq:gamma0}, it suffices to apply Proposition~\ref{prop:cbf:repr}, as we now describe.

We have
\formula{
 g(\xi) & = f^*(\xi) - f(\xi) = f^*(\xi) - \frac{c_\lambda \psi_\lambda^\dagger(\xi)}{\lambda^2 + \xi^2} \, .
}
By Lemma~\ref{lem:cbfextension}, when $\imag \xi > 0$ and $\xi \ne i \lambda$,
\formula{
 g(\xi) & = f^*(\xi) - \frac{c_\lambda \psi_\lambda(-\xi^2)}{\psi_\lambda^\dagger(-\xi) (\lambda^2 + \xi^2)} = f^*(\xi) - \frac{\psi(\lambda^2)}{\lambda^2 \psi_\lambda^\dagger(-\xi) (\psi(\lambda^2) - \psi(-\xi^2))} \, .
}
Hence, $g$ extends to a continuous function $g^+$ in the region $\imag \xi \ge 0$, and for $\xi > 0$,
\formula{
 g^+(-\xi) & = f^*(-\xi) - \frac{c_\lambda \psi(\lambda^2)}{\lambda^2 \psi_\lambda^\dagger(\xi) (\psi(\lambda^2) - \psi^+(-\xi^2))} \, .
}
Note that $f^*(-\xi)$ and $\psi_\lambda^\dagger(\xi)$ are real for $\xi > 0$. By Proposition~\ref{prop:cbf:repr}, $(-\imag g^+(-\xi)) d\xi$ is the measure $\tilde{m}_0$ in the representation~\eqref{eq:stieltjes} of the Stieltjes function $g$. Since $\gamma_\lambda = \pi^{-1} \tilde{m}_0$, $c_\lambda = \sqrt{\lambda^4 \psi'(\lambda^2) / \psi(\lambda^2)}$ and $\psi_\lambda(\lambda^2) = \psi(\lambda^2) / (\lambda^2 \psi'(\lambda^2))$, we obtain
\formula{
 \gamma_\lambda(d\xi) = \frac{-\imag g^+(-\xi)}{\pi} \, d\xi & = \frac{1}{\pi} \frac{\sqrt{\psi'(\lambda^2) \psi(\lambda^2)}}{\psi_\lambda^\dagger(\xi)} \, \imag \frac{1}{\psi(\lambda^2) - \psi^+(-\xi^2)} \, d\xi \\
 & = \frac{1}{\pi} \frac{\sqrt{\psi_\lambda(\lambda^2)}}{\psi_\lambda^\dagger(\xi)} \, \imag \frac{\lambda \psi'(\lambda^2)}{\psi(\lambda^2) - \psi^+(-\xi^2)} \, d\xi .
}
An application of~\eqref{eq:switch} completes the proof.
\end{proof}

\begin{rem}
Using~\eqref{eq:switch}, formulae of Theorem~\ref{th:eigenfunctions} can be written in a more concise form. We record them for easier reference:
\formula[]{
\label{eq:theta0}
 \thet_\lambda & = \frac{1}{\pi} \int_0^\infty \frac{\lambda}{\lambda^2 - \zeta^2} \, \log \frac{\psi_\lambda(\zeta^2)}{\psi_\lambda(\lambda^2)} \, d\zeta = \Arg \psi_\lambda^\dagger(i \lambda) ; \\
\label{eq:lf00}
 \laplace F_\lambda(\xi) & = \frac{\lambda^2}{\lambda^2 + \xi^2} \, \frac{\psi_\lambda^\dagger(\xi)}{\sqrt{\psi_\lambda(\lambda^2)}} \, ; \\
\label{eq:gamma00}
 \gamma_\lambda(d\xi) & = \frac{1}{\pi} \frac{\sqrt{\psi_\lambda(\lambda^2)}}{\psi_\lambda^\dagger(\xi)} \, \imag \frac{\lambda \psi'(\lambda^2)}{\psi(\lambda^2) - \psi^+(-\xi^2)} \, d\xi .
}
\qed
\end{rem}

\begin{rem}
\label{rem:gamma}
The formula~\eqref{eq:gamma0} for $\gamma_\lambda$ remains true for general CBF $\psi$, as long as it is understood in the distributional sense. More precisely, it can be proved that $\gamma_\lambda$ is the limit of
\formula{
 \frac{1}{\pi} \frac{\sqrt{\psi_\lambda(\lambda^2)}}{\psi_\lambda^\dagger(\xi)} \expr{\imag \frac{\lambda \psi'(\lambda^2)}{\psi(\lambda^2) - \psi(-\xi^2 + i \eps)}} \ind_{\hl}(\xi) d\xi
}
as $\eps \to 0^+$. We omit the details.\qed
\end{rem}

\begin{rem}
In order to extend Theorem~\ref{th:eigenfunctions} to more general symmetric L{\'e}vy processes, one needs two ingredients: formula~\eqref{eq:lf} must define a function $F_\lambda$, and the convolution $F \conv (R_+ \conv S) \conv R_-$ needs to be associative (that is, some regularity of $F$ and $R_\pm$ is needed). This seems to be problematic, even for general subordinate Brownian motions.\qed
\end{rem}

%
%

\subsection{Further properties of eigenfunctions}

We begin with a technical result, useful for the computation and estimates of $\thet_\lambda$.

\begin{prop}
We have
\formula[eq:integral]{
 \frac{1}{\pi} \int_0^1 \frac{-\log z}{1 - z^2} \, dz & = \frac{\pi}{8} \, .
}
\end{prop}

\begin{proof}
By series expansion and integration by parts,
\formula{
 \int_0^1 \frac{-\log z}{1 - z^2} \, dz & = \sum_{n = 0}^\infty \int_0^1 z^{2n} (-\log z) dz = \sum_{n = 0}^\infty \frac{1}{2n + 1} \int_0^1 z^{2n} dz \\
 & = \sum_{n = 0}^\infty \frac{1}{(2n + 1)^2} = \sum_{k = 1}^\infty \frac{1}{k^2} - \sum_{k = 1}^\infty \frac{1}{(2 k)^2} = \frac{3}{4} \sum_{k = 1}^\infty \frac{1}{k^2} = \frac{3}{4} \, \frac{\pi^2}{6} \, ,
}
as desired.
\end{proof}

The following formula for $\thet_\lambda$ is often more convenient than~\eqref{eq:theta}.

\begin{prop}
For $\lambda > 0$,
\formula[eq:theta2]{
 \thet_\lambda & = \frac{1}{\pi} \int_0^1 \frac{1}{1 - z^2} \, \log \frac{\psi_\lambda(\lambda^2 / z^2)}{\psi_\lambda(\lambda^2 z^2)} \, dz .
}
\end{prop}

\begin{proof}
By substituting $\zeta = \lambda \xi$ in~\eqref{eq:theta}, we obtain
\formula{
 \thet_\lambda & = -\frac{1}{\pi} \int_0^\infty \frac{1}{1 - \xi^2} \, \log \frac{\lambda^2 \psi'(\lambda^2)(1 - \xi^2)}{\psi(\lambda^2) - \psi(\lambda^2 \xi^2)} \, d\xi .
}
A substitution $\xi = z$ in the integral over $(0, 1)$ and $\xi = 1/z$ in the integral over $(1, \infty)$ yields
\formula{
 \thet_\lambda & = \frac{1}{\pi} \int_0^1 \frac{1}{1 - z^2} \, \log \frac{\psi(\lambda^2) - \psi(\lambda^2 z^2)}{(\psi(\lambda^2 / z^2) - \psi(\lambda^2)) z^2} \, dz \\ & = \frac{1}{\pi} \int_0^1 \frac{1}{1 - z^2} \, \log \frac{\psi_\lambda(\lambda^2 / z^2)}{\psi_\lambda(\lambda^2 z^2)} \, dz .
\qedhere
}
\end{proof}

\begin{prop}
\label{prop:diff:theta}
The function $\thet_\lambda$ is differentiable in $\lambda > 0$, and~\eqref{eq:theta2} can be differentiated in $\lambda$ under the integral sign.
\end{prop}

\begin{proof}
Let $h(s, z) = (\log (\psi(s) - \psi(s z^2)) - \log (\psi(s / z^2) - \psi(s)) - 2 \log z$ for $z \in (0, 1)$ and $s > 0$. Then $h(\lambda^2, z) = \log (\psi_\lambda(\lambda^2 z^2) / \psi_\lambda(\lambda^2 / z^2))$. Therefore, it suffices to prove that $(1 - z^2)^{-1} |(\partial / \partial s) h(s, z)|$ is integrable over rectangles $z \in (0, 1)$, $s \in (s_1, s_2)$, where $0 < s_1 < s_2$.

We have
\formula{
 \frac{\partial h(s, z)}{\partial s} & = \frac{1}{s} \, \frac{s \psi'(s) - s z^2 \psi'(s z^2)}{\psi(s) - \psi(s z^2)} - \frac{1}{s} \, \frac{s z^{-2} \psi'(s / z^2) - s \psi'(s)}{\psi(s / z^2) - \psi(s)} \, .
}
Using twice Cauchy's mean value theorem (for two functions $t \psi'(t)$ and $\psi(t)$), we obtain that for some $t_1 \in [s z^2, s]$ and $t_2 \in [s, s / z^2]$,
\formula{
 \frac{\partial h(s, z)}{\partial s} & = \frac{1}{s} \frac{\psi'(t_1) + t_1 \psi''(t_1)}{\psi'(t_1)} - \frac{1}{s} \frac{\psi'(t_2) + t_2 \psi''(t_2)}{\psi'(t_2)} \, .
}
Since $t \psi''(t) / \psi'(t)$ is bounded (by Proposition~\ref{prop:cbf:ests}(b)), the above expression is uniformly bounded in $s \in (s_1, s_2)$ and $z \in (0, 1)$. Furthermore, when $s \in (s_1, s_2)$ and $z \in (1/2, 1)$, by the usual mean value theorem,
\formula{
 \abs{\frac{\partial h(s, z)}{\partial s}} & \le \frac{t_2 - t_1}{s} \, \sup_{t \in [s z^2, s / z^2]} \abs{\frac{d}{d t} \frac{\psi'(t) + t \psi''(t)}{\psi'(t)}} \\
 & \le \expr{\frac{1}{z^2} - z^2} \sup_{t \in [s_1 / 4, 4 s_2]} \abs{\frac{d}{d t} \frac{\psi'(t) + t \psi''(t)}{\psi'(t)}} \le C (1 - z^2)
}
for some $C$ depending on $s_1$ and $s_2$. The proof is complete.
\end{proof}

\begin{prop}
\label{prop:theta}
We have
\formula[eq:theta:bound]{
 \thet_\lambda & \le \sup_{\xi \in \R} \frac{\xi \psi_\lambda'(\xi)}{\psi_\lambda(\xi)} \cdot \frac{\pi}{2} \, .
}
\end{prop}

\begin{proof}
Note that $(d / ds) \log \psi_\lambda(\lambda^2 s) = \lambda^2 \psi_\lambda'(\lambda^2 s) / \psi_\lambda(\lambda^2 s)$. Hence, by~\eqref{eq:theta2},
\formula{
 \thet_\lambda & = \frac{1}{\pi} \int_0^1 \frac{1}{1 - z^2} \expr{\int_{z^2}^{1 / z^2} \frac{\lambda^2 \psi_\lambda'(\lambda^2 s)}{\psi_\lambda(\lambda^2 s)} \, ds} dz .
}
It follows that
\formula{
 \thet_\lambda & \le \expr{\sup_{\xi \in \R} \frac{\xi \psi_\lambda'(\xi)}{\psi_\lambda(\xi)}} \frac{1}{\pi} \int_0^1 \frac{1}{1 - z^2} \expr{\int_{z^2}^{1 / z^2} \frac{1}{s} \, ds} dz . 
}
Finally, by~\eqref{eq:integral},
\formula{
 \frac{1}{\pi} \int_0^1 \frac{1}{1 - z^2} \expr{\int_{z^2}^{1 / z^2} \frac{1}{s} \, ds} dz & = \frac{4}{\pi} \int_0^1 \frac{-\log z}{1 - z^2} \, dz = \frac{\pi}{2} \, .
\qedhere
}
\end{proof}

\begin{cor}
If $\thet_\lambda = 0$ for some $\lambda > 0$, then $\psi$ is linear, $X_t$ is a Brownian motion, and $\thet_\lambda = 0$ for all $\lambda > 0$.
\end{cor}

\begin{proof}
Since $\psi_\lambda$ is increasing, for $z \in (0, 1)$ we have $\psi_\lambda(\lambda^2 / z^2) \ge \psi_\lambda(\lambda^2 z^2)$, and it is easily checked that equality holds if and only if $\psi$ is linear. Hence, by~\eqref{eq:theta2}, $\thet_\lambda \ge 0$, and if $\psi$ is not linear, then $\thet_\lambda > 0$.
\end{proof}

Now we turn to the properties of $F_\lambda$.

\begin{prop}
\label{prop:lf:0}
We have
\formula[eq:lf:0]{
 \lim_{\xi \to 0^+} \laplace F_\lambda(\xi) & = \sqrt{\frac{\psi'(\lambda^2)}{\psi(\lambda^2)}} \, , && \lambda > 0 .
}
\end{prop}

\begin{proof}
The result follows from $\psi_\lambda^\dagger(0) = 1$ (see Proposition~\ref{prop:dagger:prop}(d)), $\laplace F_\lambda(\xi) = c_\lambda \psi_\lambda^\dagger(\xi) / (\lambda^2 + \xi^2)$, and~\eqref{eq:clambda}.
\end{proof}

\begin{lem}
\label{lem:gest}
For all $\lambda, x > 0$, $0 \le G_\lambda(x) \le \sin \thet_\lambda$. Furthermore,
\formula[eq:gint]{
 \int_0^\infty G_\lambda(x) dx & = \frac{1}{\lambda} \expr{\cos \thet_\lambda - \sqrt{\frac{\lambda^2 \psi'(\lambda^2)}{\psi(\lambda^2)}}} , && \lambda > 0 .
}
\end{lem}

\begin{proof}
Clearly, $0 \le G_\lambda(x) \le G_\lambda(0)$, and by Lemma~\ref{lem:f0}, we have $0 \le \lim_{x \to 0^+} F_\lambda(x) = \sin \thet_\lambda - G_\lambda(0)$. This proves the first statement. Furthermore, $\laplace G_\lambda(\xi) = g(\xi) = f^*(\xi) - f(\xi) = f^*(\xi) - \laplace F_\lambda(\xi)$. By~\eqref{eq:clambda}, \eqref{eq:fstar}, \eqref{eq:lf:0} and Proposition~\ref{prop:residue},
\formula{
 \int_0^\infty G_\lambda(x) dx & = \lim_{\xi \to 0^+} \laplace G_\lambda(\xi) = f^*(0) - \lim_{\xi \to 0^+} \laplace F_\lambda(\xi) \\
 & = \frac{c_\lambda \real \psi_\lambda^\dagger(i \lambda)}{\lambda^2} - \sqrt{\frac{\psi'(\lambda^2)}{\psi(\lambda^2)}} = \frac{\cos \thet_\lambda}{\lambda} - \sqrt{\frac{\psi'(\lambda^2)}{\psi(\lambda^2)}} \, . \qedhere
}
\end{proof}

\begin{prop}
\label{prop:lfest}
We have
\formula[eq:lfest]{
 |\laplace F_\lambda(\xi)| & \le \frac{|\lambda + \xi|}{|\lambda^2 + \xi^2|} , && \lambda > 0 , \, \real \xi > 0 .
}
\end{prop}

\begin{proof}
Since $\laplace F_\lambda(\xi) = c_\lambda \psi_\lambda^\dagger(\xi) / (\lambda^2 + \xi^2)$, by~\eqref{eq:psilambdadaggerest} and~\eqref{eq:clambda} we have
\formula{
 |\laplace F_\lambda(\xi)| & \le \frac{\lambda}{|\lambda^2 + \xi^2|} \abs{1 + \sqrt{\frac{|\psi''(\lambda^2)|}{2 \psi'(\lambda^2)}} \, \xi} .
}
By Proposition~\ref{prop:cbf:ests}(b), $\lambda^2 |\psi''(\lambda^2)| \le 2 \psi'(\lambda^2)$, and~\eqref{eq:lfest} follows.
\end{proof}

\begin{lem}
\label{lem:lfest2}
Suppose that there are $C, \alpha > 0$ such that $\psi'(\xi) \ge C \xi^{\alpha-1}$ for $\xi$ large enough. Then (cf.~\eqref{eq:lfest})
\formula[eq:lfest2]{
 |\laplace F_\lambda(\xi)| & \le \sqrt{\sup_{\zeta > \lambda} \frac{\lambda^{2 - 2\alpha} \psi'(\lambda^2)}{\zeta^{2 - 2\alpha} \psi'(\zeta^2)}} \, \frac{\lambda}{|\lambda^2 + \xi^2|} \abs{1 + \frac{\xi}{\lambda}}^{1 - \alpha}
}
whenever $\lambda > 0$, and $\real \xi > 0$.
\end{lem}

\begin{proof}
Since $\psi$ is concave and increasing, for $\lambda, \xi > 0$ we have
\formula{
 \psi_\lambda(\xi) & = \frac{\psi(\lambda^2)}{\lambda^2} \, \frac{\lambda^2 - \xi}{\psi(\lambda^2) - \psi(\xi)} \le \frac{\psi(\lambda^2)}{\lambda^2} \, \max \expr{\frac{1}{\psi'(\lambda^2)}, \frac{1}{\psi'(\xi)}} .
}
Let $C_\lambda$ be the supremum in~\eqref{eq:lfest2} (finite by the assumption), so that $\psi'(\lambda^2) / \psi'(\xi) \le C_\lambda (\xi / \lambda^2)^{1 - \alpha}$ for all $\xi > \lambda^2$. Then
\formula{
 \psi_\lambda(\xi) & \le \frac{C_\lambda \psi(\lambda^2)}{\lambda^2 \psi'(\lambda^2)} \expr{1 + \frac{\xi}{\lambda^2}}^{1-\alpha} .
}
By Proposition~\ref{prop:explogbound}, we have
\formula[eq:psilambdadaggerest2]{
 \psi_\lambda^\dagger(\xi) & \le \sqrt{\frac{C_\lambda \psi(\lambda^2)}{\lambda^2 \psi'(\lambda^2)}} \abs{1 + \frac{\xi}{\lambda}}^{1 - \alpha} , && \lambda > 0 , \, \real \xi > 0 .
}
The lemma follows from $\laplace F_\lambda(\xi) = c_\lambda \psi_\lambda^\dagger(\xi) / (\xi^2 + \lambda^2)$ and~\eqref{eq:clambda}.
\end{proof}

\begin{lem}
\label{lem:power}
Suppose that there are $C, \alpha > 0$ such that $\psi'(\xi) \ge C \xi^{\alpha-1}$ for sufficiently large $\xi$. Then for each $\lambda > 0$, the function $F_\lambda$ is H{\"o}lder continuous with any exponent less than $\alpha$. More precisely, there is an absolute constant $c$ such that
\formula[eq:power]{
 \frac{|F_\lambda(x_1) - F_\lambda(x_2)|}{|x_1 - x_2|^\eps} & \le \frac{c \lambda^\eps}{\alpha - \eps} \, \sqrt{\sup_{\zeta > \lambda} \frac{\lambda^{2 - 2\alpha} \psi'(\lambda^2)}{\zeta^{2 - 2\alpha} \psi'(\zeta^2)}}
}
for all $\eps \in (0, \alpha)$, $\lambda > 0$ and $x_1, x_2 \in \R$.
\end{lem}

\begin{proof}
We decompose $\laplace F_\lambda = c_\lambda \psi_\lambda^\dagger(\xi) / (\lambda^2 + \xi^2)$ into the difference of $\tilde{f}^*$ and $\tilde{g}$ in a similar way to the decomposition $\laplace F_\lambda = f^* - g$ in the derivation of the formula for $F_\lambda$, but this time $\tilde{f}^*$ and $\tilde{g}$ have a better decay rate at infinity. When $\real \xi > 0$, define
\formula{
 \tilde{f}^*(\xi) & = \frac{c_\lambda}{2 i \lambda} \expr{\frac{\lambda \psi_\lambda^\dagger(i \lambda)}{(\xi - i \lambda)(\xi - i \lambda + \lambda)} - \frac{\lambda \psi_\lambda^\dagger(-i \lambda)}{(\xi + i \lambda)(\xi + i \lambda + \lambda)}} , \\
 \tilde{g}(\xi) & = \frac{c_\lambda}{2 i \lambda} \biggl(\frac{\lambda \psi_\lambda^\dagger(i \lambda) - (\xi - i \lambda + \lambda) \psi_\lambda^\dagger(\xi)}{(\xi - i \lambda)(\xi - i \lambda + \lambda)} \\ & \hspace{3cm} - \frac{\lambda \psi_\lambda^\dagger(-i \lambda) - (\xi + i \lambda + \lambda) \psi_\lambda^\dagger(\xi)}{(\xi + i \lambda)(\xi + i \lambda + \lambda)}\biggr) .
}
Note that $\lambda (\xi \mp i \lambda)^{-1} (\xi \mp i \lambda + \lambda)^{-1}$ is the Laplace transform of the function $(1 - e^{-\lambda x}) e^{\pm i \lambda x} \ind_{\hl}(x)$, and $c_\lambda \psi_\lambda^\dagger(\pm i \lambda) = \lambda e^{\pm i \thet_\lambda}$. Therefore, $\tilde{f}^*$ is the Laplace transform of
\formula{
 \tilde{F}^*(x) & = (1 - e^{-\lambda x}) \sin(\lambda x + \thet_\lambda) \ind_{\hl}(x) .
}
Note that $\tilde{F}^*$ is Lipschitz continuous on $\R$. In fact, for $x > 0$ we have
\formula{
 |(\tilde{F}^*)'(x)| & = |\lambda e^{-\lambda x} \sin(\lambda x + \thet_\lambda) + \lambda (1 - e^{-\lambda x}) \cos(\lambda x + \thet_\lambda)| \\
 & \le \sqrt{\lambda^2 e^{-2\lambda x} + \lambda^2 (1 - e^{-\lambda x})^2} \le \lambda ,
}
so that the Lipschitz constant of $\tilde{F}^*$ is not greater than $\lambda$.

The function $\tilde{g}$ extends continuously to $i \R$. Denote the summands in the parentheses in the definition of $\tilde{g}$ by $\tilde{g}_+$ and $\tilde{g}_-$, so that $\tilde{g}(\xi) = (c_\lambda / (2 i \lambda)) (\tilde{g}_+(\xi) - \tilde{g}_-(\xi))$. We first consider $\tilde{g}_+(i \xi)$ for $\xi \in [\lambda/2, 2\lambda]$. We have
\formula{
 |\tilde{g}_+(i \xi)| & = \abs{\frac{\lambda \psi_\lambda^\dagger(i \lambda) - (i \xi - i \lambda + \lambda) \psi_\lambda^\dagger(i \xi)}{(i \xi - i \lambda)(i \xi - i \lambda + \lambda)}} \\ & \le \frac{\lambda |\psi_\lambda^\dagger(i \lambda) - \psi_\lambda^\dagger(i \xi)| + |\xi - \lambda| |\psi_\lambda^\dagger(i \xi)|}{|\xi - \lambda| \sqrt{(\xi - \lambda)^2 + \lambda^2}} \\
 & \le \frac{1}{\sqrt{(\xi - \lambda)^2 + \lambda^2}} \expr{\lambda \sup_{\zeta \in [\lambda, \xi]} |(\psi_\lambda^\dagger)'(i \zeta)| + |\psi_\lambda^\dagger(i \xi)|} ;
}
when $\lambda > \xi$, then we understand that $[\lambda, \xi]$ denotes the interval $[\xi, \lambda]$. By Proposition~\ref{prop:cbf:ests}(c,d), $|(\psi_\lambda^\dagger)'(i \zeta)| \le 2 \psi_\lambda^\dagger(\zeta) / \zeta$ and $|\psi_\lambda^\dagger(i \zeta)| \le 2 \psi_\lambda^\dagger(\zeta)$ for $\zeta \in \R$; the latter inequality is used frequently in the remainder of the proof. Using also $\sqrt{(\xi - \lambda)^2 + \lambda^2} \ge \lambda$, we obtain
\formula{
 |\tilde{g}_+(i \xi)| & \le \frac{1}{\lambda} \expr{\lambda \sup_{\zeta \in [\lambda/2, 2\lambda]} \frac{2 \psi_\lambda^\dagger(\zeta)}{\zeta} + 2 \psi_\lambda^\dagger(\xi)} \le \frac{6 \psi_\lambda^\dagger(2 \lambda)}{\lambda} \, .
}
By~\eqref{eq:psilambdadaggerest2}, $c_\lambda \psi_\lambda^\dagger(2\lambda) \le 3^{1-\alpha} \lambda \sqrt{C_\lambda} \le 3 \lambda \sqrt{C_\lambda}$, where $C_\lambda$ is the supremum in~\eqref{eq:lfest2}. Hence,
\formula{
 |c_\lambda \tilde{g}_+(i \xi)| & \le 18 \sqrt{C_\lambda} \, , && \lambda > 0 , \, \xi \in [\lambda/2, 2\lambda] .
}
When $\xi \in [-2\lambda, \lambda/2]$, we simply have $|\xi - \lambda| \ge \lambda / 2$, $\sqrt{(\xi - \lambda)^2 + \lambda^2} \ge \lambda$, so that
\formula{
 |\tilde{g}_+(i \xi)| & \le \frac{\lambda |\psi_\lambda^\dagger(i \lambda)|}{|\xi - \lambda| \sqrt{(\xi - \lambda)^2 + \lambda^2}} + \frac{|\psi_\lambda^\dagger(i \xi)|}{|\xi - \lambda|} \le \frac{4 \psi_\lambda^\dagger(\lambda) + 4 \psi_\lambda^\dagger(2\lambda)}{\lambda} \, .
}
Using~\eqref{eq:psilambdadaggerest2}, we obtain that $c_\lambda \psi_\lambda^\dagger(\lambda) \le 2^{1-\alpha} \lambda \sqrt{C_\lambda} \le 2 \lambda \sqrt{C_\lambda}$. Together with $c_\lambda \psi_\lambda^\dagger(2 \lambda) \le 3 \lambda \sqrt{C_\lambda}$, this gives
\formula{
 |c_\lambda \tilde{g}_+(i \xi)| & \le 20 \sqrt{C_\lambda} \, , && \lambda > 0 , \, \xi \in [-2\lambda, \lambda/2] .
}
Similar estimates hold for $c_\lambda \tilde{g}_-$. It follows that
\formula{
 |\tilde{g}(i \xi)| & \le \frac{c_\lambda (|\tilde{g}_+(\xi)| + |\tilde{g}_-(\xi)|)}{2 \lambda} \le \frac{20 \sqrt{C_\lambda}}{\lambda} , && \lambda > 0 , \, \xi \in [-2\lambda, 2\lambda] .
}
The estimate for $\xi \in \R \setminus [-2\lambda, 2\lambda]$ is obtained using $|\tilde{g}(i \xi)| \le |\laplace F_\lambda(i \xi)| + |\tilde{f}^*(i \xi)|$ (here $\laplace F_\lambda(i \xi)$ denotes the continuous boundary limit). By the inequalities $|\xi - \lambda| \ge |\xi| / 2$, $\sqrt{(\xi - \lambda)^2 + \lambda^2} \ge |\xi| / 2$, we have
\formula{
 |\tilde{f}^*(i \xi)| & \le \frac{c_\lambda}{2} \expr{\frac{|\psi_\lambda^\dagger(i \lambda)|}{|\xi - \lambda| \sqrt{(\xi - \lambda)^2 + \lambda^2}} + \frac{|\psi_\lambda^\dagger(-i \lambda)|}{|\xi + \lambda| \sqrt{(\xi + \lambda)^2 + \lambda^2}}} \\ & \le \frac{8 c_\lambda \psi_\lambda^\dagger(\lambda)}{\xi^2} \le \frac{16 \lambda \sqrt{C_\lambda}}{\xi^2} \, .
}
Hence, using also~\eqref{eq:lfest2} and the inequality $\xi^2 - \lambda^2 \ge (|\xi| + \lambda)^2 / 3$, for $\xi \in \R \setminus [-2\lambda, 2\lambda]$ we obtain
\formula{
 |\tilde{g}(i \xi)| & \le |\laplace F_\lambda(i \xi)| + |\tilde{f}^*(i \xi)| \le \frac{\lambda \sqrt{C_\lambda} \, (1 + |\xi| / \lambda)^{1 - \alpha}}{\xi^2 - \lambda^2} + \frac{16 \lambda \sqrt{C_\lambda}}{\xi^2} \\
 & \le \frac{3 \sqrt{C_\lambda}}{\lambda (1 + |\xi| / \lambda)^{1 + \alpha}} + \frac{16 \sqrt{C_\lambda}}{\lambda (|\xi| / \lambda)^2} \le \frac{20 \sqrt{C_\lambda}}{\lambda} \expr{\frac{\lambda}{|\xi|}}^{1 + \alpha} .
}
This way we have estimated $|\tilde{g}(i \xi)|$ for all $\xi \in \R$. The function $\tilde{g}$ is the Laplace transform of $\tilde{G}$, and $\tilde{G}$ is the inverse Fourier transform of $\tilde{g}(-i \xi)$ ($\xi \in \R$). Fix $\eps \in (0, \alpha)$. We have
\formula{
 |\tilde{G}(x_1) - \tilde{G}(x_2)| & = \frac{1}{2 \pi} \abs{\int_{-\infty}^\infty (e^{-i x_1 \xi} - e^{-i x_2 \xi}) \tilde{g}(-i\xi) d\xi} \\
 & \le \frac{1}{2 \pi} \, \sup_{\xi \in \R} \frac{|e^{-i x_1 \xi} - e^{-i x_2 \xi}|}{|\xi|^\eps} \int_{-\infty}^\infty \xi^\eps |\tilde{g}(-i\xi)| d\xi .
}
Since $|e^{-i x_1 \xi} - e^{-i x_2 \xi}| \le \min(2, |\xi| |x_1 - x_2|)$, the supremum is at most $2^{1-\eps} |x_1 - x_2|^\eps \le 2 |x_1 - x_2|^\eps$. Furthermore,
\formula{
 \int_{-\infty}^\infty \xi^\eps |\tilde{g}(-i\xi)| d\xi & \le \frac{20 \sqrt{C_\lambda}}{\lambda} \expr{2 \int_0^{2 \lambda} \xi^\eps d\xi + 2 \int_{2 \lambda}^\infty \frac{\lambda^{1 + \alpha}}{\xi^{1 + \alpha - \eps}} \, d\xi} \\
 & = 20 \sqrt{C_\lambda} \expr{\frac{2^{2 + \eps} \lambda^\eps}{1 + \eps} + \frac{2^{1 - \alpha + \eps} \lambda^\eps}{\alpha - \eps}} \\ & = 20 \sqrt{C_\lambda} \expr{8 + \frac{2}{\alpha - \eps}} \lambda^\eps \le \frac{200 \sqrt{C_\lambda} \lambda^\eps}{\alpha - \eps} \, .
}
In particular, $\tilde{G}$ is H{\"o}lder continuous of order $\eps$, with H{\"o}lder constant at most $(200 / \pi) \lambda^\eps \sqrt{C_\lambda} / (\alpha - \eps)$. The lemma is proved.
\end{proof}

\begin{cor}
\label{cor:power}
Suppose that for some $\eps > 0$,
\formula[eq:fpt:a4]{
 \limsup_{\xi \to 0^+} \frac{\xi |\psi''(\xi)|}{\psi'(\xi)} & < 1 - \eps , & \limsup_{\xi \to \infty} \frac{\xi |\psi''(\xi)|}{\psi'(\xi)} & < 1 - \eps .
}
(see~\eqref{eq:fpt:a3}). Then there is $C > 0$ (depending on $\eps$) such that
\formula{
 |F_\lambda(x_1) - F_\lambda(x_2)| & \le C \lambda^\eps |x_1 - x_2|^\eps , && \lambda > 0, \, x_1, x_2 \in \R ,
}
and
\formula{
 |\laplace F_\lambda(\xi)| \le C \, \frac{\lambda}{\lambda^2 + \xi^2} \abs{1 + \frac{\xi}{\lambda}}^{1 - \eps} , && \lambda > 0 , \, \real \xi > 0 .
}
\end{cor}

Note that in general, $\xi |\psi''(\xi)| \le 2 \psi'(\xi)$, see Proposition~\ref{prop:cbf:ests}(b).

\begin{proof}
Choose $\alpha > \eps$ small enough, so that~\eqref{eq:fpt:a4} holds with $\eps$ replaced by $\alpha$. Let $h(\zeta) = \zeta^{1 - \alpha} \psi'(\zeta)$. For $\xi \ge \lambda > 0$ we have
\formula{
 \frac{\lambda^{2 - 2\alpha} \psi'(\lambda^2)}{\xi^{2 - 2\alpha} \psi'(\xi^2)} & = \frac{h(\lambda^2)}{h(\xi^2)} = \exp \expr{-\int_{\lambda^2}^{\xi^2} \frac{h'(\zeta)}{h(\zeta)} \, d\zeta} \\ & = \exp \expr{-\int_{\lambda^2}^{\xi^2} \expr{\frac{1 - \alpha}{\zeta} - \frac{|\psi''(\zeta)|}{\psi'(\zeta)}} d\zeta} .
}
The integrand on the right-hand side is positive for $\zeta \in (0, C_1) \cup (C_2, \infty)$ for some $C_1, C_2 > 0$. Hence the above expression is bounded uniformly in $\lambda$ and $\xi$, $0 < \lambda \le \xi$. The result follows by Lemmas~\ref{lem:lfest2} and~\ref{lem:power}.
\end{proof}

Recall that a function $f$ is said to be \emph{regularly varying} of order $\alpha > 0$ at $\infty$ or at $0$, if for all $k > 0$, $\lim_{x \to \infty} f(k x) / f(x) = k^\alpha$ or $\lim_{x \to 0^+} f(k x) / f(x) = k^\alpha$, respectively. 

\begin{cor}
\label{cor:power2}
Suppose that $\psi$ is regularly varying of order $\alpha_\infty > 0$ at $\infty$ and regularly varying of order $\alpha_0 > 0$ at $0$. Then for all $\eps < \min(\alpha_0, \alpha_\infty)$, the hypothesis of Corollary~\ref{cor:power} holds true.
\end{cor}

\begin{proof}
Let $\eps < \alpha < \min(\alpha_0, \alpha_\infty)$. Since $\psi$ is a regularly varying Bernstein function, also $\psi'$ and $\psi''$ are regularly varying at $0$ and $\infty$ (see~\cite{bib:bgt87}). Hence, if we define $h(\zeta) = \zeta^{1 - \alpha} \psi'(\zeta)$, then $h$ is regularly varying of order $\alpha_0 - \alpha > 0$ at zero, and regularly varying of order $\alpha_\infty - \alpha > 0$ at $\infty$, and a similar statement with orders decreased by $1$ is true for $h'$. It follows that the function $\zeta h'(\zeta) / h(\zeta)$ (continuous on $\hl$) has positive limits at $0$ and~$\infty$, and we can now repeat the proof of Corollary~\ref{cor:power}.
\end{proof}

When $\psi$ is regularly varying at infinity, then the behavior of $F_\lambda$ near $0$ for a fixed $\lambda > 0$ has a very simple description. We write $f(x) \sim g(x)$ as $x \to 0^+$ or $x \to \infty$ when $f(x) / g(x)$ converges to $1$ as $x \to 0^+$ or $x \to \infty$, respectively.

\begin{lem}
\label{lem:regular}
If $\psi$ is unbounded in $\hl$ and $\psi(\xi)$ is regularly varying at infinity of order $\alpha \in [0, 1]$, then
\formula[]{
 \label{eq:lfregular} & \laplace F_\lambda(\xi) \sim \sqrt{\frac{\lambda^2 \psi'(\lambda^2)}{\xi^2 \psi(\xi^2)}} && \text{as $\xi \to \infty$,} \\
 \label{eq:fregular} & F_\lambda(x) \sim \frac{1}{\Gamma(1 + \alpha)} \sqrt{\frac{\lambda^2 \psi'(\lambda^2)}{\psi(x^{-2})}} && \text{as $x \to 0^+$.}
}
\end{lem}

\begin{proof}
Consider an auxiliary function:
\formula{
 h(\xi, \zeta) & = \frac{\lambda^2 \psi(\xi^2)}{\psi(\lambda^2) \zeta^{2 - 2\alpha} \xi^2} \, \psi_\lambda(\xi^2 \zeta^2) \, , && \xi, \zeta > 0 .
}
Since $\psi_\lambda$ is a complete Bernstein function, both $\psi_\lambda(\xi)$ and $\xi / \psi_\lambda(\xi)$ are increasing (Proposition~\ref{prop:cbf:prop}(a)). It follows that for $\xi > 0$, $\zeta \ge 1$, we have $\psi_\lambda(\xi^2) \le \psi_\lambda(\xi^2 \zeta^2) \le \zeta^2 \psi_\lambda(\xi^2)$. Hence
\formula[eq:hest]{
 \frac{\lambda^2}{\psi(\lambda^2) \zeta^{2 - 2\alpha}} \, \frac{\psi(\xi^2) \psi_\lambda(\xi^2)}{\xi^2} & \le h(\xi, \zeta) \le \frac{\lambda^2}{\psi(\lambda^2) \zeta^{-2\alpha}} \, \frac{\psi(\xi^2) \psi_\lambda(\xi^2)}{\xi^2} \, .
}
Note that $\psi(\xi^2) \psi_\lambda(\xi^2) / \xi^2$ has a positive limit $\psi(\lambda^2) / \lambda^2$ as $\xi \to \infty$. Hence, it is bounded above and below by positive constants when $\xi \in (1, \infty)$. It follows that for some $C > 0$ (depending on $\lambda$),
\formula{
 |\log h(\xi, \zeta)| & \le C (1 + |\log \zeta|) && \xi > 1, \, \zeta \ge 1 .
}
For $\zeta \in (0, 1]$, the inequalities in~\eqref{eq:hest} are reversed, and we obtain the same bound for $|\log h(\xi, \zeta)|$. Since $\psi$ is regularly varying of order $\alpha$ and unbounded,
\formula{
 \lim_{\xi \to \infty} h(\xi, \zeta) & = \lim_{\xi \to \infty} \frac{\psi(\xi^2)}{\zeta^{2 - 2\alpha} \xi^2} \, \frac{\lambda^2 - \xi^2 \zeta^2}{\psi(\lambda^2) - \psi(\xi^2 \zeta^2)} \\
 & = \lim_{\xi \to \infty} \frac{\zeta^{2 \alpha} \psi(\xi^2)}{\psi(\xi^2 \zeta^2)} \, \frac{\lambda^2 / (\xi \zeta)^2 - 1}{\psi(\lambda^2) / \psi(\xi^2 \zeta^2) - 1} = 1
}
for all $\zeta > 0$. Hence, by dominated convergence,
\formula{
 & \lim_{\xi \to \infty} \expr{\frac{1}{\pi} \int_0^\infty \frac{\log \psi_\lambda(\xi^2 \zeta^2)}{1 + \zeta^2} \, d\zeta - \frac{1}{\pi} \int_0^\infty \frac{1}{1 + \zeta^2} \, \log \frac{\psi(\lambda^2) \zeta^{2 - 2\alpha} \xi^2}{\lambda^2 \psi(\xi^2)} \, d\zeta} \\
 & \qquad = \lim_{\xi \to \infty} \int_0^\infty \frac{\log h(\xi, \zeta)}{1 + \zeta^2} \, d\zeta = 0 .
}
Since the integral of $(\log \zeta) / (1 + \zeta^2)$ over $(0, \infty)$ is zero, we have
\formula{
 \frac{1}{\pi} \int_0^\infty \frac{1}{1 + \zeta^2} \, \log \frac{\psi(\lambda^2) \zeta^{2 - 2\alpha} \xi^2}{\lambda^2 \psi(\xi^2)} \, d\zeta & = \frac{1}{2} \, \log \frac{\psi(\lambda^2) \xi^2}{\lambda^2 \psi(\xi^2)} \, .
}
These two formulae yield that (see~\eqref{eq:dagger})
\formula{
 \psi_\lambda^\dagger(\xi) & = \exp\expr{\frac{1}{\pi} \int_0^\infty \frac{\log \psi_\lambda(\xi^2 \zeta^2)}{1 + \zeta^2} \, d\zeta} \sim \sqrt{\frac{\psi(\lambda^2) \xi^2}{\lambda^2 \psi(\xi^2)}} && \text{as $\xi \to \infty$.}
}
Since $\laplace F_\lambda = c_\lambda \psi_\lambda^\dagger(\xi) / (\xi^2 + \lambda^2)$, this (and~\eqref{eq:clambda}) gives
\formula{
 \laplace F_\lambda(\xi) & \sim \sqrt{\frac{\psi'(\lambda^2)}{\psi(\lambda^2)}} \, \frac{\lambda^2}{\xi^2} \, \sqrt{\frac{\psi(\lambda^2) \xi^2}{\lambda^2 \psi(\xi^2)}} && \text{as $\xi \to \infty$,}
}
and~\eqref{eq:lfregular} follows. Since $F_\lambda$ is increasing on the interval $(0, \pi/2 - \thet_\lambda)$ (see~\eqref{eq:f}), formula~\eqref{eq:fregular} follows by Karamata's Tauberian theorem and monotone density theorem (see~\cite{bib:bgt87}).
\end{proof}

%
%

\section{Spectral representation in half-line}
\label{sec:spectral}

In this section we study the $L^2(\hl)$ properties of the operators $P^\hl_t$ and prove Theorems~\ref{th:spectral}, \ref{th:pdt} and~\ref{th:fpt}. Let us define the operator $\Pi^*$ by the formula
\formula[eq:pi]{
 \Pi^* f(x) & = \int_0^\infty f(\lambda) F_\lambda(x) d\lambda , && x > 0 ,
}
for $f \in C_c(\hl)$. Here $F_\lambda(x) = \sin(\lambda x + \thet_\lambda) - G_\lambda(x)$ is given by Theorem~\ref{th:eigenfunctions}. At least formally, $\Pi^*$ is the adjoint of $\Pi$ defined in~\eqref{eq:pistar}.

In this section, we write $\|f\|_2$ for $\|f\|_{L^2(\hl)}$, and $\tscalar{f, g}$ for $\tscalar{f, g}_{L^2(\hl)}$.

\begin{lem}
\label{lem:pi:bounded}
The operator $\Pi^*$ extends to a bounded operator on $L^2(\hl)$.
\end{lem}

\begin{proof}
We follow the proof of Theorem~3 in~\cite{bib:kkms10}. Let $f \in C_c(\hl)$. Define
\formula{
 \Pi^*_1 f(x) & = \int_0^\infty f(\lambda) \sin(\lambda x + \thet_\lambda) d\lambda = \imag \fourier \tilde{f}(x), && x > 0 ,
}
where $\tilde{f}(\lambda) = e^{i \thet_\lambda} f(\lambda)$. Clearly, $\|\Pi^*_1 f\|_2 \le \|\fourier \tilde{f}\|_{L^2(\R)} = \sqrt{2 \pi} \|f\|_2$. Let
\formula{
 \Pi^*_2 f(x) & = \int_0^\infty f(\lambda) G_\lambda(x) d\lambda , && x > 0 ,
}
so that $\Pi^* = \Pi^*_1 - \Pi^*_2$. We have
\formula{
 \int_0^\infty (\Pi^*_2 f(x))^2 dx & \le \int_0^\infty \int_0^\infty \int_0^\infty |f(\lambda_1)| |f(\lambda_2)| G_{\lambda_1}(x) G_{\lambda_2}(x) dx d\lambda_1 d\lambda_2 .
}
By Lemma~\ref{lem:gest}, when $0 < \lambda_1 \le \lambda_2$,
\formula{
 \int_0^\infty G_{\lambda_1}(x) G_{\lambda_2}(x) dx & \le \sin \thet_{\lambda_1} \int_0^\infty G_{\lambda_2}(x) dx \\ & \hspace*{-3em} = \frac{\sin \thet_{\lambda_1}}{\lambda_2} \expr{\cos \thet_{\lambda_2} - \sqrt{\frac{\lambda_2^2 \psi'(\lambda_2^2)}{\psi(\lambda_2^2)}}} \le \frac{1}{\lambda_2} \, .
}
By symmetry, for all $\lambda_1, \lambda_2 > 0$,
\formula{
 \int_0^\infty G_{\lambda_1}(x) G_{\lambda_2}(x) dx & \le \frac{1}{\max(\lambda_1, \lambda_2)} \, .
}
Hence, using Hardy-Hilbert's inequality (see e.g.~\cite{bib:h52}) in the last step,
\formula{
 \norm{\Pi^*_2 f}_2^2 = \int_0^\infty (\Pi^*_2 f(x))^2 dx & \le \int_0^\infty \int_0^\infty \frac{|f(\lambda_1)| |f(\lambda_2)|}{\max(\lambda_1, \lambda_2)} \, d\lambda_1 d\lambda_2 < 4 \norm{f}_2^2 .
}
The lemma follows.
\end{proof}

The extension of $\Pi^*$ to $L^2(\hl)$ is denoted by the same symbol. By Lemma~\ref{lem:pi:bounded}, we see that $\Pi^*$ is indeed the adjoint of $\Pi$ defined in~\eqref{eq:pistar}.

\begin{lem}
\label{lem:pi:spec}
We have $\tscalar{\Pi^* f, \Pi^* g} = (\pi / 2) \tscalar{f, g}$ for $f, g \in L^2(\hl)$. Furthermore, for $f \in L^2(\hl)$ such that $e^{-t \psi(\lambda^2)} f(\lambda)$ is integrable in $\lambda > 0$, we have
\formula[eq:pdt:spec]{
 P^\hl_t \Pi^* f(x) & = \int_0^\infty e^{-t \psi(\lambda^2)} f(\lambda) F_\lambda(x) d\lambda , && t, x > 0 .
}
\end{lem}

\begin{proof}
Again the argument follows the proof of Theorem~3 in~\cite{bib:kkms10}. Let $p^\hl_t(x, dy)$ be the kernel of $P^\hl_t$ (i.e. the $\pr_x$ distribution of the killed process $X^\hl_t$). If $f \in C_c(\hl)$, then $p^\hl_t(x, dy) f(\lambda) F_\lambda(y)$ is integrable in $y, \lambda > 0$. Theorem~\ref{th:eigenfunctions} and Fubini yield~\eqref{eq:pdt:spec} in this case. The general case $f \in L^2(\hl)$ follows by approximation.

Let $f, g \in C_c(\hl)$, $k \in \Z$, and define $f_k(\lambda) = e^{-k t \psi(\lambda^2)} f(\lambda)$, $g_k(\lambda) = e^{-k t \psi(\lambda^2)} g(\lambda)$. From~\eqref{eq:pdt:spec} it follows that $P^\hl_t \Pi^* f_k = \Pi^* f_{k+1}$ and $P^\hl_t \Pi^* g_k = \Pi^* g_{k+1}$. The operators $P^\hl_t$ are self-adjoint, so that
\formula{
 \scalar{\Pi^* f, \Pi^* g} & = \scalar{P^\hl_t \Pi^* f_{-1}, \Pi^* g(x)} \\
 & = \scalar{\Pi^* f_{-1}, P^\hl_t \Pi^* g} = \scalar{\Pi^* f_{-1}, \Pi^* g_1} .
}
By induction, we have $\scalar{\Pi^* f, \Pi^* g} = \scalar{\Pi^* f_{-k}, \Pi^* g_k}$ for $k \ge 0$. In particular, if $\supp f \sub (0, \lambda_0)$ and $\supp g \sub (\lambda_0, \infty)$, then $\scalar{\Pi^* f, \Pi^* g} = \langle e^{-k t \psi(\lambda_0^2)} f_{-k}, e^{k t \psi(\lambda_0^2)} g_k \rangle$. The right-hand side tends to zero as $k \to \infty$, so that $\Pi^* f$ and $\Pi^* g$ are orthogonal in $L^2(\hl)$. By approximation, this holds true for any $f, g \in L^2(\hl)$ such that $f(\lambda) = 0$ for $\lambda \ge \lambda_0$ and $g(\lambda) = 0$ for $\lambda \le \lambda_0$.

Define $m(E) = \|\Pi^* \ind_E\|_2^2$ for Borel $E \sub \hl$. Clearly, $0 \le m(E) \le C \|\ind_E\|_2^2 = C |E|$. If $E_1 \sub (0, \lambda_0)$ and $E_2 \sub (\lambda_0, \infty)$, then
\formula{
 m(E_1 \cup E_2) & = \norm{\Pi^* \ind_{E_1}}_2^2 + \norm{\Pi^* \ind_{E_2}}_2^2 + 2 \scalar{\Pi^* \ind_{E_1}, \Pi^* \ind_{E_2}} = m(E_1) + m(E_2) .
}
Finally, suppose that $E = \bigcup_{n = 1}^\infty E_n$, where $E_1 \sub E_2 \sub ...$ and $|E| < \infty$. Since $\ind_{E_n}$ converges to $\ind_E$ in $L^2(\hl)$, also $\Pi^* \ind_{E_n}$ converges to $\Pi^* \ind_E$ in $L^2(\hl)$, and so $m(E) = \lim_{n \rightarrow \infty} m(E_n)$. It follows that $m$ is an absolutely continuous measure on $(0, \infty)$, and by approximation,
\formula{
 \scalar{\Pi^* f, \Pi^* g} & = \int_0^\infty f(\lambda) g(\lambda) m(d\lambda)
}
for any $f, g \in L^2(\hl)$. The lemma will be proved if we show that $m(E) = (\pi / 2) |E|$.

Fix $\lambda_0 > 0$ and define $f_\delta = (1 / \sqrt{\delta}) \ind_{[\lambda_0, \lambda_0 + \delta]}$ for $\delta > 0$. Clearly, $\|f_\delta\|_2 = 1$, and $m([\lambda_0, \lambda_0 + \delta]) / \delta = \|\Pi^* f_\delta\|_2^2$. Furthermore, $\sqrt{\delta} \, \Pi^* f_\delta = g_1 + g_2 - g_3$, where
\formula{
 g_1(x) & = \int_{\lambda_0}^{\lambda_0 + \delta} \sin(\lambda x + \theta_{\lambda_0}) d\lambda , \\
 g_2(x) & = \int_{\lambda_0}^{\lambda_0 + \delta} (\sin(\lambda x + \theta_\lambda) - \sin(\lambda x + \theta_{\lambda_0})) d\lambda , \\
 g_3(x) & = \int_{\lambda_0}^{\lambda_0 + \delta} G_\lambda(x) d\lambda .
}
We study the behavior of $g_1$, $g_2$ and $g_3$ as $\delta \to 0^+$.

By Lemma~\ref{lem:gest}, we have $G_\lambda(x) \le 1$ and $\int_0^\infty G_\lambda(x) dx \le 1 / \lambda$. Hence, by Fubini,
\formula{
 \norm{g_3}_2^2 & = \int_0^\infty \expr{\int_{\lambda_0}^{\lambda_0 + \delta} G_\lambda(x) d\lambda}^2 dx \le \int_0^\infty \delta \expr{\int_{\lambda_0}^{\lambda_0 + \delta} G_\lambda(x) d\lambda} dx \le \frac{\delta^2}{\lambda_0} ,
}
so that $\|g_3\|_2 / \sqrt{\delta}$ tends to zero as $\delta \to 0^+$.

The function $g_2$ is the imaginary part of the Fourier transform of the function $\tilde{f}(\lambda) = (\exp(i \thet_\lambda) - \exp(i \thet_{\lambda_0})) \ind_{[\lambda_0, \lambda_0 + \delta]}$. By Proposition~\ref{prop:diff:theta}, we have $|\thet_\lambda - \thet_{\lambda_0}| \le C |\lambda - \lambda_0|$ for some $C > 0$ (depending on $\lambda_0$). Hence $|\tilde{f}(\lambda)| \le C \delta$, and so $\|\tilde{f}\|_2 \le C \delta^{3/2}$. It follows that $\|g_2\|_2 / \sqrt{\delta} \le \sqrt{2 \pi} \, \|\tilde{f}\|_2 / \sqrt{\delta} \le \sqrt{2 \pi} \, C \delta$ also tends to zero as $\delta \to 0^+$.

Finally, for the function $g_1$, by Fubini,
\formula{
 & 2 \int_0^\infty (g_1(x))^2 e^{-\eps x} dx \\ & \hspace*{1.5em} = \int\limits_{\lambda_0}^{\lambda_0 + \delta} \int\limits_{\lambda_0}^{\lambda_0 + \delta} \int\limits_0^\infty 2 \sin(\lambda_1 x + \theta_{\lambda_0}) \sin(\lambda_2 x + \theta_{\lambda_0}) e^{-\eps x} dx d\lambda_1 d\lambda_2 \\
 & \hspace*{1.5em} = \int\limits_{\lambda_0}^{\lambda_0 + \delta} \int\limits_{\lambda_0}^{\lambda_0 + \delta} \int\limits_0^\infty (\cos((\lambda_1 - \lambda_2) x) - \cos((\lambda_1 + \lambda_2) x + 2 \theta_{\lambda_0})) e^{-\eps x} dx d\lambda_1 d\lambda_2 \\
 & = \int\limits_{\lambda_0}^{\lambda_0 + \delta} \int\limits_{\lambda_0}^{\lambda_0 + \delta} \expr{\frac{\eps}{\eps^2 + (\lambda_1 - \lambda_2)^2} + \frac{\eps \cos(2 \thet_{\lambda_0}) - (\lambda_1 + \lambda_2) \sin(2 \thet_{\lambda_0})}{\eps^2 + (\lambda_1 + \lambda_2)^2}} d\lambda_1 d\lambda_2 .
}
By taking the limit $\eps \to 0^+$ and using dominated convergence for the second term of the integrand, we obtain that
\formula{
 \norm{g_1}_2^2 & = \int_0^\infty (g_1(x))^2 dx = \frac{1}{2} \int_{\lambda_0}^{\lambda_0 + \delta} \expr{\pi - \sin(2 \thet_{\lambda_0}) \log \frac{\lambda_0 + \delta + \lambda_2}{\lambda_0 + \lambda_2}} d\lambda_2 .
}
Hence, using $0 \le \log(1 + s) \le s$ ($s \ge 0$),
\formula{
 \lim_{\delta \to 0^+} \frac{\norm{g_1}_2^2}{\delta} & = \frac{\pi}{2} - \lim_{\delta \to 0^+} \frac{\sin(2 \thet_{\lambda_0})}{2 \delta} \int_{\lambda_0}^{\lambda_0 + \delta} \log \expr{1 + \frac{\delta}{\lambda_0 + \lambda_2}} d\lambda_2 = \frac{\pi}{2} \, .
}
We conclude that for all $\lambda_0 > 0$,
\formula{
 \lim_{\delta \to 0^+} \frac{m([\lambda_0, \lambda_0 + \delta])}{\delta} & = \lim_{\delta \to 0^+} \expr{\frac{\|g_1 + g_2 - g_3\|_2}{\sqrt{\delta}}}^2 = \frac{\pi}{2} \, ,
}
that is, $m(d\lambda) = (\pi / 2) d\lambda$.
\end{proof}

\begin{proof}[Proof of Theorem~\ref{th:spectral} when $\Pi$ is injective]
By Lemma~\ref{lem:pi:spec}, $\sqrt{2 / \pi} \, \Pi^*$ is an $L^2(\hl)$ isometry, and by the extra assumption, its adjoint $\sqrt{2 / \pi} \, \Pi$ is injective. Hence, both operators are unitary on $L^2(\hl)$. The theorem follows now directly from Lemma~\ref{lem:pi:spec} and $(\Pi^*)^{-1} = (2 / \pi) \Pi^*$.
\end{proof}

In~\cite{bib:kmr11a}, Theorem~\ref{th:spectral} is proved in full generality. More precisely, it is shown that $\Pi$ is always injective. Here we are satisfied with the following observation: When $X_t$ is a symmetric $\alpha$-stable process, then, by Example~\ref{ex:stable} below, $F_\lambda(x) = F_x(\lambda)$, and therefore $\Pi^* = \Pi$. Since a self-adjoint isometry is necessarily unitary, we see (by Lemma~\ref{lem:pi:spec}) that $\Pi$ is injective. Hence, we have proved the following result.

\begin{cor}
\label{cor:stable}
When $\alpha \in (0, 2]$ and $X_t$ is the symmetric $\alpha$-stable process (that is, $\psi(\xi) = \xi^{\alpha/2}$), then the hypothesis of Theorem~\ref{th:spectral} holds true.\qed
\end{cor}

In the following proofs of Theorems~\ref{th:pdt} and~\ref{th:fpt}, we use the full statement of Theorem~\ref{th:spectral}. In other words, we assume that $\Pi$ is injective (which is proved to be true in general in~\cite{bib:kmr11a}).

\begin{proof}[Proof of Theorem~\ref{th:pdt}]
Let $f \in C_c(\hl)$. By Theorem~\ref{th:spectral},
\formula{
 P^\hl_t f(x) & = \frac{2}{\pi} \, \Pi^* (e^{-t \psi(\lambda^2)} \Pi f(\lambda))(x) = \frac{2}{\pi} \int_0^\infty e^{-t \psi(\lambda^2)} F_\lambda(x) \Pi^* f(\lambda) d\lambda
}
for all $t, x > 0$. The function $e^{-t \psi(\lambda^2)} F_\lambda(x) F_\lambda(y) f(y)$ is jointly integrable in $y, \lambda > 0$. By Fubini,
\formula{
 P^\hl_t f(x) & = \int_0^\infty \expr{\frac{2}{\pi} \int_0^\infty e^{-t \psi(\lambda^2)} F_\lambda(x) F_\lambda(y) d\lambda} f(y) dy .
}
Since $f \in C_c(\hl)$ is arbitrary, the kernel of $P^\hl_t$ has the form~\eqref{eq:pdt}.
\end{proof}

\begin{proof}[Proof of Theorem~\ref{th:fpt} under additional assumption~\eqref{eq:fpt:a3}]
By monotone convergence,
\formula{
 \pr_x(\tau_\hl > t) & = P^\hl_t \ind_\hl(x) = \lim_{\eps \to 0^+} P^\hl_t e_\eps(x) ,
}
where $e_\eps(x) = e^{-\eps x} \ind_\hl(x)$. We have
\formula{
 \Pi e_\eps(\lambda) & = \int_0^\infty e^{-\eps x} F_\lambda(x) dx = \laplace F_\lambda(\eps) .
}
By Theorem~\ref{th:spectral},
\formula[eq:lfpt]{
 P^\hl_t e_\eps(x) & = \frac{2}{\pi} \int_0^\infty e^{-t \psi(\lambda^2)} F_\lambda(x) \laplace F_\lambda(\eps) d\lambda .
}
If we can use dominated convergence, then, by Proposition~\ref{prop:lf:0},
\formula{
 \pr_x(\tau_\hl > t) & = \frac{2}{\pi} \int_0^\infty e^{-t \psi(\lambda^2)} F_\lambda(x) \expr{\lim_{\eps \to 0^+} \laplace F_\lambda(\eps)} d\lambda \\
 & = \frac{2}{\pi} \int_0^\infty e^{-t \psi(\lambda^2)} F_\lambda(x) \sqrt{\frac{\psi'(\lambda^2)}{\psi(\lambda^2)}} \, d\lambda ,
}
as desired. Hence, it remains to show that the integrand in~\eqref{eq:lfpt} is dominated by an integrable function when $\eps \to 0^+$.

By the assumption~\eqref{eq:fpt:a3}, there are $C, \alpha > 0$ such that $\psi(\xi) \ge C \xi^\alpha$ for $\xi > 1$ (see the proof of Corollary~\ref{cor:power}). Hence, $e^{-t \psi(\lambda^2)}$ is integrable in $\lambda > 0$. By Proposition~\ref{prop:lfest}, we have $\laplace F_\lambda(\eps) \le 2 / \lambda$, and by Corollary~\ref{cor:power} (for $x_1 = 0$ and $x_2 = x$), $|F_\lambda(x)| \le C \min(1, \lambda^{\alpha/2})$ for some $C > 0$ (depending on $x$). Hence, the integrand in~\eqref{eq:lfpt} is bounded by $2 C e^{-t \psi(\lambda^2)} \min(1, \lambda^{\alpha/2}) / \lambda$, which is integrable in $\lambda > 0$. Formula~\eqref{eq:fpt} is proved.

By the same argument, the integrand in~\eqref{eq:fptd} is integrable jointly in $\lambda > 0$ and $t > t_0$. Formula~\eqref{eq:fptd} follows by integrating both sides of it in $t$ over $(t_0, \infty)$, using Fubini and comparing the result with~\eqref{eq:fpt}.
\end{proof}

%
%

\section{Examples}
\label{sec:examples}

\begin{exa}
\label{ex:stable}
Let $\psi(\xi) = \xi^{\alpha/2}$, where $\alpha \in (0, 2)$. Then $Z_t$ is the $(\alpha / 2)$-stable subordinator, and $X_t = B_{Z_t}$ is the symmetric $\alpha$-stable L{\'e}vy process. By~\eqref{eq:theta2} and~\eqref{eq:integral},
\formula{
 \thet_\lambda & = \frac{1}{\pi} \int_0^1 \frac{1}{1 - z^2} \, \log \frac{1 - z^\alpha}{(z^{-\alpha} - 1) z^2} \, dz = \frac{2 - \alpha}{\pi} \int_0^1 \frac{-\log z}{1 - z^2} \, dz = \frac{(2 - \alpha) \pi}{8} \, .
}
(Note that the \emph{phase-shift} of $(2 - \alpha) \pi/8$ was conjectured, in somewhat different setting, in~\cite{bib:zrk07}; see~\cite{bib:k10a} for further discussion.) By Theorem~\ref{th:eigenfunctions}, the eigenfunctions of $P^\hl_t$ and $\A_\hl$ are given by the formula
\formula{
 F_\lambda(x) & = \sin(\lambda x + (2 - \alpha) \pi / 8) - \int_0^\infty \gamma_\lambda(\xi) e^{-x \xi} d\xi , && x > 0 ,
}
where (see~\eqref{eq:gamma00} and~\eqref{eq:dagger})
\formula{
 \gamma_\lambda(\xi) & = \frac{\sqrt{2 \alpha} \, \lambda^{\alpha - 1}}{2 \pi} \, \frac{\xi^\alpha \sin (\alpha \pi / 2)}{\lambda^{2\alpha} + \xi^{2 \alpha} - 2 \lambda^\alpha \xi^\alpha \cos (\alpha \pi / 2)} \times \\
 & \qquad \times \exp\expr{\frac{1}{\pi} \int_0^\infty \frac{1}{1 + \zeta^2} \, \log \frac{1 - \xi^2 \zeta^2 / \lambda^2}{1 - \xi^\alpha \zeta^\alpha / \lambda^\alpha} \, d\zeta} .
}
Clearly $\gamma(s) = \lambda \gamma_\lambda(\lambda s)$ does not depend on $\lambda$, and finally,
\formula{
 F_\lambda(x) & = F(\lambda x) = \sin(\lambda x + (2 - \alpha) \pi / 8) - \int_0^\infty \gamma(s) e^{-\lambda s x} ds ,
}
where
\formula{
 \gamma(s) & = \frac{\sqrt{2 \alpha} \, \sin (\alpha \pi / 2)}{2 \pi} \, \frac{s^\alpha}{1 + s^{2 \alpha} - 2 s^\alpha \cos (\alpha \pi / 2)} \\ & \qquad \times \exp\expr{\frac{1}{\pi} \int_0^\infty \frac{1}{1 + \zeta^2} \, \log \frac{1 - s^2 \zeta^2}{1 - s^\alpha \zeta^\alpha} \, d\zeta} .
}
By Lemma~\ref{lem:regular}, we have $F_\lambda(x) \sim (\sqrt{\alpha / 2} \, \Gamma(\alpha / 2))^{-1} (\lambda x)^{\alpha/2}$ as $x \to 0^+$. By Theorem~\ref{th:spectral} (or Corollary~\ref{cor:stable}), the functions $F_\lambda$ yield a generalized eigenfunction expansion of $\A_\hl$ and $P^\hl_t$, and by Theorems~\ref{th:pdt} and~\ref{th:fpt},
\formula{
 p^\hl_t(x, y) & = \frac{2}{\pi} \int_0^\infty e^{-t \lambda^\alpha} F(\lambda x) F(\lambda y) d\lambda , && t, x, y > 0 , \\
 \pr_x(\tau_\hl > t) & = \frac{\sqrt{2 \alpha}}{\pi} \int_0^\infty \frac{e^{-t \lambda^\alpha} F(\lambda x)}{\lambda} \, d\lambda , && t , x > 0 .
}
For $\alpha = 1$, these results were obtained in~\cite{bib:kkms10}, and the formulae for $p^\hl_t(x, y)$ and $\pr_x(\tau_\hl > t)$ can be significantly simplified in this case.
\end{exa}

\begin{exa}
\label{ex:relativistic}
If $\psi(\xi) = \sqrt{m^2 + \xi} - m$ for some $m > 0$, then $Z_t$ is the relativistic subordinator with mass $m$, and the corresponding L{\'e}vy process is the relativistic $1$-stable process with mass $m$. Its infinitesimal operator $\A$ is the (quasi-)relativistic Hamiltonian of a free particle with mass $m$, and $\A_\hl$ is the Hamiltonian corresponding to the semi-infinite potential well. By~\eqref{eq:theta2},
\formula{
 \thet_\lambda & = \frac{1}{\pi} \int_0^1 \frac{1}{1 - z^2} \, \log \frac{\sqrt{m^2 + \lambda^2} - \sqrt{m^2 + \lambda^2 z^2}}{\expr{\sqrt{m^2 + \lambda^2 / z^2} - \sqrt{m^2 + \lambda^2}} z^2} \, dz ,
}
which increases with $\lambda$ from $0$ as $\lambda \to 0^+$ to $\pi / 8$ as $\lambda \to \infty$. The eigenfunctions of $P^\hl_t$ have the form $F_\lambda(x) = \sin(\lambda x + \thet_\lambda) - G_\lambda(x)$, where $G_\lambda$ is the Laplace transform of an explicit nonnegative function. By Lemma~\ref{lem:regular}, $F_\lambda(x) \sim \sqrt{2 \lambda x / \pi}$ as $x \to 0^+$. See~\cite{bib:kkm11} for a detailed discussion of this example.
\end{exa}

\begin{exa}
\label{ex:cauchybrownian}
Let $\psi(\xi) = \xi^{\alpha/2} + \beta \xi$, $\alpha \in (0, 2)$, $\beta > 0$. Then $Z_t$ is the $(\alpha/2)$-stable subordinator with drift, and the corresponding L{\'e}vy process $X_t$ is a mixture of the Brownian motion and the symmetric $\alpha$-stable L{\'e}vy process. When $\alpha = 1$, then, by~\eqref{eq:theta2},
\formula{
 \thet_\lambda & = \frac{1}{\pi} \int_0^1 \frac{1}{1 - \zeta^2} \, \log \frac{(\lambda + \beta \lambda^2) - (\lambda \zeta + \beta \lambda^2 \zeta^2)}{((\frac{\lambda}{\zeta} + \frac{\beta \lambda^2}{\zeta^2}) - (\lambda + \beta \lambda^2)) \zeta^2} \, d\zeta \, \\
 & = \frac{1}{\pi} \int_0^1 \frac{1}{1 - \zeta^2} \, \log \frac{1 + \beta \lambda + \beta \lambda \zeta}{\zeta + \beta \lambda + \beta \lambda \zeta} \, d\zeta .
}
This decreases with $\lambda$ from $\pi / 8$ as $\lambda \to 0^+$ to $0$ as $\lambda \to \infty$, and can be written explicitly in terms of the dilogarithm function $\li_2$. For general $\alpha$, the expression for $\thet_\lambda$ is more complicated, and it can be proved that $\thet_\lambda$ decreases from $(2 - \alpha) \pi / 8$ to $0$. By Lemma~\ref{lem:regular}, $F_\lambda(x) \sim \sqrt{\beta + (\alpha/2) \lambda^{\alpha - 2}} \, \lambda x$ as $x \to 0^+$.
\end{exa}

\begin{exa}
\label{ex:gamma}
Let $\psi(\xi) = \log(1 + \xi)$, so that $Z_t$ is the gamma subordinator. The subordinate process $X_t$ is called variance gamma process. By~\eqref{eq:theta2},
\formula{
 \thet_\lambda & = \frac{1}{\pi} \int_0^1 \frac{1}{1 - z^2} \, \log \frac{\log(1 + \lambda^2) - \log(1 + \lambda^2 z^2)}{(\log(1 + \lambda^2 / z^2) - \log(1 + \lambda^2)) z^2} \, dz \, .
}
It can be proved that $\thet_\lambda$ increases with $\lambda$ from $0$ as $\lambda \to 0^+$ to $\pi / 4$ as $\lambda \to \infty$. By Lemma~\ref{lem:regular}, the eigenfunctions $F_\lambda$ satisfy $F_\lambda(x) \sim C_\lambda / \sqrt{|\log x|}$ as $x \to 0^+$, with $C_\lambda = \lambda / \sqrt{2 + 2 \lambda^2}$. In particular, $F_\lambda$ is not H{\"o}lder continuous. Note that the assumptions of Theorem~\ref{th:pdt} are satisfied only for $t > 1/2$, but Theorem~\ref{th:fpt} applies for all $t > 0$. On the other hand, the more restrictive condition~\eqref{eq:fpt:a3} is not met, so this case is not covered by the proof of Theorem~\ref{th:fpt} given above. The general argument of~\cite{bib:kmr11a} needs to be used.
\end{exa}

\begin{exa}
\label{ex:loglog}
Let $\psi(\xi) = \log(1 + \log(1 + \xi))$. It can be verified that in this case $\thet_\lambda$ is greater than $\pi/4$ for some $\lambda$, e.g. $\thet_8 \approx 0.287 \pi$. This proves that it is not true in general that $\thet_\lambda \le \pi/4$, even if $\psi$ is unbounded. Furthermore, $e^{-t \psi(\lambda^2)}$ is not integrable in $\lambda > 0$ for any $t > 0$, so Theorem~\ref{th:pdt} is of no use here. Theorem~\ref{th:fpt} applies for $t > 1/2$, but the additional assumption~\eqref{eq:fpt:a2} is not satisfied.
\end{exa}

\begin{exa}
\label{ex:bounded}
Let $\psi(\xi) = \xi / (1 + \xi)$. In this case the subordinator $Z_t$ is the compound Poisson process with exponential jump distribution, and the jumps of $X_t$ have Laplace distribution with density $(1/2) e^{-|x|}$. Note that $1 / (\psi(\lambda^2) - \psi(\zeta))$ is holomorphic in $\zeta \in \C \setminus \{\lambda\}$, and hence $\gamma_\lambda$ vanishes. It follows that $F_\lambda = \sin(\lambda x + \thet_\lambda) \ind_{\hl}(x)$. By~\eqref{eq:theta} and a contour integration,
\formula{
 \thet_\lambda & = \frac{1}{2 \pi} \int_{-\infty}^\infty \frac{\lambda}{\lambda^2 - z^2} \, \log \frac{1 + z^2}{1 + \lambda^2} \, dz = \arctan \lambda ;
}
we omit the details. Hence, $F_\lambda(x) = \sin(\lambda x + \arctan \lambda) \ind_{\hl}(x)$ satisfies $\A F_\lambda(x) = -\psi(\lambda^2) F_\lambda(x)$ for all $x \in \hl$. This can be alternatively proved by a direct calculation.

It can be easily proved that $\gamma_\lambda$ vanishes if and only if $\psi(\xi) = \beta \xi$ (i.e. $X_t$ is a Brownian motion) or $\psi(\xi) = C_1 \xi / (\xi + C_2)$ for some $C_1, C_2 > 0$ (which corresponds to the process $X_t$ studied in this example, up to time and space scaling).
\end{exa}

\begin{exa}
In general, $\gamma_\lambda$ may have non-zero singular part. For example, consider $\psi(\xi) = 5 x / (x + 1) + x / (x + 5)$ and $\lambda = 1$. Then $\psi^+(-4) = 8 / 3 = \psi(1)$, and hence $\gamma_\lambda$ has an atom at $2$ for $\lambda = 1$.
\end{exa}

%
%

\section{Application to systems of PDEs}
\label{sec:applications}

As it was the case in~\cite{bib:kkms10}, where the Cauchy process was studied, the eigenfunctions $F_\lambda$ are the boundary values of solutions of a certain system of PDEs, a spectral problem with spectral parameter in the boundary. In~\cite{bib:kkms10}, the formula for $F_\lambda$ (for the case of the Cauchy process) was found by solving the corresponding system of PDEs. Here the argument goes in the opposite direction, and the result for PDEs is an application of the formula for $F_\lambda$.

We need some background on Kre{\u{\i}}n correspondence. For a complete Bernstein function $\psi$, there is a unique \emph{string} $m$ (a non-decreasing, right-continuous unbounded function from $[0, \infty)$ to $[0, \infty]$) which is the \emph{Kre\u{\i}n representation} of $\psi$. For simplicity, we only consider the case when $m(0) = 0$, $m$ is continuously differentiable on $(0, \infty)$, and $m'(y) > 0$ for all $y > 0$. Denote $a(y) = 1 / m'(y)$. Then, for each $\lambda \ge 0$ there exists a (unique) nonnegative, non-increasing continuous function $g_\lambda(y) = g(\lambda, y)$ such that $a(y) g_\lambda''(y) = \lambda g_\lambda(y)$ for $y > 0$, and $g_\lambda(0) = 1$. Furthermore,
\formula{
 \psi(\lambda) & = -g_\lambda'(0) , && \lambda \ge 0 .
}
For the details, we refer to~\cite{bib:kw82}, or Chapter~14 in~\cite{bib:ssv10}, and the references therein. The Kre\u{\i}n's correspondence between $m$ and $\psi$ is a rather complicated concept, and there are only few examples of explicit pairs, see~\cite{bib:ssv10}.

Let $f \in \schwartz$ and define $u(x, y)$ using the Fourier transform in $x$,
\formula{
 \fourier_x u(\xi, y) & = g(\xi^2, y) \fourier f(\xi) , && \xi \in \R, \, y \ge 0 .
}
It is easy to check that $u$ is the solution of the problem
\formula{
 \expr{\frac{\partial^2}{\partial x^2} + a(y) \, \frac{\partial^2}{\partial y^2}} u(x, y) & = 0 && x \in \R, \, y > 0 ; \\
 u(x, 0) & = f(x) , && x \in \R .
}
Furthermore,
\formula{
 \fourier_x \expr{\frac{\partial}{\partial y} \, u}(\xi, 0) & = \lim_{y \to 0^+} \frac{g(\xi^2, y) - 1}{y} \, \fourier f(\xi) = -\psi(\xi^2) \fourier f(\xi) ,
}
so that $(\partial / \partial y) u(x, 0) = \A f(x)$, where $\A$ is the operator with Fourier symbol $-\psi(\xi^2)$. Theorems~\ref{th:eigenfunctions} implies the following result.

\begin{thm}
\label{th:pde}
Let $a(y)$ be a positive continuous function on $(0, \infty)$ such that $1 / a$ is integrable over $(0, 1)$, but not over $(1, \infty)$. For $\lambda \ge 0$, let $g_\lambda(y) = g(\lambda, y)$ be the nonnegative, non-increasing continuous solution of $a(y) g_\lambda''(y) = \lambda g_\lambda(y)$ ($y > 0$), satisfying $g_\lambda(0) = 1$. Define $\psi(\lambda) = g_\lambda'(0)$, and let $F_\lambda$ be the function given in Theorem~\ref{th:eigenfunctions}. Finally, let
\formula[eq:pde:u]{
 u_\lambda(x, y) & = \fourier_x^{-1} (g(\xi^2, y) \fourier F_\lambda(\xi)) , && x \in \R , \, y \ge 0 .
}
Then $u_\lambda$ is continuous and bounded in $x \in \R$, $y \ge 0$, twice differentiable in $x \in \R$, $y > 0$, and it satisfies
\formula[]{
 \nonumber \expr{\frac{\partial^2}{\partial x^2} + a(y) \, \frac{\partial^2}{\partial y^2}} u_\lambda(x, y) & = 0 && x \in \R, \, y > 0 ; \\
 \label{eq:pde:problem} u_\lambda(x, 0) & = 0 , && x < 0 ; \\
 \nonumber \frac{d}{dy} \, u_\lambda(x, 0) & = -\psi(\lambda^2) u_\lambda(x, 0) , && x > 0 .
}
\end{thm}

The Fourier transform and inverse Fourier transform in~\eqref{eq:pde:u} needs to be understood in the distributional sense. 

\begin{proof}[Sketch of the proof]
Using the properties of $g(\lambda, y)$, it is easy to check that $u_\lambda$ is a weak solution of~\eqref{eq:pde:problem}. Since $g(\xi^2, y) \fourier F_\lambda(\xi)$ (for a fixed $y \ge 0$) is a sum of an $L^2(\R)$ function and $C_y / (\xi - \lambda) - \bar{C}_y / (\xi + \lambda)$ (in a similar way as in Section~\ref{sec:eigenfunctions}), $u(x, y)$ is a bounded function. Continuity of $u_\lambda$ and its first and second partial derivatives in $x \in \R$, $y > 0$ follows by elliptic regularity theorem. Continuity of $u_\lambda$ on the boundary is proved using the properties of $g(\lambda, y)$ and $F_\lambda(x)$.
\end{proof}

We provide two examples. If $a(y) = 1 / (1 + 2 a y)$ for a constant $a > 0$, then we find that $g_\lambda(y) = (1 + 2 a y)^{-(\sqrt{\lambda + a^2} - a) / (2a)}$ and $\psi(\lambda) = \sqrt{\lambda + a^2} - a$. Hence $\psi$ is the characteristic exponent of the relativistic subordinator, described in Example~\ref{ex:relativistic}. For a related example, see Section~2.7 in~\cite{bib:py03}.

Consider now $\alpha \in (0, 2)$ and $a(y) = \alpha^2 c_\alpha^2 y^{2 - 2/\alpha}$, where $c_\alpha = 2^{-\alpha} \Gamma(1-\alpha/2) / \Gamma(1+\alpha/2)$. We have $g_\lambda(y) = C_\alpha (c_\alpha \sqrt{\lambda} \, y)^{1/2} K_{\alpha/2}((c_\alpha \sqrt{\lambda} \, y)^{1/\alpha})$, where $K_{\alpha/2}$ is the modified Bessel function of the second kind and $C_\alpha = \alpha 2^{-\alpha/2} / \Gamma(1 + \alpha/2)$. Moreover, $\psi(\lambda) = \lambda^{\alpha/2}$, so that $\psi$ is the characteristic exponent of the $(\alpha/2)$-stable subordinator, considered in Example~\ref{ex:stable}. When $\alpha = 1$, $a(y) = 1$.

The problem~\eqref{eq:pde:problem} with $a(y) = 1$ was studied in~\cite{bib:kkms10} to find the eigenfunctions of $P^\hl_t$. Earlier, a similar relation for more general domains (also in higher dimensions) was applied e.g. in~\cite{bib:bk04, bib:bk06, bib:bbkrsv09, bib:s58}, and for general symmetric stable processes e.g. in~\cite{bib:d90, bib:d04, bib:dm07, bib:mo69}. Related problems appear frequently in hydrodynamics (the \emph{sloshing problem}), see the references in~\cite{bib:kkms10}.

The operator $d^2 / dx^2 + a(y) d^2 / dy^2$ is the generator of a two-dimensional diffusion $Y_t$ on $\R \times [0, \infty)$, with reflecting boundary. By repeating the construction given in Example~3.1 in~\cite{bib:ksv10b}, we can identify the subordinate Brownian motion $X_t$ studied in the present article with the trace left on the horizontal axis by $Y_t$. We remark that the connection between subordinate Brownian motion in $\R^d$ and traces of diffusions in $\R^{d+1}$ was applied e.g. in~\cite{bib:bmr09, bib:bmr10} to find formulae for the distribution of some L{\'e}vy processes stopped at the time of first exit from the half-line or the interval. A related problem for the trace of a two-dimensional jump-type stable process was studied in~\cite{bib:i09}.

%
%

\subsection*{Acknowledgements} I would like to express my gratitude to Tadeusz Kulczycki for many enlightening discussions. I also thank Krzysztof Bogdan, Tomasz Byczkowski, Piotr Graczyk, Alexey Kuznetsov, Jacek Ma{\l}ecki, Micha{\l} Ryznar, Renming Song and Zoran Vondra\v{c}ek for their comments to the preliminary version of the article. I am particularly thankful to the anonymous referee for helpful remarks and many detailed suggestions.

%
%

%
%


\begin{thebibliography}{00}

\normalsize
\baselineskip=17pt

\bibitem{bib:a04}
D.~Applebaum,
\emph{L{\'e}vy Processes and Stochastic Calculus}.
Cambridge University Press, Cambridge, 2004.

\bibitem{bib:bk04}
R.~Ba{\~n}uelos, T.~Kulczycki,
\emph{The Cauchy process and the Steklov problem}.
J.~Funct. Anal. 211(2) (2004), pp.~355--423.

\bibitem{bib:bk06}
R.~Ba{\~n}uelos, T.~Kulczycki,
\emph{Spectral gap for the Cauchy process on convex symmetric domains}.
Comm. Partial Diff. Equations 31 (2006), pp.~1841--1878.

\bibitem{bib:bd57}
G.~Baxter, M.~D.~Donsker,
\emph{On the distribution of the supremum functional for processes with stationary independent increments}.
Trans. Amer. Math. Soc. 85 (1957) 73--87.

\bibitem{bib:b98}
J.~Bertoin,
\emph{L{\'e}vy Processes}.
Cambridge University Press, Melbourne-New York, 1998.

\bibitem{bib:bgt87}
N.~H.~Bingham, C.~M.~Goldie, J.~L.~Teugels,
\emph{Regular Variation}.
Cambridge University Press, Cambridge, 1987.

\bibitem{bib:bg68}
R.~M.~Blumenthal, R.~K.~Getoor,
\emph{Markov Processes and Potential Theory}.
Academic Press, Reading, MA, 1968.

\bibitem{bib:bb99}
K.~Bogdan, T.~Byczkowski,
\emph{Potential theory for the $\alpha$-stable Schr\"odinger operator on bounded Lipschitz domains}.
Studia Math. 133(1) (1999) 53--92.

\bibitem{bib:bb00}
K.~Bogdan, T.~Byczkowski,
\emph{Potential theory of Schr{\"o}dinger operator based on fractional Laplacian}.
Probab. Math. Statist. 20(2) (2000), pp.~293--335.

\bibitem{bib:bbkrsv09}
K.~Bogdan, T.~Byczkowski, T.~Kulczycki, M.~Ryznar, R.~Song, Z.~Vondracek,
\emph{Potential Analysis of Stable Processes and its Extensions}.
Lecture Notes in Mathematics 1980, Springer, 2009.

\bibitem{bib:bgr10}
K.~Bogdan, T.~Grzywny, M.~Ryznar,
\emph{Heat kernel estimates for the fractional Laplacian}.
Annals Probab. 38(5) (2010), pp.~1901--1923.

\bibitem{bib:bmr09}
T.~Byczkowski, J.~Ma{\l}ecki, M.~Ryznar,
\emph{Bessel potentials, hitting distributions and Green functions}.
Trans. Amer. Math. Soc. 361 (2009), pp.~4871--4900.

\bibitem{bib:bmr10}
T.~Byczkowski, J.~Ma{\l}ecki, M.~Ryznar,
\emph{Hitting half-spaces by Bessel-Brownian diffusions}.
Potential Anal. Volume 33(1) (2010), pp.~47--83.

\bibitem{bib:cks12}
Z.~Q.~Chen, P.~Kim, R.~Song,
\emph{Sharp Heat Kernel Estimates for Relativistic Stable Processes in Open Sets}.
Preprint (2009) arXiv:0908.1509v1, To appear in Ann. Probab. 

\bibitem{bib:cs97}
Z.~Q.~Chen, R.~Song,
\emph{Intrinsic ultracontractivity and conditional gauge for symmetric stable processes}.
J.~Funct. Anal. 150(1) (1997) 204--239.

\bibitem{bib:cs05}
Z.Q.~Chen, R.~Song,
\emph{Two sided eigenvalue estimates for subordinate processes in domains}.
J.~Funct. Anal. 226(1) (2005), pp.~90--113.

\bibitem{bib:cs06}
Z.Q.~Chen, R.~Song,
\emph{Spectral properties of subordinate processes in domains}.
Stochastic Analysis and Partial Differential Equations, pp.~77--84 (Eds. G-Q Chen, E. Hsu and M. Pinsky),
AMS Contemp. Math. 429 (2007).

\bibitem{bib:c85}
K.~L.~Chung,
\emph{Doubly-Feller process with multiplicative functional}.
In: \emph{Seminar on Stochastic Processes, 1985} (Gainesville, Fla., 1985),
Progr. Probab. Statist.~12, Birkh\"{a}user, Boston, Boston, MA., 1986, 63--78.

\bibitem{bib:d90}
R.~D.~DeBlassie,
\emph{The first exit time of a two-dimensional symmetric stable process from a wedge}.
Ann. Probab. 18 (1990), pp.~1034--1070.

\bibitem{bib:d04}
R.~D.~DeBlassie,
\emph{Higher order PDE's and symmetric stable processes}.
Probab. Theory Related Fields 129 (2004), pp.~495--536.

\bibitem{bib:dm07}
R.~D.~DeBlassie, P.~J.~M{\'e}ndez-Hern{\'a}ndez,
\emph{$\alpha$-continuity properties of the symmetric $\alpha$-stable process}.
Trans. Amer. Math. Soc. 359 (2007), pp.~2343--2359.

\bibitem{bib:dv78}
P.~Dierolf, J.~Voigt,
\emph{Convolution and $\mathcal{S}'$-convolution of distributions}.
Collectanea Math. 29 (1978), pp.~185--196.

\bibitem{bib:d07}
R.~A.~Doney,
\emph{Fluctuation Theory for L{\'e}vy Processes}.
Lecture Notes in Math. 1897, Springer, Berlin, 2007.

\bibitem{bib:ds10}
R.~A.~Doney, M.~S.~Savov,
\emph{The asymptotic behavior of densities related to the supremum of a stable process}.
Ann. Probab. 38(1) (2010), pp.~316--326.

\bibitem{bib:d65}
E.~B.~Dynkin,
\emph{Markov processes, Vols. I and II}.
Springer-Verlag, Berlin-G{\"o}tingen-Heidelberg, 1965.

\bibitem{bib:els08}
M.~J.~Esteban, M.~Lewin, E. S{\'e}r{\'e},
\emph{Variational methods in relativistic quantum mechanics}.
Bull. Amer. Math. Soc. 45(4) (2008) pp.~545--593.

\bibitem{bib:fg11}
R.~L.~Frank, L.~Geisinger,
\emph{Refined semiclassical asymptotics for fractional powers of the Laplace operator}.
Preprint 2011, arXiv:1105.5181v1.

\bibitem{bib:f74}
B.~E.~Fristedt,
\emph{Sample functions of stochastic processes with stationary, independent increments}.
In: \emph{Advances in Probability and Related Topics}, vol. 3, Dekker, New York, 1974, 241--396.

\bibitem{bib:gko95}
P.~Garbaczewski, J.~R.~Klauder, R.~Olkiewicz,
\emph{Schr{\"o}dinger problem, L{\'e}vy processes, and noise in relativistic quantum mechanics}.
Phys. Rev. E 51 (1995), pp.~4114--4131.

\bibitem{bib:g59}
R.~K.~Getoor,
\emph{Markov operators and their associated semi-groups}.
Pacific J.~Math. 9 (1959), pp.~449--472.

\bibitem{bib:ggk90}
I.~C.~Gohberg, S.~Goldberg, M.~A.~Kaashoek,
\emph{Classes of linear operators. Vol.~1}.
Birkh{\"a}user, 1990.

\bibitem{bib:gj10}
P.~Graczyk, T.~Jakubowski,
\emph{On Wiener-Hopf factors of stable processes}.
Ann. Inst. Henri Poincaré (B) 47(1) (2010), pp.~9--19.

\bibitem{bib:h52}
G.~H.~Hardy, J.~E.~Littlewood, G.~Polya,
\emph{Inequalities}.
Cambridge Univ. Press, London, 1952.

\bibitem{bib:ho58}
Y. Hirata, H. Ogata,
\emph{On the exchange formula for distributions}.
J.~Sci. Hiroshima Univ. Ser. A~22 (1958), pp.~147--152.

\bibitem{bib:hk09}
T.~R.~Hurd, A.~Kuznetsov,
\emph{On the first passage time for Brownian motion subordinated by a Levy process}.
J. Appl. Prob. 46 (2009), pp.~181--198.

\bibitem{bib:i09}
Y.~Isozaki,
\emph{Hitting of a line or a half-line in the plane by two-dimensional symmetric stable L{\'e}vy processes}.
Stoch. Proc. Appl. 121(8) (2011), pp.~1749--1769.

\bibitem{bib:j01}
N.~Jacob,
\emph{Pseudo Differential Operators and Markov Processes}, Vol. 1.
Imperial College Press, London, 2001

\bibitem{bib:kkm11}
K.~Kaleta, M.~Kwa{\'s}nicki, J.~Ma{\l}ecki,
\emph{One-dimensional quasi-relativistic particle in the box}.
In preparation.

\bibitem{bib:k82}
A.~Kami{\'n}ski,
\emph{Convolution, product and Fourier transform of distributions}.
Studia Math. 74 (1982), pp.~83--96.

\bibitem{bib:ksv09}
P.~Kim, R.~Song, Z.~Vondra\v{c}ek,
\emph{Boundary Harnack Principle for Subordinate Brownian Motions}.
Stoch. Proc. Appl. 119(5) (2009), pp.~1601--1631.

\bibitem{bib:ksv10a}
P.~Kim, R.~Song, Z.~Vondra\v{c}ek,
\emph{On the potential theory of one-dimensional subordinate Brownian motions with continuous components}.
Potential Anal. 33(2) (2010), pp.~153--173

\bibitem{bib:ksv10b}
P.~Kim, R.~Song, Z.~Vondra\v{c}ek,
\emph{On harmonic functions for trace processes}.
Math. Nachrichten (2010), to appear.

\bibitem{bib:ksv10}
P.~Kim, R.~Song, Z.~Vondra\v{c}ek,
\emph{Two-sided Green function estimates for killed subordinate Brownian motions}.
Preprint, 2010, arXiv:1007.5455v2.

\bibitem{bib:ksv11}
P.~Kim, R.~Song, Z.~Vondra\v{c}ek,
\emph{Potential theory of subordinate Brownian motions revisited}.
Preprint, 2011, arXiv:1102.1369v2.

\bibitem{bib:kw82}
S.~Kotani, S.~Watanabe,
\emph{Krein's spectral theory of strings and generalized diffusion processes}.
In \emph{Functional analysis in Markov processes (Katata/Kyoto, 1981)}, pp.~235--259, Lecture Notes in Math. 923, Springer, Berlin, 1982.

\bibitem{bib:k62}
M.~G.~Kre{\u{\i}}n,
\emph{Integral equations on a half-line with kernel depending upon the difference of the arguments}.
Transl. Amer. Math. Soc. (2) 22 (1962), pp.~163--288.
(Translated from Uspekhi Mat. Nauk 13(5) (1958), pp.~3--120.)

\bibitem{bib:k97}
T.~Kulczycki,
\emph{Properties of Green function of symmetric stable processes}.
Probab. Math. Statist.~17(2) (1997) 339--364.

\bibitem{bib:k98}
T.~Kulczycki,
\emph{Intrinsic ultracontractivity for symmetric stable processes}.
Bull. Polish Acad. Sci. Math. 46(3) (1998) 325--334.

\bibitem{bib:kkms10}
T.~Kulczycki, M.~Kwa{\'s}nicki, J.~Ma{\l}ecki, A.~St{\'o}s,
\emph{Spectral Properties of the Cauchy Process on Half-line and Interval}.
Proc. London Math. Soc. 30(2) (2010), pp. 353--368.

\bibitem{bib:ks06}
T.~Kulczycki, B.~Siudeja,
\emph{Intrinsic ultracontractivity of Feynman-Kac semigroup for relativistic stable processes}.
Trans. Amer. Math. Soc. 358(11) (2006) 5025--5057.

\bibitem{bib:k10}
A.~Kuznetsov,
\emph{On extrema of stable processes}.
Ann. Probab. 39(3) (2011), pp.~1027--1060.

\bibitem{bib:kkp10}
A.~Kuznetsov, A.~E.~Kyprianou, J.~C.~Pardo,
\emph{Meromorphic L{\'e}vy processes and their fluctuation identities}.
Preprint, 2010, arXiv:1004.4671v2. To appear in Ann. Appl. Probab.

\bibitem{bib:k10:v1}
M.~Kwa{\'s}nicki,
\emph{Spectral analysis of subordinate Brownian motions in half-line. Preliminary version}.
Preprint, 2010, arXiv:1006.0524v1.

\bibitem{bib:k10a}
M.~Kwa{\'s}nicki,
\emph{Eigenvalues of the fractional Laplace operator in the interval}.
Preprint, 2010, arXiv:1012.1133v1.

\bibitem{bib:kmr11}
M.~Kwa{\'s}nicki, J.~Ma{\l}ecki, M.~Ryznar,
\emph{Suprema of L{\'e}vy processes}.
Preprint, 2011, arXiv:1103.0935v1. To appear in Ann. Probab.

\bibitem{bib:kmr11a}
M.~Kwa{\'s}nicki, J.~Ma{\l}ecki, M.~Ryznar,
\emph{First passage times for subordinate Brownian motions}.
Preprint, 2011, arXiv:1110:0401v1.

\bibitem{bib:k06}
A.~E.~Kyprianou,
\emph{Introductory lectures on fluctuations of L{\'e}vy processes with applications}.
Universitext, Springer-Verlag, Berlin, 2006.

\bibitem{bib:ls09}
E.~H.~Lieb, R.~Seiringer,
\emph{The Stability of Matter in Quantum Mechanics}.
Cambridge University Press, 2010.

\bibitem{bib:mo69}
S.~A.~Molchanov, E.~Ostrowski,
\emph{Symmetric stable processes as traces of degenerate diffusion processes}.
Theor. Prob. Appl. 14(1) (1969), pp.~128--131.

\bibitem{bib:py03}
J.~Pitman, M.~Yor,
\emph{Hitting, occupation and inverse local times of one-dimensional diffusions: martingale and excursion approaches}.
Bernoulli 9(1) (2003), pp.~1--24.

\bibitem{bib:psw89}
T.~Poerschke, G.~Stolz, J.~Weidmann,
\emph{Expansions in generalized eigenfunctions of selfadjoint operators}.
Math. Z. 202(3) (1989), pp.~397--408.

\bibitem{bib:r83}
L.~C.~G.~Rogers,
\emph{Wiener-Hopf Factorization of Diffusions and L{\'e}vy Processes}.
Proc. London Math. Soc. 47 (1983), pp.~177--191.

\bibitem{bib:r02}
M.~Ryznar,
\emph{Estimates of Green Function for Relativistic $\alpha$-Stable Process}.
Potential Anal. 17 (2002), pp.~1--23.

\bibitem{bib:s99}
K.~Sato,
\emph{L{\'e}vy processes and infinitely divisible distributions}.
Cambridge Univ. Press, Cambridge, 1999.

\bibitem{bib:si64}
R.~Shiraishi, M.~Itano,
\emph{On the multiplicative products of distributions}.
J.~Sci. Hiroshima Univ. Ser. A~28 (1964), pp.~223--235.

\bibitem{bib:ssv10}
R.~Schilling, R.~Song, Z.~Vondra\v{c}ek,
\emph{Bernstein Functions: Theory and Applications}.
De Gruyter, Studies in Math. 37, Berlin, 2010.

\bibitem{bib:s82}
B.~Simon,
\emph{Schr{\"o}dinger semigroups}.
Bull. Amer. Math. Soc., New Ser. 7(3) (1982), pp.~447--526.

\bibitem{bib:sv06}
R.~Song, Z.~Vondracek,
\emph{Potential theory of special subordinators and subordinate killed stable processes}.
J.~Theoret. Probab. 19(4) (2006), pp.~817--847.

\bibitem{bib:sv08}
R.~Song, Z.~Vondra\v{c}ek,
\emph{On the relationship between subordinate killed and killed subordinate process}.
Electronic Communications in Probability 13 (2008), pp.~325--336.

\bibitem{bib:s58}
F.~Spitzer,
\emph{Some theorems concerning 2-dimensional Brownian motion}.
Trans. Amer. Math. Soc. 87 (1958), pp.~187--197.

\bibitem{bib:v02}
V.~S.~Vladimirov,
\emph{Methods of the theory of generalized functions}.
Taylor and Francis, New York, 2002.

\bibitem{bib:wh31}
N.~Wiener, E.~Hopf,
\emph{{\"U}ber eine Klasse singul{\"a}rer Integralgleichungen}.
Sitzungber. Akad. Wiss. Berlin (1931), pp.~696--706.
(Reprinted in \emph{Norbert Wiener: Collected Works}, Vol.~3, MIT Press, Cambridge, Mass., 1981.)

\bibitem{bib:zrk07}
A.~Zoia, A.~Rosso, M.~Kardar,
\emph{Fractional Laplacian in bounded domains}.
Phys. Rev. E 76, 061121 (2007).

\end{thebibliography}
\end{document}